\documentclass[reqno]{amsart}

\usepackage[
left=3cm,
right=3cm,
top=3cm, 
bottom=3cm,
a4paper]{geometry}
\usepackage{setspace}
\usepackage{stmaryrd} 

\let\tempone\itemize
\let\temptwo\enditemize
\renewenvironment{itemize}{\tempone \vspace{5pt}\addtolength{\itemsep}{0.5\baselineskip}}{\vspace{5pt} \temptwo}

\let\tempenum\enumerate
\let\tempenumtwo\endenumerate
\renewenvironment{enumerate}{\tempenum \vspace{5pt} \addtolength{\itemsep}{0.5\baselineskip}}{ \vspace{5pt} \tempenumtwo}

\makeatletter
\let\origsection\section
\renewcommand\section{\@ifstar{\starsection}{\nostarsection}}
\newcommand\nostarsection[1]
{\sectionprelude\origsection{#1}\sectionpostlude}
\newcommand\starsection[1]
{\sectionprelude\origsection*{#1}\sectionpostlude}
\newcommand\sectionprelude{%
	\vspace{1em} 
}
\newcommand\sectionpostlude{%
	\vspace{1em}   
}
\makeatother

\newcommand\Item[1][]{%
	\ifx\relax#1\relax  \item \else \item[#1] \fi
	\abovedisplayskip=0pt\abovedisplayshortskip=0pt~\vspace*{-\baselineskip}}

\makeatletter
\let\origsubsection\subsection
\renewcommand\subsection{\@ifstar{\starsubsection}{\nostarsubsection}}
\newcommand\nostarsubsection[1]
{\sectionprelude\origsubsection{#1}\subsectionpostlude}
\newcommand\starsubsection[1]
{\subsectionprelude\origsubsection*{#1}\subsectionpostlude}
\newcommand\subsectionprelude{%
	\vspace{0.25em} 
}
\newcommand\subsectionpostlude{%
	\vspace{0.0em}   
}
\makeatother

\usepackage[bb=boondox]{mathalfa}
\usepackage[T1]{fontenc}
\usepackage{lmodern}
\usepackage[utf8]{inputenc}
\usepackage{amsmath,amssymb,amsthm,graphicx,amscd,amsfonts,latexsym}
\usepackage{bm}
\usepackage{mathrsfs}
\usepackage{bbm}
\usepackage{faktor} 
\usepackage{mathtools}
\usepackage{esvect}
\usepackage[UKenglish]{babel}
\usepackage{mathrsfs}
\usepackage{tikz-cd}
\usepackage{tikz}
\usetikzlibrary{decorations.markings}
\usetikzlibrary{shapes.geometric, patterns, angles, quotes}
\usetikzlibrary{hobby}

\usepackage{pstricks}
\usepackage[normalem]{ulem}
\usepackage{lscape}
\usepackage[hidelinks]{hyperref}
\usepackage[capitalize, noabbrev]{cleveref}
\usepackage{tikz}
\usetikzlibrary{arrows}
\usetikzlibrary{decorations}
\usetikzlibrary{shapes}
\usetikzlibrary{decorations.pathreplacing,decorations.markings} 
\usetikzlibrary{calc,intersections,through,backgrounds}
\usepackage{graphicx}
\usepackage{caption}
\usepackage{float}
\usepackage{upgreek}

\usetikzlibrary{shapes.geometric,positioning}
\usetikzlibrary{arrows,decorations.pathmorphing,decorations.pathreplacing}

\usepackage{ifthen} 

\makeatletter
\newcases{rightalignedcases}{\quad}{%
	\hfil$\m@th\displaystyle{##}$\hfil}{$\m@th\displaystyle{##}$\hfil}{\lbrace}{.}
\makeatother

\newcounter{pos} 
\tikzset{									
	initcounter/.code={\setcounter{pos}{0}},
	style between/.style n args={3}{
		postaction={
			initcounter,
			decorate,
			decoration={
				show path construction,
				curveto code={
					\addtocounter{pos}{1}
					\pgfmathtruncatemacro{\min}{#1 - 1}
					\ifthenelse{\thepos < #2 \AND \thepos > \min}{
						\draw[#3]
						(\tikzinputsegmentfirst)
						..
						controls (\tikzinputsegmentsupporta) and (\tikzinputsegmentsupportb)
						..
						(\tikzinputsegmentlast);
					}{}
				}
			}
		},
	},
}

\tikzset{
	clip even odd rule/.code={\pgfseteorule}, 
	invclip/.style={
		clip,insert path=
		[clip even odd rule]{
			[reset cm](-\maxdimen,-\maxdimen)rectangle(\maxdimen,\maxdimen)
		}
	}
}

\usepackage{epstopdf}

\makeatletter

\newcommand{\colim}{\operatorname{colim}}

\DeclareMathOperator{\Hom}{Hom}

\DeclareMathOperator{\End}{End}

\DeclareMathOperator{\Aut}{Aut}
\DeclareMathOperator{\Perf}{Perf}

\newcommand{\Dfd}[1]{\mathsf{D}_{\operatorname{fd}}(#1)}

\DeclareMathOperator{\DPic}{\mathcal{D}Pic}
\DeclareMathOperator{\Pic}{Pic}
\DeclareMathOperator{\Out}{Out}
\DeclareMathOperator{\OutO}{Out_{\circ}}
\DeclareMathOperator{\RHom}{RHom}

\newcommand{\invex}{{\scriptstyle \text{\rm !`}}}

\newcommand{\Ob}[1]{\operatorname{Ob}(#1)}

\newcommand{\Inter}[1]{\mathring{#1}}
\DeclareMathOperator{\Flow}{\mathsf{Flow}}

\newcommand{\Int}{\operatorname{Int}}

\DeclareMathOperator{\ad}{ad}
\DeclareMathOperator{\GL}{GL}
\DeclareMathOperator{\PGL}{PGL}
\DeclareMathOperator{\rad}{rad}

\newcommand{\MC}{\mathsf{MC}}

\tikzset{
	set arrow inside/.code={\pgfqkeys{/tikz/arrow inside}{#1}},
	set arrow inside={end/.initial=>, opt/.initial=},
	/pgf/decoration/Mark/.style={
		mark/.expanded=at position #1 with
		{
			\noexpand\arrow[\pgfkeysvalueof{/tikz/arrow inside/opt}]{\pgfkeysvalueof{/tikz/arrow inside/end}}
		}
	},
	arrow inside/.style 2 args={
		set arrow inside={#1},
		postaction={
			decorate,decoration={
				markings,Mark/.list={#2}
			}
		}
	},
}

\makeatletter
\tikzset{commutative diagrams/.cd,arrow style=tikz,diagrams={>=latex'}}\tikzset{join/.code=\tikzset{after node path={%
			\ifx\tikzchainprevious\pgfutil@empty\else(\tikzchainprevious)%
			edge[every join]#1(\tikzchaincurrent)\fi}}}
\makeatother

\tikzset{>=stealth',every on chain/.append style={join},
	every join/.style={->}}

\tikzset{every loop/.style={min distance=25mm,in=50,out=100,looseness=5}}

\tikzcdset{
	cells={font=\everymath\expandafter{\the\everymath\displaystyle}},
}

\newtheorem{prf}{Proof}[section]

\theoremstyle{remark}

\newtheoremstyle{ownTheoremStyle}
{.75em}
{.75em}
{\itshape}
{}
{\bfseries}
{.}
{ }
{}

\newtheoremstyle{ownDefinitionStyle}
{.75em}
{.75em}
{}
{}
{\bfseries}
{.}
{ }
{}

\theoremstyle{ownTheoremStyle}

\newtheorem{thm}[prf]{Theorem}

\newtheorem{Introthm}{Theorem}

\newtheorem{lem}[prf]{Lemma}
\newtheorem{prp}[prf]{Proposition}

\newtheorem{cor}[prf]{Corollary}
\newtheorem{Introcor}[Introthm]{Corollary}

\theoremstyle{ownDefinitionStyle}
\newtheorem{exa}[prf]{Example}

\newtheorem{definition}[prf]{Definition}
\newtheorem{rem}[prf]{Remark}

\newtheorem{notation}[prf]{Notation}

\newcommand{\quotient}[2]{{\left.\raisebox{.2em}{$#1$}\middle/\raisebox{-.2em}{$#2$}\right.}}

\numberwithin{equation}{section}

\newcommand{\cA}{\mathcal{A}}
\newcommand{\cB}{\mathcal{B}}
\newcommand{\cC}{\mathcal{C}}
\newcommand{\cD}{\mathcal{D}}
\newcommand{\cE}{\mathcal{E}}
\newcommand{\cF}{\mathcal{F}}
\newcommand{\cG}{\mathcal{G}}
\newcommand{\cH}{\mathcal{H}}
\newcommand{\cK}{\mathcal{K}}
\newcommand{\cL}{\mathcal{L}}
\newcommand{\cM}{\mathcal{M}}
\newcommand{\cN}{\mathcal{N}}
\newcommand{\cI}{\mathcal{I}}
\newcommand{\cO}{\mathcal{O}}
\newcommand{\cP}{\mathcal{P}}
\newcommand{\cR}{\mathcal{R}}
\newcommand{\cQ}{\mathcal{Q}}
\newcommand{\cS}{\mathcal{S}}
\newcommand{\cT}{\mathcal{T}}
\newcommand{\cU}{\mathcal{U}}
\newcommand{\cV}{\mathcal{V}}
\newcommand{\cW}{\mathcal{W}}

\newcommand{\cY}{\mathcal{Y}}

\newcommand{\bA}{\mathbb{A}}
\newcommand{\bB}{\mathbb{B}}

\newcommand{\op}{\operatorname{op}}

\DeclareMathOperator{\MCG}{\operatorname{MCG}}

\newcommand{\HHH}{\operatorname{HH}}
\newcommand{\HH}{\operatorname{H}}
\newcommand{\reg}{\operatorname{reg}}
\newcommand{\Tw}{\operatorname{Tw}}
\newcommand{\gr}{\operatorname{gr}}

\newcommand{\para}[1]{{#1}_{\odot}}

\newcommand{\Fuk}{\operatorname{Fuk}}

\DeclareMathOperator{\Hqe}{\mathsf{Hqe}}
\DeclareMathOperator{\Hmo}{\mathsf{Hmo}}

\DeclareMathOperator{\dgcat}{\mathsf{dgcat}}
\DeclareMathOperator{\Acat}{\mathsf{A_{\infty}-cat}}
\DeclareMathOperator{\Acatc}{\mathsf{A_{\infty}-cat^c}}
\DeclareMathOperator{\HAcat}{\mathsf{Ho(A_{\infty}-cat})}
\newcommand{\Arc}[1]{\mathsf{Arc}_{#1}}

\DeclareMathOperator{\Loc}{\mathsf{Loc}}

\DeclareMathOperator{\Fun}{\mathsf{Fun}}

\newcommand{\BCH}[2]{\operatorname{BCH}(#1, #2)}

\newcommand{\fin}{\mathrm{fin}}

\title[Autoequivalences of Fukaya categories of surfaces and gentle algebras]{Autoequivalences of Fukaya categories of surfaces and graded gentle algebras}
\author{Sebastian Opper}

\begin{document}

	\begin{abstract}
		We compute the derived Picard groups of partially wrapped Fukaya categories of surfaces in the sense of Haiden-Katzarkov-Kontsevich and the related graded gentle algebras. This includes the wrapped cases as introduced by Bocklandt.  In combination with results by Burban-Drozd and Lekili-Polishchuk, this leads to a description of derived Picard groups of commutative and non-commutative nodal curves. An important ingredient for our proof in characteristic zero is the exponential map from Hochschild cohomology to the derived Picard group introduced in recent work by the author. In positive characteristics, we combine deformation theory and formality results for Hochschild complexes to prove our results. Along the way we show that the surface together with its decorations forms a complete derived invariant of partially wrapped Fukaya categories and we prove analogous results for graded gentle algebras. This removes all previous restrictions from earlier results of this kind.
	\end{abstract}
	\maketitle

	\setcounter{tocdepth}{1}
	\tableofcontents

\section*{Introduction}

\noindent The study of equivalences between triangulated categories has a long history \cite{HappelBook, RickardMoritaTheory, OrlovK3surfaces, SeidelThomas}. In practice, most triangulated categories in algebraic settings are \textit{enhanced} and therefore arise as a subcategory of the derived category of a dg category. The \textit{derived Picard group} is then the `enhanced' variant of the autoequivalence group of an enhanced triangulated category. Formally,  the derived Picard group of a dg category $\cC$ is the group $\DPic(\cC)$ formed by the $\cC$-bimodules which admit a two-sided inverse under `multiplication' via the derived tensor product. In other words, $\DPic(\cC)$ is the Picard group of the monoidal category of $\cC$-bimodules. Depending on the setting, there are several equivalent incarnations of this group: as the automorphism group of $\cC$ in the Morita category of dg categories \cite{ToenDerivedMoritaTheory}, as equivalence classes of certain $A_\infty$-functors \cite{CanonacoOrnaghiStellari, CanonacoOrnaghiStellariCommutativeRings} or, as the group of kernels of invertible Fourier-Mukai transforms \cite{ToenDerivedMoritaTheory, Orlov}. Many classical symmetry groups, such as the Picard and automorphism group of a variety or the outer automorphism group of an associative algebra, constitute subgroups of the associated derived Picard group. Every element of $\DPic(\cC)$ induces an autoequivalence of both the derived category of $\cC$ and its category of perfect complexes. This yields a comparison map from $\DPic(\cC)$ to the group of autoequivalences of either category which is injective or surjective in common setups \cite{Orlov, GenoveseUniquenessLifting}. One advantage of focusing on $\DPic(\cC)$ instead of autoequivalence groups is the fact that one has access to higher invariants such as Hochschild cohomology which are not defined on the purely triangulated level.\medskip

\noindent In this article we give a description of the derived Picard groups of all partially wrapped Fukaya categories of surfaces in the sense of \cite{HaidenKatzarkovKontsevich}, including the wrapped cases considered earlier in \cite{Bocklandt} and more recently in \cite{BocklandtVanDeKreeke, VanDeKreeke}. To define such categories, one starts with a \textit{graded marked surface}, an oriented surface with boundary which is equipped with additional decorations: a distinguished subset of the boundary and a line field, roughly speaking, an `orientation free' generalisation of a vector field. To a graded marked surface $\Sigma$ the authors of \cite{HaidenKatzarkovKontsevich} assign a certain triangulated $A_\infty$-category, the  \textit{partially wrapped Fukaya category} $\Fuk(\Sigma)$. For punctured surfaces, this agrees with the wrapped Fukaya category considered in \cite{Bocklandt} and its appendix.

  Apart from the punctured cases, the Fukaya categories of \cite{HaidenKatzarkovKontsevich} are equivalent to the categories of compact objects in the derived categories of certain graded algebras. Without gradings, these so-called \textit{gentle algebras} were first studied by Assem and Skowronski \cite{AssemSkowronski} and since then have been featured in the works by many others. Gentle algebras are often studied in two variants:  \textit{proper} (=finite dimensional) ones and those which are \textit{homologically smooth} in the sense of \cite{KellerDifferentialGradedCategories}. Homologically smooth graded gentle algebras are precisely those which appear as generators of the aforementioned Fukaya categories while proper graded gentle algebras  arise as generators of the corresponding compact Fukaya category under certain conditions. The two classes are related by Koszul duality with proper graded gentle algebras appearing as the Koszul duals of homologically smooth graded gentle algebras.

  The passage from Fukaya categories to graded gentle algebras via generators has a counterpart in the other direction. To every proper or homologically smooth graded gentle algebra $A$ one can associate a graded marked surface $\Sigma_A$  \cite{LekiliPolishchukGentle, OpperPlamondonSchroll}. In the homologically smooth case, this yields an equivalence $\Fuk(\Sigma_A)\simeq \Perf(A)$ and in the proper case, provides an equivalence between the category $\Dfd{A}$ of dg $A$-modules with finite-dimensional cohomology and a (derived) completion of $\Fuk(\Sigma_A)$ \cite{OpperPlamondonSchroll, BoothGoodbodyOpper}. If $A$ is proper and $\Sigma_A$ satisfies certain conditions, this completion agrees with $\Fuk(\Sigma_A)$ and  $\Perf{A}$ can be viewed as a model for the infinitesimal Fukaya category of $\Sigma_A$.\medskip
  
  \noindent The following is our first main result and is stated in a simplified form. Throughout the entire paper we will work over a fixed field $\Bbbk$ and to simplify some arguments we will assume that its characteristic is different from $2$.
\begin{Introthm}[{\Cref{cor: Theorem A} \& \Cref{cor: Theorem A (2)}}]\label{IntroThmA} Let $\Sigma$ be a graded marked surface and let $\Fuk(\Sigma)$ denote the associated partially wrapped Fukaya category. If $\operatorname{char} \Bbbk=0$ or $\Fuk(\Sigma)$ is homologically smooth and proper, there exists an isomorphism 
	\begin{equation}\label{eq: intro semi-direct product}
		\DPic(\Fuk(\Sigma)) \cong \cN(\Sigma) \rtimes \MCG_{\gr}(\Sigma),
	\end{equation}
\noindent where $\MCG_{\gr}(\Sigma)$ denotes the graded mapping class group of $\Sigma$ and $\cN(\Sigma)$ is an explicit group which can be read off from $\Sigma$. There is an analogous description for the derived Picard group of a proper graded gentle algebra $A$ as a semi-direct extension of $\MCG_{\gr}(\Sigma_A)$.
\end{Introthm}
\noindent The condition that $\Fuk(\Sigma)$ is smooth and proper is equivalent to simple conditions on $\Sigma$, cf.~\Cref{PropositionPropertiesFukayaCategories}. For a punctured surface $\Sigma$, one has $\cN \cong \HH^1(\Sigma, \Bbbk^{\times})$ (with one exception, cf.~\Cref{TheoremDerivedPicardGroupWrappedCase}) and in the other cases $\cN$ is described in  \Cref{cor: Theorem A} and \Cref{cor: Theorem A (2)}.  The graded mapping class group is a central extension of the mapping class group of $\Sigma$ by the shift functor. Here, elements in the mapping class group are assumed to preserve all decorations of $\Sigma$ and the homotopy class of the line field. The section $\MCG_{\gr}(\Sigma) \rightarrow \DPic(\Fuk(\Sigma))$ which corresponds to the semi-direct product \eqref{eq: intro semi-direct product} arises from an action of $\MCG_{\gr}(\Sigma)$ on the Fukaya category which was first constructed in \cite{DyckerhoffKapranov}. The group $\cN$ is determined by the first Hochschild cohomology of $\Fuk(\Sigma)$ and admits a description in terms of singular cohomology of $\Sigma$, the boundary components of $\Sigma$ as well as their winding numbers with respect to the line field.  The entire Hochschild cohomology of a graded gentle algebra and the Gerstenhaber bracket are computed in \cite{OpperHochschildCohomologyGentle} (see also \cite{BianSchrollSolotarWangWen}) which builds on a number of earlier computations \cite{Ladkani, RedondoRoman, ChaparroSchrollSolotar,  ChaparroSchrollSolotarSuarezAlvarez} in the ungraded case. For punctured surfaces, the description of Hochschild cohomology is found in \cite{BocklandtVanDeKreeke}.\medskip

\noindent \Cref{IntroThmA} directly provides a description of derived Picard groups of categories with algebro-geometric origin. Let $\mathbb{X}$ be a \textit{tame non-commutative nodal projective curve of gentle type}  in the sense of \cite{BurbanDrozdNoncommutativeNodal}. Examples include all Kodaira cycles of projective lines such as the the nodal cubic (cf.~\cite{BurbanDrozdCoherentSheavesSimpleDoublePoints}), as well as stacky chains and stacky cycles of projective lines in the sense of \cite{LekiliPolishchukAuslanderOrders}. The results of \cite{BurbanDrozdNoncommutativeNodal} show that the bounded derived category of $\mathbb{X}$ admits a categorical resolution through its so-called \textit{Auslander curve}. The bounded derived category of such Auslander curves are equivalent to categories of the form $\Fuk(\Sigma)$. 
\begin{Introcor}\label{IntroCorNodalCurves}
Let $\mathbb{X}$ be a non-commutative nodal projective curve of gentle type and let $\mathcal{A}$ denote its Auslander curve. Then the derived Picard group of $\cD^b(\mathcal{A})$ is of the form $\mathcal{N}(\Sigma) \rtimes \MCG_{\gr}(\Sigma)$ for a graded marked surface $\Sigma$.
\end{Introcor}
\noindent It is known in several cases that $\cD^b(\mathbb{X})$ is equivalent to a Verdier quotient of $\cD^b(\mathcal{A})$ which corresponds to the wrapped Fukaya category of the surface obtained from $\Sigma$ by removing all stops. This includes all Kodaira cycles, stacky chains and cycles from above and, expectedly, should include all non-commutative nodal projective curves of gentle type. In particular, we recover the descriptions of autoequivalence groups from \cite{BurbanKreusslerGenusOne, OpperKodairaCycles} as special cases.\medskip

\noindent Our third result is concerned with the question of when two Fukaya categories or derived categories of gentle algebras are equivalent. It is known that a diffeomorphism  $\Sigma \cong \Sigma'$ which is compatible with the decorations on both sides induces an equivalence $\Fuk(\Sigma) \simeq \Fuk(\Sigma')$ \cite{HaidenKatzarkovKontsevich}. Hence, up to Morita equivalence, $\Fuk(\Sigma)$ only depends on the orbit of the graded marked surface $\Sigma$ under the action of the (ungraded) mapping class group $\MCG(\Sigma)$. Therefore, if $A, A'$ are homologically graded gentle algebras and $\Sigma_A \cong \Sigma_{A'}$ we get an equivalence $\cD(A)\simeq \cD(A')$. The converse, that is, whether $\cD(A)\simeq \cD(A')$ implies $\Sigma_A\cong \Sigma_{A'}$ as graded marked surfaces is a subtle question which requires one to extract geometric invariants such as the surface $\Sigma_A$ purely from the categorical structure of the subcategory of compact objects in $\cD(A)$.
  
 \begin{Introthm}[{\Cref{thm: derived invariant Fukaya categories}}]\label{IntroThmB} $\Sigma, \Sigma'$ be graded marked surfaces. Then $\Fuk(\Sigma)$ and $\Fuk(\Sigma')$ are equivalent if and only if  $\Sigma$ and $\Sigma'$ are isomorphic as graded marked surfaces. Likewise, if $A, A'$ are graded gentle algebras which are homologically smooth or proper (but not necessarily both), then $\cD(A) \simeq \cD(A')$ if and only if $\Sigma_A \cong \Sigma_{A'}$.
 \end{Introthm}
 \noindent The work \cite{LekiliPolishchukGentle} provides numerical invariants which allow one to decide whether two graded marked surfaces are isomorphic. Thus, \Cref{IntroThmB} provides an effective derived equivalence classification of Fukaya categories and graded gentle algebras. \Cref{IntroThmB} generalises several earlier results: \cite{AmiotPlamondonSchroll, OpperDerivedEquivalences} prove an analogous result for $A$ proper and concentrated in degree $0$ and \cite{JinSchrollWang} proves it under the assumption that $A$ is graded and homologically smooth. Our results for the proper (but not necessarily homologically smooth) case and for punctured surfaces are therefore new and our proof provides a unified approach to all cases via a different strategy from those in \cite{AmiotPlamondonSchroll} and \cite{JinSchrollWang}.\medskip

\noindent The proofs of \Cref{IntroThmA} and \Cref{IntroThmB} follow the general strategy of \cite{OpperDerivedEquivalences}. First, one constructs a `geometrisation map' $\DPic(\Fuk(\Sigma)) \rightarrow \MCG_{\gr}(\Sigma)$ and, in a second step, its section. Finally, one describes the kernel of the geometrisation map.  The geometrisation map is based on a connection between equivalences between the triangulated categories in \Cref{IntroThmB} and maps between arc complexes of the associated surfaces.  The description of the kernel in the last step was a serious limiting factor for a while since it requires a good control over all elements in $\DPic(\Fuk(\Sigma))$ which are `close' to the identity in a certain sense. Roughly speaking, elements of $\cN$ preserve the isomorphism classes of `most' objects in $\Fuk(\Sigma)$ and map the remaining ones to a deformation. The new tool which allows one to classify such `almost identities' (or certain deformations of the identity functor) is the \textit{exponential map} from recent work of the author \cite{OpperIntegration}. The idea is to treat $\DPic(\Fuk(\Sigma))$ like a Lie group and view the kernel $\cN$ as its neutral component. In analogy to the exponential map between a Lie group and its Lie algebra, one constructs a purely algebraic exponential which `integrates' a portion of the first Hochschild cohomology of an $A_\infty$-category in characteristic zero to elements in the corresponding derived Picard group. The multiplication of the derived Picard group is then encoded through the Baker-Campbell-Hausdorff formula and the Gerstenhaber bracket. When combined with the description of the Hochschild cohomology from \cite{OpperHochschildCohomologyGentle}, this proves the characteristic zero case of \Cref{IntroThmA}.
 
  In positive characteristic, where the exponential is no longer available, we are able to sidestep the occurring difficulties by exploiting the fact that the Hochschild complex of a homologically smooth and proper graded gentle algebra is formal as a dg algebra.  We then show  that for any graded algebra whose Hochschild dg algebra is formal, the respective deformations of the identity functor are governed by the first Hochschild cohomology.
  
  Finally, to prove a variant of \Cref{IntroThmA} for proper graded gentle algebras, we rely heavily on the fact that the category of perfect complexes over such algebras are \textit{reflexive} \cite{BoothGoodbodyOpper} in the sense of \cite{KuznetsovShinder}. In a nutshell, reflexivity implies a strong duality between an enhanced triangulated category $\cT$ and the category  $\Dfd{\cT}$, including an isomorphism of their derived Picard groups. For example, a reflexive category $\cT$ and $\Dfd{\cT}$ have isomorphic Hochschild cohomology and derived Picard groups \cite{GoodbodyReflexivity}.

\subsection*{Structure of the paper} The first five sections are predominantly expository and review the necessary background material about $A_\infty$-categories, Fukaya categories of surfaces, gentle algebras and their surface models as well as other topics. Section \ref{Section Geometrisation Homomorphism} discusses and constructs the geometrisation homomorphisms. It also contains proof of some cases of \Cref{IntroThmB}. \Cref{SectionGroupActionFukayaCategory} discusses actions of the mapping class group and establishes the existence of a section for the geometrisation homomorphism. \Cref{SectionIntegrationHochschild} recalls the exponential map from \cite{OpperIntegration} while \Cref{SectionLieAlgebraStructureGradedGentle} discusses the Hochschild cohomology of graded gentle algebras. \Cref{SectionKernel} provides a description of the group $\cN$ in \Cref{IntroThmA} and \Cref{SectionFormalityHochschildComplex} proves \Cref{IntroThmA} in positive characteristic. Finally, \Cref{SectionDerivedPicardGroupWrappedFukayaCategory} contains the proof of \Cref{IntroThmA} for the Fukaya categories of punctured surfaces.
\subsection*{Acknowledgments} I like to thank Fiorela Rossi Bertone and Julia Maria Redondo for answering my questions about comparison maps between Hochschild complexes and the Bardzell resolutions as well as Raf Bocklandt and Jasper van de Kreeke for answering questions regarding the Hochschild cohomology of Fukaya categories. Finally, I thank  Isaac Bird and Alexandra Zvonareva for their comments on an earlier draft of this article. Throughout this project, I was supported by the Primus grant PRIMUS/23/SCI/006.

	\section{\texorpdfstring{Reminder on $A_\infty$-categories and their Hochschild cohomology}{Reminder on A-categories and their Hochschild cohomology}}
	\noindent We  review some aspects of the theory of  $A_\infty$-categories. We will follow the notation from \cite{OpperIntegration} and, in particular, use cohomological grading conventions.

\subsection{Hochschild spaces and brace operations}\label{SectionBraceAlgebra}

	\begin{definition}\label{DefinitionAInfinityCategory} A  \textbf{$\Bbbk$-quiver} $\cG$ is a set of objects $\Ob{\cG}$ together with a collection of graded vector spaces $\{\cG(X,Y)\}_{X, Y \in \Ob{\cG}}$.  Its \textbf{Hochschild space}  is the graded vector space $C(\cG)$ whose homogeneous component in degree $n$ is
			\begin{equation}\label{EquationHochschildSpace}
				C^n(\cG)  \coloneqq \prod_{p=0}^{\infty} \prod_{X_0
					, \dots, X_1 \in \Ob{\cG}} \Hom_\Bbbk^{n}\Big(\cG(X_{p-1}, X_p)[1] \otimes \cdots \otimes \cG(X_{0}, X_1)[1], \cG(X_0,X_p)[1]\Big),
				\end{equation}
\noindent Here, $\Hom_{\Bbbk}^{i}(-, -)$ denotes the space of  homogeneous graded $\Bbbk$-linear maps of degree $i$ and an empty tensor product is interpreted as $\Bbbk$.
\end{definition}
\noindent For $p=0$, the expression the interior product of \eqref{EquationHochschildSpace} ranges over all $X_0 \in \Ob{\cG}$ with factors $\Hom^n_{\Bbbk}(\Bbbk[1], \cG(X_0,X_0)[1])\cong \cG^n(X_0, X_0)$. We refer to a homogeneous element $f \in C^n(\cG)$ through a tuple $f=(f^i)_{i \geq 0}$, where $f^i$ denotes the product of all $i$-ary components of $f$ (the case $p=i$ in \eqref{EquationHochschildSpace}). 
The Hochschild space $C=C(\cG)$ is filtered by its \textbf{weight filtration}
\begin{displaymath}
	C= W_0C \supseteq W_1C \supseteq \cdots,
\end{displaymath}
\noindent where $W_iC$ is defined by restricting the index $p$ in \eqref{EquationHochschildSpace} to $p \geq i$. It endows each  $C^n(\cG)$ as well as $C(\cG)$ with the structure of a Banach space such that $W_iC(\cG)$ coincides with the closed ball of radius $2^{-i}$ and center $0$ with respect to the metric. The Hochschild space admits continuous linear operations of degree $0$,
\begin{displaymath}\begin{tikzcd}
-\{\cdots\}_r\colon  C^{\otimes (r+1)} \arrow{r} & C
\end{tikzcd}
\end{displaymath}
\noindent where $r \geq 1$, which are called \textbf{braces}. If $\cG$ has a single object and hence is equivalent to a graded vector space $V$, then for homogeneous $f=(f^i)_{i \geq 0}$, $g_1, \dots, g_r$ with $g_j=(g_j^i)_{i \geq 0}$ and each $m \geq 0$, the expression $f\{g_1, \dots, g_r\}$ is defined as
\begin{equation}\label{EquationDefinitionBracesHochschildComplex}
	{(f\{g_1, \dots, g_r\}_r)}^m \coloneqq \sum  f^{m-\sum_{j=0}^r i_j}\Big(\operatorname{Id}_{V[1]}^{\otimes i_1} \otimes \, g_1^{n_1} \otimes \operatorname{Id}_{V[1]}^{\otimes i_2} \otimes \, g_2^{n_2} \otimes \cdots \otimes g_r^{n_r} \otimes \operatorname{Id}_{V[1]}^{\otimes i_r}\Big),	
\end{equation}
\noindent where the sum is indexed by all $(i_0, \dots, i_{r}, n_1, \dots, n_r) \in \mathbb{N}_{\geq 0}^{2r+1}$ so that $\sum_{j=0}^r i_j + \sum_{j=1}^r n_j=m-r$. We employ the Koszul sign rule to evaluate a tensor product of functions $f \otimes g$ on an element $u \otimes v$. The case of arbitrary graphs is analogous. We drop the index $r$ from the notation of the braces whenever it is apparent from the context. The brace operations are compatible with the index shifted weight filtration, that is, $-\{\cdots\}_r$ restricts to maps
\begin{displaymath}
\begin{tikzcd}
W_{1+ m_0}C \otimes \cdots \otimes W_{1+ m_r}C \arrow{r} & W_{1+ \sum_{i=0}^r m_i}C,
\end{tikzcd}
\end{displaymath} 
\noindent for all $m_0, \dots, m_r \geq -1$ and where $W_jC\coloneqq 0$ for $j < 0$. The binary operation $-\{-\}_2$ is referred to as the \textbf{composition product} and will be denoted by $- \star -$. It is not associative but a so-called \textit{pre-Lie algebra} and hence determines a graded Lie bracket on $C(\cG)$ via its graded commutator
\begin{displaymath}
[f, g] \coloneqq f \star g - (-1)^{|f| |g|} g \star f.
\end{displaymath}
\noindent The identity $\mathbf{1}\coloneqq \operatorname{Id}_{\cG} \in C^1(\cG)$ whose components are the identities on all spaces $\cG(X,Y)$, $X, Y \in \Ob{\cG}$, is a left unit of the composition product and satisfies $\mathbf{1}\{g_1, \dots, g_r\}_r=0$ for all $r \geq 2$.  

\subsection{$A_\infty$-categories and $A_\infty$-functors}

\begin{definition}
A (shifted) $A_\infty$-category is a $\Bbbk$-quiver $\bA$ together with an element $\mu_{\bA} \in W_1C^1(\bA)$ such that $\mu_{\bA} \star \mu_{\bA}=0$.
\end{definition}
\noindent Using the shift map $s: V \rightarrow V[1]$ of degree $-1$ of a graded vector space $V$ and the Koszul sign rule, the family $\mu_{\bA}=(\mu_{\bA}^i)_{i \geq 0}$ is equivalent to a family of maps $(m_{\bA}^i)_{i \geq 0}$ of degree $2-i$ between tensor products of the vector spaces $\bA(-,-)$ instead of their suspensions. These satisfy the quadratic relations in \cite{KellerAInfinityAlgebrasInRepresentationTheory, GetzlerJones} for $A_\infty$-categories. Equivalent sign conventions which are commonly used in symplectic geometry can be found in \cite{SeidelBook, HaidenKatzarkovKontsevich}. All $A_\infty$-categories in this paper will be \textbf{strictly unital} which means that every object $A \in \bA$ possesses a strict identity morphism $\operatorname{Id}_A \in \bA^0(A,A)$ subject to the conditions
\begin{displaymath}
\begin{array}{lcr}
m_{\bA}^2(f, \operatorname{Id}_{A})=f=m_{\bA}^2(\operatorname{Id}_{A'}, f) & \text{and} & m_{\bA}^i(f_n, \dots, f_1)=0,
\end{array}
\end{displaymath}
\noindent for all $f \in \bA(A,A')$, $i \geq 3$, and whenever at least one of the $f_i$ is of the form $\operatorname{Id}_{A}$.
Like dg categories, the operation $\mu^1$ are differentials on the morphism spaces of an $A_\infty$-category and every $A_\infty$-category $\bA$ has an associated graded homotopy category $\HH^{\bullet}(\bA)$ whose composition is induced by $\mu_{\bA}^2$. Restricting to morphisms in degree $0$ then yields the  homotopy  category $\HH^0(\bA)$. An $A_\infty$-category $\bA$ is \textbf{minimal} if $\mu_{\bA}^1=0$ vanishes identically. In this case, $A_\infty$-category has an underlying graded $\Bbbk$-linear category with composition $\mu_{\bA}^2$. In general however, unless $\bA$ is minimal, $\mu_{\bA}^2$ fails to be associative. The datum of a dg category $\cU$ is equivalent to an $A_\infty$-category $\tilde{\cU}$ with the same objects such that $\mu_{\tilde{\cU}}^i=0$ for all $i \geq 3$ such that the differentials and composition law of $\cU$ are encoded in the operations $m_{\tilde{\cU}}^1$ and $m_{\tilde{\cU}}^2$.

An \textbf{$A_\infty$-functor} $F\colon  \bA \rightarrow \bB$ between $A_\infty$-categories $\bA$ and $\bB$ consists of a map $F^0\colon  \Ob{\bA} \rightarrow \Ob{\bB}$ and a family $\{F^r\}_{r \geq 1}$ of multi-linear maps, also called \textbf{Taylor coefficients},
\begin{displaymath}
	\begin{tikzcd}
F^d\colon  \bA(A_{d-1}, A_d)[1] \otimes \cdots \otimes \bA(A_0, A_1)[1] \arrow{r} & \bA\big(F^0(A_0), F^0(A_d)\big)[1],
\end{tikzcd}
\end{displaymath}
\noindent of degree $0$ which satisfy a generalisation of the covariance condition for ordinary functors. All $A_\infty$-functors in this paper will be \textbf{unital}, cf.~\cite{CanonacoOrnaghiStellariCommutativeRings} for the definition. Any $A_\infty$-functor induces a functor between the associated (graded) homotopy categories and an $A_\infty$-functor $F$ is \textbf{strict} if $F^d=0$ for all $d \geq 2$.  After suspension, a dg functor $G\colon  \bA \rightarrow \bB$ between dg categories is equivalent to a strict $A_\infty$-functor whose first Taylor coefficient $F^1$ encodes the action of $G$ on morphisms. For any pair $(\bA, \bB)$ of $A_\infty$-categories the set of (unital) $A_\infty$-functors $\bA \rightarrow \bB$  form the objects of an $A_\infty$-category $\Fun(\bA, \bB)$. Functors $F, G \in \Fun(\bA, \bB)$ are \textbf{weakly equivalent}, denoted by $F \approx G$, if $F$ and $G$ are isomorphic in $\HH^0(\Fun(\bA, \bB))$. A finer notion of equivalence, called \textit{homotopy}, implies weak equivalence. Both relations are compatible with composition of $A_\infty$-functors.

An $A_\infty$-functor is \textbf{weakly invertible} if it admits an inverse up to weak equivalence. A special case of these are \textbf{quasi-equivalences}, $A_\infty$-functors which induces an equivalence between the associated graded homotopy categories. For an $A_\infty$-category $\bA$, we denote by $\Aut^{\infty}(\bA)$ the group of weak equivalence classes of weakly invertible $A_\infty$-functors $\bA \rightarrow \bA$. Throughout the paper, we are often interested in the following special case of $A_\infty$-functors. 
\begin{definition}
Let $\bA$ be an $A_\infty$-category. An \textbf{$A_\infty$-isotopy} of $\bA$ is an $A_\infty$-functor $F\colon  \bA \rightarrow \bA$ such that $F^1=\operatorname{Id}_{\bA}$.
\end{definition}
\noindent The set of $A_\infty$-isotopies of an $A_\infty$-category $\bA$ is closed under composition. Moreover, are all $A_\infty$-isotopies are invertible (in the strict sense) and therefore weakly invertible.
\begin{definition}
	We denote by $\Aut^{\infty}_+(\bA) \subseteq \Aut^{\infty}(\bA)$ the group consisting of the weak equivalence classes of $A_\infty$-isotopies of $\bA$.
\end{definition} 
\noindent  We note that $A_\infty$-isotopies of a graded $\Bbbk$-linear category are homotopic if and only if they are weakly equivalent. This is a consequence of a more general criterion, cf.~\cite[Propsition 6.15]{OpperIntegration} and the discussion thereafter. The $A_\infty$-functor equations as well as compositions of $A_\infty$-isotopies can be expressed through the brace operations. 

\begin{lem}[{cf.~\cite[Remark 2.12]{SeidelFormalGroups}, \cite[Section 3.6]{OpperIntegration}}]\label{LemmaFormulaCompositionIsotopies}
Let  $F$ and $G$ be $A_\infty$-isotopies of an $A_\infty$-category $\bA$. Write $F_+\coloneqq F-\operatorname{Id}_{\bA} \in W_2C^1(\bA)$ and $G_+\coloneqq G-\operatorname{Id}_{\bA} \in W_2C^1(\bA)$. Then $G \circ F$ is the unique $A_\infty$-isotopy such that
\begin{displaymath}
{(G \circ F)}_+ = G_+ \odot F_+ \coloneqq \sum_{r \geq 1}G_+\{F_+, \dots, F_+\}_r,
\end{displaymath}
\noindent where $G_+\{\cdots\}_0 \coloneqq G_+$.
 The $A_\infty$-functor equation of $F$ is equivalent to
\begin{displaymath}
0 = [\mu_{\bA}, F_+] + \sum_{r \geq 2}\mu_{\bA}\{F_+, \dots, F_+\}_r.
\end{displaymath}
\end{lem}

\subsection{Twisted complexes and (perfect) derived categories}\ \medskip

\noindent Every strictly unital $A_\infty$-category $\bA$ has a derived category $\cD(\bA)$ which is constructed in the same way as the derived category of a dg category. The category of $\bA$-modules forms a dg category and there is a derived Yoneda embedding $\cY\colon \HH^0(\bA) \rightarrow \cD(\bA)$. An $A_\infty$-category $\bA$ is \textbf{pretriangulated} (resp.~\textbf{perfect}) if the essentially image of its Yoneda embedding is a triangulated (resp.~thick) subcategory of $\cD(\bA)$. For a perfect $A_\infty$-category the essential image of $\cY$ is equivalent to the category of compact objects of $\cD(\cA)$ and will be denoted by $\Perf(\bA)$. We define $\Dfd{\bA} \subseteq \cD(\bA)$ the full subcategory (resp.~dg subcategory) spanned by the $\bA$-modules $M$ such that $M(A) \in \Perf(\Bbbk)$ for all $A \in \bA$. By abuse of notation, we also denote by $\Perf(\bA)$ and $\Dfd{\bA}$ the canonical $A_\infty$-categories consisting of all $\bA$-modules whose image in $\cD(\bA)$ lie in $\Perf(\bA)$ and $\Dfd{\bA}$ respectively.

The triangulated hull of $\bA$ inside $\cD(\bA)$ admits a particularly simple $A_\infty$-enhancement through the $A_\infty$-category of one-sided twisted complexes $\Tw^+(\bA)$ over $\bA$, cf.~\cite{HaidenKatzarkovKontsevich}. Objects in $\Tw^+(\bA)$ a certain pairs $(X, \delta)$, where $X$ is a finite formal sum of formal shifts of objects in $\bA$ and $\delta\colon X \rightarrow X$ is a differential. In other words, $\Tw^+(\bA)$ is the equivalent of the category of finite chain complexes of free modules over an algebra. Instead of recalling the general definition here, we only mention that $\bA \mapsto \Tw^+(\bA)$ defines a monad whose multiplication map $\Tw^+\big(\!\Tw^+(\bA)\big) \rightarrow \Tw^+(\bA)$ is a quasi-equivalence of $A_\infty$-categories and  is given by a direct sum totalisation.

\subsection{The Hochschild $A_\infty$-algebra of an $A_\infty$-category}\label{SectionStructureHochschildComplex}\ \medskip

\noindent The Hochschild space $C=C(\bA)$ of an $A_\infty$-category $\bA$ naturally carries the structure of an $A_\infty$-algebra whose shifted $A_\infty$-operations are given by
\begin{equation}\label{EquationAInfinityStructureHochscchildComplex}
	\begin{aligned}
		\mu_C^1(f) & \coloneqq \big[\mu_{\bA}, f\big] =\mu_{\bA} \star f - (-1)^{|f|} f \star \mu_{\bA},\\
		\mu_C^i(f_1, \dots, f_i) & \coloneqq \mu_{\bA}\{f_1, \dots, f_i\}, \, i \geq 2,
	\end{aligned}
\end{equation}
\noindent where $f, f_1, \dots, f_i \in C$ are homogeneous.

\begin{rem}
An equivalent formulation of Lemma \ref{LemmaFormulaCompositionIsotopies} is that $F$ is an $A_\infty$-isotopy if and only if $F_+ \in W_2C^1(\bA)$ is a \textbf{Maurer-Cartan element} of the $A_\infty$-algebra $C=C(\bA)$, that is, $\zeta=F_+$ satisfies
\begin{equation}\label{EquationMaurerCartanEquation}
\sum_{i=1}^{\infty} \mu_{C}^i(\zeta, \dots, \zeta)=0.
\end{equation}
\noindent Here, the left hand side is considered as the limit of its finite partial sums, which exists due to the fact that $\mu_{C}^i(\zeta, \dots, \zeta) \in W_iC$. 
\end{rem}
\noindent If $\bA$ is a dg category, then $C(\bA)$ is a shifted dg algebra, which we will refer to as the \textit{Hochschild dg algebra}. The operation $m_{C}^2$ is also known as the \textbf{cup product}, usually denoted by $\cup$. The cohomology group
\begin{displaymath}
	\HHH^i(\bA, \bA) \coloneqq \HH^{i-1}\big(C(\bA)),
\end{displaymath}
\noindent is called the $i$-th \textbf{Hochschild cohomology} of $\bA$. The total Hochschild cohomology $\HHH^{\bullet}(\bA, \bA)$ inherits the structure of a graded algebra while the Lie bracket determined by the composition product endows $\HHH^{\bullet+1}(\bA, \bA)$ with a compatible graded Lie algebra structure. Its Lie bracket is the \textbf{Gerstenhaber bracket} and was first described by Gerstenhaber for associative algebras. The weight filtration on $C(\bA)$ is compatible with the $A_\infty$-operations and descends to a filtration
\begin{displaymath}
	\HHH^{\bullet}(\bA, \bA) = W_0 \HHH^{\bullet}(\bA, \bA) \supseteq W_1 \HHH^{\bullet}(\bA, \bA) \supseteq \cdots.
\end{displaymath}
\noindent which we call the (cohomological) \textbf{weight filtration}. 
	
	\section{Topological Fukaya categories after Haiden-Katzarkov-Kontsevich}\label{SectionFukayaCategories}
	\noindent We briefly recall the construction of the topological Fukaya category of a marked surface after Haiden-Katzarkov-Kontsevich. All manifolds, maps and homotopies between them will be smooth unless specified otherwise. 
	\subsection{Marked surfaces}\ \medskip
	
	\noindent A \textbf{marked surface} is an oriented, compact surface $\Sigma$ with non-empty boundary $\partial \Sigma$ and a compact subset $\cM \subset \partial\Sigma$ which has non-trivial intersection with each boundary component. A boundary component $B$ is \textbf{fully marked} if $B \subseteq \cM$ and we say that $\Sigma$ is \textbf{punctured} if all its boundary components are fully marked. Equivalently, fully marked boundary components are in bijection with connected components of $\cM$ which are homeomorphic to $S^1$. The remaining connected components of $\cM$, namely the contractible ones, are called \textbf{marked intervals}. A \textbf{boundary segment} is the closure of a connected component of $\partial \Sigma \setminus \cM$. A \textbf{morphism of marked surfaces} $\varphi\colon  (\Sigma, \cM) \rightarrow (\Sigma', \cM')$ is an orientation preserving immersion mapping $\cM$ to a subset of $\cM'$ and boundary segments to distinct non-isotopic paths. Isotopies of morphisms are isotopies of the underlying immersions which fixes $\cM'$ setwise. 
	
	By a \textbf{line field} on $\Sigma$, we mean a section $\eta\colon  \Sigma \rightarrow \mathbb{P}(T\Sigma)$ of the projectivisation $\mathbb{P}(T\Sigma)$ of the tangent bundle of $\Sigma$. For example, every non-vanishing vector field $\chi$ induces a line field by taking the $1$-dimensional vector space spanned by $\chi(x)$ at every point $x \in \Sigma$.  A triple $(\Sigma, \cM, \eta)$ is called a \textbf{graded marked surface} but frequently we drop $\cM$ and $\eta$ from the notation. A morphism $(\Sigma, \cM, \eta) \rightarrow (\Sigma', \cM', \eta')$ is a pair $(\varphi, \widetilde{\varphi})$, consisting of a morphism  $\varphi$ between the underlying marked surfaces and the homotopy class $\widetilde{\phi}$ of a path from $\varphi^{\ast}\eta' \coloneqq \varphi^{\ast}\eta'\coloneqq {\mathbb{P}(T\varphi)}^{-1}\circ \eta' \circ \varphi$  to $\eta$ in the space of line fields, that is, the space of sections of $\mathbb{P}(T\Sigma)$. Here, $\phi^{\ast}\eta'$ is the unique line field determined by the commutativity of the diagram
	\begin{displaymath}
	\begin{tikzcd}
	\Sigma \arrow{d}[swap]{\varphi} \arrow{r}{\varphi^{\ast}\eta'} & \mathbb{P}(T\Sigma) \arrow{d}{\mathbb{P}(T\varphi)} \\
	\Sigma' \arrow{r}{\eta'} & \mathbb{P}(T\Sigma'). 
\end{tikzcd}
	\end{displaymath}
	Compositions of morphisms are defined via concatenation of maps and paths and notions such as isotopy extend in the natural way. 
	
	A \textbf{curve} on $\Sigma$ is an immersion $\gamma\colon \Omega \rightarrow \Sigma$, where $\Omega$ is a compact connected $1$-dimensional manifold\footnote{In other words, $\Omega$ is diffeomorphic to $[0,1]$ of the $1$-sphere $S^1$.} such that $\gamma^{-1}(\cM)=\partial \Omega$. When want to be more precise, we say that $\gamma$ is a \textbf{loop} if $\Omega \cong S^1$ and an \textbf{arc} otherwise. A \textbf{boundary arc} is an arc which is homotopic relative end points to a map with image in $\partial \Sigma$.  An \textbf{infinite} (resp.~\textit{semi-infinite}) arc is one which has at least one (resp.~exactly one) end point in a fully marked component. The remaining arcs are called \textbf{finite}.  A \textbf{homotopy} between curves is a homotopy $H\colon  \Omega \times [0,1] \rightarrow \Sigma$ relative\footnote{This means that the image of points in $\partial \Omega$ are allowed to move freely inside $\cM$.} to $(\partial \Omega, \cM)$ so that that $H_{t}\coloneqq H|_{\Omega \times \{t\}}\colon  \Omega \rightarrow \Sigma$ is a curve for all $t \in [0,1]$. A curve $\gamma$ is \textbf{simple} if its \textbf{interior}, that is, the restriction $\gamma|_{\Inter{\Omega}}$ to $\Inter{\Omega}\coloneqq \Omega \setminus \partial \Omega$, is an embedding. An \textbf{arc system} is a collection of pairwise non-homotopic and pairwise disjoint simple arcs on $\Sigma$.

\subsection{Gradings and winding numbers}\ \medskip

\noindent Every line field $\eta$ allows one to associate a \textbf{winding number} $\omega_{\eta}(\gamma) \in \mathbb{Z}$ to any loop $\gamma\colon  S^1 \rightarrow \Sigma$ given by the (signed) intersection number of the submanifold $\eta(\Sigma) \subseteq \mathbb{P}(T\Sigma)$ with the image of the derivative $\dot{\gamma}\colon  S^1 \rightarrow \mathbb{P}(T\Sigma)$ of $\gamma$. Discarding $\eta$ from the notation from now on, the integer $\omega(\gamma)$ only depends on the regular homotopy class of $\gamma$ (which preserve the homotopy class of $\dot{\gamma}$). Moreover, if $\gamma$ bounds no ``tear drops'', cf.~\cite[Figure 3]{HaidenKatzarkovKontsevich}, then $\omega(\gamma)$ in fact only depends on the homotopy class of $\gamma$. Every closed curve $\gamma$ is homotopic to one without tear drops and for any two such curves homotopy and regular homotopy coincide.  A line field $\eta$ determines a $\mathbb{Z}$-fold covering $\pi_{\eta}\colon \widetilde{\mathbb{P}}_{\eta} \rightarrow \mathbb{P}(T\Sigma)$. A \textbf{grading} on a curve $\gamma\colon  \Omega \rightarrow \Sigma$ is a lift $\widetilde{\gamma}$ of the derivative $\dot{\gamma}$ along $\pi_{\eta}$ and the pair $(\gamma, \widetilde{\gamma})$ a \textbf{graded curve}. As before, we will drop $\widetilde{\gamma}$ from the notation in most instances. As a consequence of the definition, every arc admits $\mathbb{Z}$ possible gradings, whereas a loop admits a grading if and only if $\omega(\gamma)=0$.

	\subsection{Marked intervals versus marked points}\ \medskip

\noindent It is often convenient to contract connected components of $\cM$ to a point so that marked interval are identified with a marked boundary point and fully marked boundary components are identified with punctures. This creates an intersection between arcs with end points in the same marked interval. Homotopy classes of curves on the original surface are in bijection with homotopy classes of curves on the new surface. Here an arc is a path with end points at marked points (which includes punctures) whose interior is required to be disjoint from the set of marked points and homotopies of arcs are understood to be relative end points and such that all intermediate paths are arcs. Throughout the paper, we will switch freely between those two points of view without further notice and the chosen perspective should be apparent from the context.

\subsection{Topological Fukaya categories} \ \medskip

	\begin{definition}
	Let $\gamma, \gamma'$ be arcs. A \textbf{flow} from $\gamma$ to $\gamma'$ is an orientation preserving immersion $f\colon [0,1] \rightarrow \cM$ such that $f(0)$ is an end point of $\gamma$ and $f(1)$ is an end point of $\gamma'$.
\end{definition}
\noindent 
Flows $f$ and $g$ are composed in the same way as paths if $f(1)=g(0)$ and are considered up to homotopy. The set of homotopy classes of flows from $\gamma$ to $\gamma'$ is denoted by $\Flow(\gamma, \gamma')$. As a result if $\gamma, \gamma'$ have end points $p \in \gamma$, $q \in \gamma'$ in the same fully marked component $B$, there is an embedding $ \mathbb{N}_{\geq 0} \subseteq \Flow(\gamma, \gamma')$ corresponding to all the clockwise traversing paths in $B$ starting on $p$ and ending on $q$. If $\gamma$ and $\gamma'$ are graded, then, as explained in \cite[Section 3.2 (3.15)]{HaidenKatzarkovKontsevich}, each flow $f$ from $\gamma$ to $\gamma'$ inherits a well-defined degree $\deg(f) \in \mathbb{Z}$ which  endows the vector space  $\Bbbk\Flow(\gamma, \gamma')$ with a grading.
\begin{definition}
	Let $\Gamma$ be a graded arc system. Then let $\cC=\cC_{\Gamma}$ denote the category with objects $\Gamma$ and graded morphism spaces
	\begin{displaymath}
		\Hom_{\cC}(\gamma, \gamma') =\Bbbk \oplus \Bbbk \Flow(\gamma, \gamma').
	\end{displaymath}
	\noindent  The composition is given by composition of flows with the copies of $\Bbbk$ acting as identity morphisms. We denote by $\Flow_{\operatorname{irr}}(\gamma, \gamma') \subseteq \Flow(\gamma, \gamma')$ the set of  \textbf{irreducible flows}, that is, flows which are not themselves non-trivial compositions of flows between arcs of $\Gamma$. 
\end{definition}
\noindent The category algebra $A_{\Gamma}$ of $\cC_{\Gamma}$ is isomorphic to the algebra of a graded quiver with relations $(Q_{\Gamma}, R_{\Gamma})$ with  $Q_{\Gamma}^0=\Gamma$ with arrows $\alpha_f\colon\gamma \rightarrow \gamma'$ corresponding bijectively to irreducible flows $f \in \Flow_{\operatorname{irr}}(\gamma, \gamma')$. The set of relations $R_{\Gamma}$ consists of the paths $\alpha_g \alpha_f$ such that $f(1) \neq g(0)$. We call any quiver of the form $(Q_{\Gamma}, R_{\Gamma})$ a \textbf{gentle quiver}. This stems from their connection to gentle algebras which we discuss in the next section.

Let $\Gamma$ be an arc system on $(\Sigma, \cM)$ and consider a marked surface $(\mathbb{D},\cN)$ such that $\mathbb{D}$ is diffeomorphic to a closed disk with $\cN$ consisting of $m \geq 3$ components and let $f_1, \dots, f_m$ denote the distinct components of $\cN$ in clockwise order considered as flows between the boundary segments. For any map  $\varphi\colon  (\mathbb{D},\cN) \rightarrow (\Sigma, \cM)$ sending all boundary segments of $(\mathbb{D}, \cN)$ to arcs of $\Gamma$, the sequence $\varphi \circ f_1, \dots, \varphi \circ f_m$ is a sequence of flows and any such cyclic sequence of flows on $\Sigma$ obtained in this way is called a \textbf{disk sequence}. This is used in the definition of Fukaya categories.

\begin{definition}[{\cite{HaidenKatzarkovKontsevich}}]\label{DefinitionPreFukayaCategory}
Let $\Gamma$ be a graded arc system on $\Sigma$. Let $\cF=\cF(\Gamma)$ denote the minimal $A_\infty$-category whose underlying graded category is $\cC_{\Gamma}$ and whose higher multiplications are given by the following rule:
\begin{displaymath}
	\begin{aligned}
	\mu_{\cF}(d_1, \dots, d_mf) & = (-1)^{|f|} f \\
	\mu_{\cF}(gd_1, \dots, d_m) & = g,
	\end{aligned}
	\end{displaymath}
	
	\noindent for  every disk sequence $d_1, \dots, d_m$ and flows $f, g$ such that $g \circ d_1 \neq 0 \neq d_m \circ  f$ in $\cC_{\Gamma}$.
	\end{definition}

\begin{definition}
An arc system $\Gamma \subseteq (\Sigma, \cM)$ is \textbf{full} if $\Sigma \setminus \Gamma$ is homeomorphic to a disjoint union of topological disks and it is \textbf{formal} if it does not admit any disk sequences, that is, if $\cF(\Gamma)$ has no higher $A_\infty$-compositions. It is \textbf{proper} if all its arcs are disjoint from the fully marked components and a proper arc system is \textbf{finitely full} if $\Sigma \setminus \Gamma$ is a disjoint union of topological disks and topological cylinders $C$ such that $\cM \supseteq \partial C \cap \partial \Sigma \cong S^1$.
\end{definition}
\noindent  In other words, each cylinder $C$ shares exactly one of its boundary components with the boundary of $\Sigma$ and this component is fully marked. Of course, on a marked surface without fully marked components, full coincides with finitely full but not in general. We collect a few facts on the category $\cF(\Gamma)$ which are either proved in \cite{HaidenKatzarkovKontsevich} or which follows immediately from Definition \ref{DefinitionPreFukayaCategory} or general facts.

\begin{prp}[{\cite[Lemma 3.3, Proposition 3.3., Proposition 3.5]{HaidenKatzarkovKontsevich}}]\label{PropositionPropertiesFukayaCategories}Let $(\Sigma, \cM)$ be a marked surface and let $\Gamma \subseteq \Sigma$ be an arc system. The following are true.
	\begin{enumerate}
		\item $(\Sigma, \cM)$ admits a full arc system and if $\cM \subsetneq \partial \Sigma$, then it admits a full formal arc system.
		\item If $\cM \subsetneq \partial \Sigma$, then a full arc system is formal if and only if it is minimal with respect to inclusion within the set of full arc systems.
		\item If $\cM \subsetneq \partial \Sigma$, then  $\cF(\Gamma)$ is homologically smooth and if $\Gamma$ is formal, then so is $\cF(\Gamma)$.
		\item The $A_\infty$-category $\cF(\Gamma)$ is proper if and only if no arc in $\Gamma$ has an end point on a fully marked component.
	\end{enumerate}
\end{prp}

\begin{definition}\label{DefinitionFukayaCategory}
Suppose $\Sigma$ is not an annulus with two fully marked components and let $\Gamma \subset \Sigma$ be a full graded arc system. The (topological) \textbf{partially wrapped Fukaya category} of $\Sigma$ is the $A_\infty$-category $\Fuk(\Sigma) \coloneqq \Tw^+ \cF(\Gamma)$. If $\Sigma$ is punctured, we refer to $\Fuk(\Sigma)$ as the \textbf{wrapped Fukaya category}.
\end{definition}
\noindent The Fukaya category in the remaining case of an annulus with two fully marked components $B$ and $B'$ is by definition the category $\Perf(\Bbbk[x,x^{-1}]$ where $x$ is of degree $\omega(B)=-\omega(B')$.

\begin{rem}\label{RemarkAbouzaidWrappedFukaya}
It was proved by Abouzaid in the appendix of \cite{BocklandtMirrorSymmetryPuncturedSurfaces} that the wrapped Fukaya category of a punctured surface in the sense of \cite{BocklandtMirrorSymmetryPuncturedSurfaces, HaidenKatzarkovKontsevich} agrees with its symplectic version. We will make use of this fact in our discussion of derived Picard groups of such categories of this case in Section \ref{SectionDerivedPicardGroupWrappedFukayaCategory}.
\end{rem}
\noindent  As proved in \cite{HaidenKatzarkovKontsevich}, the Morita class of $\Fuk(\Sigma)$ is independent of the choice of the graded arc system and hence a well-defined invariant of the graded marked surface $(\Sigma, \cM, \eta)$. Their result is a consequence of the following statement.

\begin{prp}[{\cite[Lemma 3.2]{HaidenKatzarkovKontsevich}}]\label{PropositionInclusionGivesMoritaEquivalences}
Let $(\Sigma, \cM, \eta)$ be a graded marked surface and $\Gamma \subseteq \Gamma' \subseteq \Sigma$ be two full arc systems. Then the canonical inclusion functor $\cF(\Gamma) \hookrightarrow \cF(\Gamma')$ is a Morita equivalence.
\end{prp}
\noindent The same proof as in loc.cit. shows that an inclusion $\Gamma \hookrightarrow \Gamma'$ of finitely full arc systems yields a Morita equivalence $\cF(\Gamma) \rightarrow \cF(\Gamma')$. One obtains the following.

\begin{prp}[{\cite[Proposition 3.3]{HaidenKatzarkovKontsevich},\cite[Proposition 9.2.5]{BoothGoodbodyOpper}}]\label{prop: Morita invariance gentle algbebras}
Let $\Sigma$ be a graded marked surface and $\cA, \cB \subseteq \Sigma$ two full (resp.~finitely-full) arc systems. Then $\cF(\cA)$ and $\cF(\cB)$ are Morita equivalent.
\end{prp}
\noindent A non-obvious fact is that the converse is also true: the Morita class of $\cF(\cA)$ determines the graded marked surface $\Sigma$ up to isomorphism, cf.~the discussion in Section \ref{SectionGradedGentleAlgebras}. The category $\Fuk(\Sigma)$ enjoys the following localisation property.
\begin{thm}[{\cite[Proposition 3.6]{HaidenKatzarkovKontsevich}}]
Let $\Sigma$ be graded marked surface and let $\gamma \subseteq \partial \Sigma$ be a boundary arc.  Denote by $\cT \subseteq \Fuk(\Sigma)$ the thick closure of $\gamma$ and $\Sigma'$ the graded marked surface obtained from $\Sigma$ by adding $\gamma$ to the set of marked points. Then $\Fuk(\Sigma')$ is a Drinfeld quotient of $\Fuk(\Sigma)$ and there exists a triangle equivalence
\begin{displaymath}
 \quotient{\Fuk(\Sigma)}{\cT} \cong \Fuk(\Sigma').
\end{displaymath}
\end{thm}

\section{Graded gentle algebras}\label{SectionGradedGentleAlgebras}

\noindent The category algebras of the $\Bbbk$-linear category underlying the $A_\infty$-category $\cF(\Gamma)$ of a full or finitely full formal arc system on a graded marked surface belongs to the class of so-called \textit{graded gentle algebras}. As shown in \cite{LekiliPolishchukGentle, OpperPlamondonSchroll} every graded gentle algebra $A$ which is homologically smooth or proper arises non-uniquely from a graded arc system on graded marked surface $(\Sigma_A, \cM_A, \eta_A)$. Hence, for the purposes of this paper we \textit{define} a \textbf{graded gentle algebra} as any graded algebra which is isomorphic to $A_{\cF(\Gamma)}$ for some full or finitely full formal graded arc system $\Gamma$. As such and unless specified otherwise, every graded gentle algebra which we consider in this paper will be assumed to be homologically smooth \textit{or} proper. In the former case one can assume that $\Gamma$ is full and in the latter case that $\Gamma$ is finitely full. For an equivalent definition of graded gentle algebras in terms of graded quivers with relations we refer the reader to \cite{OpperPlamondonSchroll}. We often treat the two cases of graded gentle algebras -  the homologically smooth ones and the proper ones - simultaneously. To facilitate discussions, we use the following short hand notation.

\begin{notation}Let $A$ be a graded gentle algebra. Define
	\begin{displaymath}
		\cT_A \coloneqq \begin{cases} \Dfd{A}, & \text{if $A$ is proper;} \\
			\Perf(A), & \text{if $A$ is homologically smooth.}
		\end{cases}
	\end{displaymath}
\end{notation}
\noindent We record the following:
\begin{itemize}
	\item If $A$ is homologically smooth and proper then there is a quasi-equivalence $\Perf(A) \simeq \Dfd{A}$, so that, up to quasi-equivalence, the category $\cT_A$ is well-defined in either case.
	\item The category $\cT_A$ contains a thick subcategory $\cT_{A}^{\fin}$ consisting of all objects $Y \in \cT_A$ such that $\Hom_{\cT_A}^{\bullet}(X, Y)$ has finite-dimensional total cohomology for all $X \in  \cT_A$. If $A$ is homologically smooth then $\Dfd{A}=\cT_A^{\fin} \subseteq \cT_A$. Similarly, if $A$ is proper then $\cT_A^{\fin}=\Perf(A) \subseteq \Dfd{A}=\cT_A$.
	
	\item If $A$ is proper, then  the category $\cT_A^{\fin}$ admits a Serre-functor and Auslander-Reiten triangles as follows from \cite{OpperPlamondonSchroll, OpperDerivedEquivalences}. Throughout the paper, we denote by $\tau\colon \cT_A^{\fin} \rightarrow \cT_A^{\fin}$ the Auslander-Reiten translation in this case.
\end{itemize}
\noindent Finally, we recall the following description of the category $\cT_A$ in the proper case.
	\begin{thm}[{\cite[Theorem 9.4.4]{BoothGoodbodyOpper}}]\label{TheoremGentleAlgebrasAreLocal}
		Let $A$ be a proper graded gentle algebra. Then $\Dfd{A}$ is equivalent to the thick subcategory of $\cD(A)$ generated by the simple $A$-modules, or equivalently, the thick closure of $A/\rad(A)$, considered as a dg $A$-module with trivial differential.
	\end{thm}
	\noindent We note that \Cref{TheoremGentleAlgebrasAreLocal} is far from trivial since we do not make any assumptions on the grading of $A$. Theorem \ref{TheoremGentleAlgebrasAreLocal} is closely connected to the  \textit{reflexivity} of $\cT_A$ which we discuss in Section \ref{SectionReflexivity}. 

\subsection{Koszul duality: smooth versus proper}\label{SectionKoszulDualityFukayaGentle}\ \medskip

\noindent The classes of proper and homologically smooth graded gentle algebras are Koszul dual to each other.   Given a graded marked surface $(\Sigma, \cM, \eta)$ and a boundary component $B \subseteq \partial \Sigma$ equipped with its induced orientation, the \textbf{winding number of $B$} is by definition the winding number of any orientation preserving diffeomorphism $S^1 \xrightarrow{\sim} B$.

\begin{prp}\label{PropositionKoszulFunctors} Let $A$ be a proper graded gentle algebra. Then there exists a homologically smooth graded gentle algebra $A^{\invex}$ such that $A$ is quasi-isomorphic to the Koszul dual of $A^{\invex}$ and triangulated functors
	\begin{displaymath}
		\begin{array}{lcr}
			\Perf(A) \hookrightarrow \Dfd{A^{\invex}} & \text{and}& \Perf(A^{\invex}) \rightarrow \Dfd{A},
		\end{array}
	\end{displaymath}	
	\noindent such that the first map is a fully-faithful embedding and the second map is essentially surjective. Moreover, either of these functors is an equivalence if and only if $\Sigma_A$ contains no fully marked components with vanishing winding number. 
\end{prp}
\begin{proof}
The essential surjectivity follows from {\cite[Appendix C, Step 6]{OpperPlamondonSchroll}} while the assertion about fully-faithfulness of the first functor is a consequence of the definition of Koszul duality.
\end{proof}

\noindent The failure of the functors in \Cref{PropositionKoszulFunctors} to be equivalences is explained by the fact that there is an equivalence $\Dfd{A} \simeq \Perf(\widehat{A^{\invex}})$, where  $\widehat{A^{\invex}} \simeq {(A^{\invex})}^{!!} \simeq A^!$. 
denotes the (graded) completion of $A^{\invex}$ at its arrow ideal $\cI$. From a categorical view point the double Koszul dual is equivalent to the derived completion in the sense of \cite{EfimovDerivedCompletion} along the thick subcategory generated by $A^{\invex}/\cI \in \Perf(A^{\invex})$.

The algebra $A^{\invex}$ is quasi-isomorphic to the ``left'' Koszul dual $\Omega(A^{\vee})$ of $A$, where $\Omega$ denotes the cobar construction and $A^{\vee}$ the dual coalgebra with respect to the duality $(-)^{\vee}=\Hom_{\cS}(-, \cS)$ over the semi-simple algebra $\cS=A/\rad(A)$. The inverse of this process starts from a homologically smooth graded gentle algebra $B$ and passes to its usual Koszul dual $B^!$ over $\cS$ which is quasi-isomorphic to a proper graded gentle algebra. In terms of arc systems, the duality can further be interpreted as follows. Suppose $A$ is a homologically smooth graded gentle algebra and $\Gamma \subseteq \Sigma$ is a full formal arc system such that $A \cong A_{\cF(\Gamma)}$. Then $A^{!}$ is quasi-isomorphic to $A_{\cF(\Gamma^{\bot})}$, where $\Gamma^{\bot} \subseteq \Sigma$ is the dual arc system, cf.~\cite[Section 9.3]{BoothGoodbodyOpper}. In combinatorial terms this says that if $(Q_{\Gamma}, R_{\Gamma})$ denotes the gentle quiver associated to $\Gamma$, then the gentle quiver $(Q_{\Gamma^!}, R_{\Gamma^!})$ is its quadratic dual obtained by reversing the direction of all arrows in $Q_{\Gamma}$ and $R_{\Gamma^!}\coloneqq \{\alpha \beta \mid \alpha, \beta \in Q_{\Gamma}^1, \alpha \beta \not \in R_{\Gamma}\}$. Likewise, if $A$ is proper instead, the quadratic dual allows one to pass back to $A^{\invex}$. After additionally passing to arrow completions of the associated path algebra one can also give a description of $A^!  \simeq {(A^{\invex})}^{!!}$.

\begin{rem}\label{rem: block decomposition}The proof of \Cref{PropositionKoszulFunctors} actually shows something slightly stronger: the category $\Dfd{A^{\invex}}$ has an orthogonal decomposition of the form
\begin{equation}
	\Dfd{A^{\invex}} \simeq \Perf(A) \times \prod_{C} \Dfd{\Bbbk[x, x^{-1}]}.
\end{equation}
\noindent where $|x|=0$ and $C$ ranges over all fully marked boundary components of $\Sigma_A$ with vanishing winding number.
\end{rem}

	\subsection{Geometric model for derived categories of gentle algebras}\ \medskip

\noindent We recall a few basic facts about the structure of derived categories of graded gentle algebras  A \textit{local system} on a curve $\gamma\colon \Omega \rightarrow \Sigma$ on a graded marked surface is a vector space valued functor from the fundamental groupoid\footnote{For a topological space $X$, objects in the category $\pi_1(X)$ are points of $X$ and for $x,y \in X$, $\Hom_{X)}(x,y)$ is the set of homotopy classes of paths from $x$ to $y$ in $X$. Composition is defined via concatenation.} $\pi_1(\Omega)$ of $\Omega$. Upon the choice of a base point $\ast \in \Omega$ the category of such local systems is equivalent to representation of $\pi_1(\Omega, \ast)$ and hence equivalent to modules over $\Bbbk$ if $\Omega\simeq [0,1]$ or $\Bbbk[t,t^{-1}]$ if $\Omega \simeq S^1$.
	
	\begin{thm}[{\cite{HaidenKatzarkovKontsevich, OpperPlamondonSchroll}]}]\label{TheoremBijectionObjectsCurves}
		Let $A$ be a graded gentle algebra. Then the following is true.
		\begin{itemize}
			\item If $A$ is homologically smooth, there exists a bijection between isomorphism classes of indecomposable objects in $\Perf(A)$ and homotopy classes of graded curves on $\Sigma_A$ equipped with the isomorphism class of an indecomposable local system.
			\item If $A$ is proper, there exists a bijection between isomorphism classes of indecomposable objects in $\Dfd{A}$ and homotopy classes of graded curves on $\Sigma_A$ equipped with an isomorphism class of an indecomposable local system, which are not isotopic to a fully marked boundary component.
		\end{itemize}
	\end{thm}
\noindent Given an indecomposable object $X \in \cT_A$, we denote by $\gamma(X)$ the associated homotopy class of curves and by $\gamma_X$ any representative in there. Likewise, given such a curve $\gamma$, we denote by $X_{\gamma}$ a representative in the corresponding isomorphism class of objects. An arc is further called \textbf{projective}, if it represents the isomorphism class associated to a shift of a direct summand of $A$. There also descriptions of the morphism spaces in $\cT_A$ if $A$ is proper, cf.~\cite[Appendix A]{OpperPlamondonSchroll}. Here, we will not need the full extent of the description of morphisms in $\cT_A$ in this case but we only record the following
\begin{prp}[{\cite[Theorem A.2]{OpperPlamondonSchroll}}]\label{prop: description homs}
Let $A$ be a proper graded gentle algebra. Let $X ,Y \in \cT_A$ be indecomposable and let $\gamma_X$ and $\gamma_Y$ be representing curves on $\Sigma_A$ which we regard as a surface with marked points. We may and will assume that $\gamma_X$ and $\gamma_Y$ are in minimal position, that is, the number of intersections between $\gamma_X$ and $\gamma_Y$ is minimal within their respective homotopy classes. If $\gamma_X$ and $\gamma_Y$ intersect, then $\Hom^{\bullet}(X,Y)\neq 0$.
\end{prp}

	\section{Derived Picard groups}
	
	\noindent Being central to the present article, we remind ourselves of a few basic facts about the relationship between $A_\infty$-functors, morphisms in the homotopy category of dg categories and derived Picard groups.
	
	\subsection{The homotopy category of dg categories and derived Picard groups}\ \medskip
	
	\noindent The category $\dgcat$ of small dg categories admits a Dwyer-Kan type model structure whose weak equivalences are the quasi-equivalences \cite{TabuadaDwykerKanModelStructure}. We denote by $\Hqe$, its homotopy category, the \textit{homotopy category of dg categories}. This is a symmetric monoidal category whose tensor product is given by the derived tensor product of dg categories. We also denote by $\Hmo$ the \textbf{Morita category}, the full subcategory of $\Hqe$ consisting of all perfect dg categories. We denote by $\Acat$ the category of small $A_\infty$-categories with $A_\infty$-functors as morphisms and $\Acatc$ the category of strictly unital $A_\infty$-categories and strictly unital $A_\infty$-functors.  Canonaco-Ornaghi-Stellari compared the category $\Hqe$ with the homotopy theory of $A_\infty$-categories.
	\begin{thm}[{\cite{CanonacoOrnaghiStellari}}]The inclusion of $\dgcat \subseteq \Acatc$ descends to equivalences 	
	\begin{displaymath}
		\Hqe \simeq \HAcat \simeq  \quotient{\Acatc}{\approx},
	\end{displaymath}
	\noindent where $\quotient{\Acatc}{\approx}$ denotes the quotient of $\Acatc$ modulo weak equivalences and $\HAcat$ denotes the localisation of $\Acat$ along the class of quasi-equivalences (``homotopy category of $A_\infty$-categories'').
	\end{thm}
\noindent 	The result allows us the freedom to choose to work with $A_\infty$-categories whenever convenient. The category $\Hqe$ is a \textit{closed} symmetric monoidal category due to a result by To\"en \cite{ToenDerivedMoritaTheory}. An internal Hom object $[\bA, \bB]$ can be described in terms of special bimodules: a cofibrant bimodule $M:\bA \otimes {\bB}^{\op} \rightarrow \operatorname{Ch}(\Bbbk)$ is \textit{right quasi-representable} if for all $A \in \bA$, the right dg $\bB$-module $M(-, A)$ is quasi-isomorphic to a representable one. Such bimodules are also called \textit{quasi-functors} in \cite{GenoveseUniquenessLifting} and the dg categories formed by them are a model $[\bA, \bB]$. However, for the most part we will use the following alternative description which was suggested by Kontsevich and later proved by Canoncao-Ornaghi-Stellari.
	\begin{thm}[{\cite{CanonacoOrnaghiStellari, CanonacoOrnaghiStellariCommutativeRings}}]\label{TheoremInternalHomThroughAInfinityFunctors}Let $\bA$ and $\bB$ be (strictly unital) $A_\infty$-categories. Then the internal Hom object $[\bA, \bB]$ inside $\HAcat$ is quasi-equivalent to the $A_\infty$-category $\Fun(\bA, \bB)$. Moreover, the quasi-equivalence is natural in $\bA$ and $\bB$.
	\end{thm}
	\noindent In particular, $\Hom_{\Hqe}(\bA, \bB)$ is naturally in bijection with the set of weak equivalence classes of $A_\infty$-functors $\bA \rightarrow \bB$ for arbitrary strictly unital  $A_\infty$-categories.\medskip
	
\noindent	 The category $\Hqe \cong \HAcat$ is the correct framework of universal properties for various constructions.
\begin{prp}[{\cite[Theorem 7.2]{ToenDerivedMoritaTheory}}]\label{PropositionUniversalPropertiesHQE} The assignment  $\bA \mapsto \Perf \bA$ descends to a left adjoint to the inclusion of the full subcategory of $\HAcat$ consisting of all perfect $A_\infty$-categories.
\end{prp}
\noindent The unit of the adjunction corresponds to the canonical morphism in $\Hqe$ 
\begin{displaymath}
\iota\colon  \bA \rightarrow \Perf(\bA).
\end{displaymath}
\noindent which induces a fully-faithful functor on the level of homotopy categories. \Cref{PropositionUniversalPropertiesHQE} asserts that for each $A_\infty$-functor $F\colon  \bA \rightarrow \bA$,  the composition $\iota \circ F$ can be extended to an $A_\infty$-functor $\Perf(\bA) \rightarrow \Perf(\bA)$ in a unique way up to weak equivalence and hence yields a canonical embedding 
\begin{equation}\label{EquationEmbeddingAutomorphismDerivedPicardGroups}
	\begin{tikzcd}
	\Aut^{\infty}(\bA) \arrow[hookrightarrow]{r} & \Aut^{\infty}\big(\Perf(\bA)\big),
	\end{tikzcd}
\end{equation}

For later use, we record the following observation which allows one to replace the objects in the image of an $A_\infty$-functor by quasi-isomorphic ones.
\begin{lem}\label{LemmaImageOfInclusion}
	The image of \eqref{EquationEmbeddingAutomorphismDerivedPicardGroups} agrees with the weak equivalence classes of all functors $F \in \Aut^{\infty}\big(\Perf(\bA)\big)$ such that for all $A \in \bA$, there exists $A' \in \bA$ such that $F(A)\cong A'$ in $\cD(\bA)$.
\end{lem}
\begin{proof}One inclusion is clear, so let $F\colon  \Perf(\bA) \rightarrow \Perf(\bA)$ be an $A_\infty$-functor such that $F(A) \cong A$ in $\cD(\bA)$ for all $A \in \bA$. 	 The embedding \eqref{EquationEmbeddingAutomorphismDerivedPicardGroups} is natural in $\bA$ with respect to weakly invertible functors and hence we may assume up to quasi-isomorphism that $\bA$ is a dg category and that $\bA \hookrightarrow \Perf(\bA)$ is a fully-faithful dg functor, that is, induces isomorphisms on morphism complexes. Let $F(\bA) \coloneqq  \{F(A) \mid A \in \bA\}$ denote the full dg subcategory of the category of dg $A$-modules spanned by the image of $F$.  By our assumptions on $F$, the inclusions $\bA \hookrightarrow \bA \cup F(\bA)$ and $\iota\colon F(\bA) \hookrightarrow \bA \cup F(\bA)$ are quasi-equivalences between full dg subcategories of $\Perf(\bA)$ and hence weakly invertible. Let $\kappa\colon \bA \cup F(\bA) \rightarrow \bA$ denote a weak inverse of the first inclusion. We claim that the composition $G$ of $\kappa \circ \iota $ with $\bA \hookrightarrow \Perf(\bA)$ is weakly equivalent to the inclusion $F(\bA) \hookrightarrow \Perf(\bA)$. The diagram of dg functors
		
\begin{displaymath}
\begin{tikzcd}
		 & \Perf(\bA) &  \\
		 & \bA \cup F(\bA) \arrow[hookrightarrow]{u} & \\
		 \bA \arrow[hookrightarrow]{uur} \arrow[hookrightarrow]{ur} & &  F(\bA) \arrow[hookrightarrow]{uul} \arrow[hookrightarrow]{ul}
\end{tikzcd}
\end{displaymath}
\noindent commutes. As such, it commutes in $\HAcat$ and the claim follows from the fact that $\kappa$ is the inverse to $\bA \hookrightarrow \bA \cup F(\bA)$ in $\HAcat$. We conclude that $F' \coloneqq G \circ F|_{\bA} \approx j \circ  F|_{\bA}=F$ as $A_\infty$-functors $\bA \rightarrow \Perf(\bA)$, where $F|_{\bA}\colon \bA \rightarrow F(\bA)$ denotes the restriction of $F$ to $\bA$ and $j\colon F(\bA) \hookrightarrow \Perf(\bA)$ the inclusion. Now $F'$ extends canonically to an element of $\Aut^{\infty}(\Perf(\bA))$ which is weakly equivalent to $F$ and which satisfies $F'(A) \in \bA$ by construction.
\end{proof}

\begin{definition}Let $\bA$ be a cohomologically unital $A_\infty$-category. Its \textbf{derived Picard group} is the group $\DPic(\bA)$ of units of $\Hom_{\Hqe}(\Perf(\bA), \Perf(\bA))$, or, equivalently, the group $\Aut^{\infty}(\Perf(\bA))$.
\end{definition}
\noindent For dg categories $\bA$ and using quasi-functors, one can give an equivalent definition of $\DPic(\bA)$  in terms of bimodules cf.~\cite{KellerDerivedPicardGroup, ToenDerivedMoritaTheory}. The category $\cD(\bA \otimes \bA^{\otimes})$ is closed non-symmetric monoidal with tensor product the derived tensor product $- \otimes_{\bA}^{\mathbb{L}} -$ and monoidal unit the diagonal bimodule. T\"oen's description of the internal hom in $\Hqe$ and Theorem \ref{TheoremInternalHomThroughAInfinityFunctors} then imply $\DPic(\bA)$ is isomorphic to the Picard group of $\cD(\bA \otimes \bA)$. Hence for every $X \in \DPic(\bA)$, the functor $-\otimes_{\bA}^{\mathbb{L}} X$ is a coproduct prerserving exact autoequivalence of $\cD(\bA)$ and any such equivalence restricts to an exact equivalence of $\Perf(\bA)$. In this sense, $\DPic(\bA)$ may be viewed as dg enhanced symmetry group of $\Perf(\bA)$ and $\cD(\bA)$ respectively. In particular, there exists a comparison homomorphism
 one obtains a comparison homomorphism
\begin{equation}\label{EquationComparisonMapDerivedPicardGroupAutoequivalenceGroup}
\begin{tikzcd}
\DPic(\bA) \arrow{r} & \Aut(\Perf(\bA)).
\end{tikzcd}
\end{equation}
\noindent  In the next section, we look at subgroups of derived Picard groups defined by special $A_\infty$-functors, called \textit{isotopies}. 
	
	\subsection{Special subgroups of the derived Picard group}\label{SectionSpecialSubgroups}\ \medskip
	
	\noindent Throughout this section, we fix a graded $\Bbbk$-linear category $\cC$ and an  $A_\infty$-category $\bA$.

	\subsubsection{Outer automorphism groups}\ \medskip
	
	\noindent We denote by $\Out(\cC)$ the \textbf{outer automorphism group} of $\cC$, the group of  $\Bbbk$-linear autoequivalences\footnote{Such functors preserve the grading on the morphism spaces.} of $\cC$ modulo natural isomorphisms. This generalises the outer automorphism group of an ordinary algebra.  If $\Bbbk$ is algebraically closed and the category algebra of $\cC$ is finite-dimensional (which implies that $\Ob{\cC}$ is finite) then $\Out(A)$ naturally admits the structure of an algebraic group whose identity component we denote by $\OutO(\cC)$. Each autoequivalence of $\cC$ is naturally identified with a strict $A_\infty$-endofunctor of $\cC$ which, via \eqref{EquationEmbeddingAutomorphismDerivedPicardGroups}, induces a series of embeddings 
	\begin{equation}\label{EquationSeriesEmbeddings}
	\begin{tikzcd}
		\OutO(\cC) \arrow[hookrightarrow]{r} &  \Out(\cC) \arrow[hookrightarrow]{r} & \Aut^{\infty}(\cC) \arrow[hookrightarrow]{r} & \DPic(\cC),
	\end{tikzcd}
\end{equation}
\noindent cf.~\cite[Section 4.2]{OpperIntegration}.

\subsubsection{A candidate for the identity component of the derived Picard group}\ \medskip

\noindent In \cite[Section 4.3]{OpperIntegration}  the author gave a conjectural description of the identity component of the derived Picard group of a graded category $\cC$ which is a ``thickening'' of $\Out^{\circ}(\cC)$. One of the expected properties of this group is its Morita invariance. Indeed, the results of this paper verify this for graded gentle algebras.

 	\begin{definition}\label{DefinitionGroups} Suppose the category algebra of $\cC$ is finite-dimensional. Let $\Aut^{\infty}_{\circ}(\cC) \subseteq \Aut^\infty(\cC)$ the subgroup consisting of all $A_\infty$-functors $F$ such that $F^1 \in \OutO(\cC)$. Likewise, we denote by $\Aut^{\infty}_{+}(\cC) \subseteq \Aut^\infty(\cC)$ the subgroup consisting of the weak equivalence classes of all $A_\infty$-functors $F$ such that $F^1=\operatorname{Id}_{\cC}$.
 \end{definition}
\noindent By \cite[Proposition 4.14]{OpperIntegration}, the projection $F \mapsto F^1$ induces a split surjective homomorphism
\begin{equation}\label{eq: projection to outer automorphism group}
	\begin{tikzcd}
	\Aut^{\infty}_{\circ}(\cC) \arrow{r} & \OutO(\cC).
	\end{tikzcd}
\end{equation}
\noindent The section is the obvious embedding $\OutO(\cC) \hookrightarrow \Aut^{\infty}_{\circ}(\cC)$ as strict $A_\infty$-functors as in \eqref{EquationSeriesEmbeddings}. The kernel of \eqref{eq: projection to outer automorphism group} is isomorphic to $\Aut^{\infty}_+(\bA)$ which yields the decomposition
	\begin{displaymath}
		\Aut^{\infty}_{\circ}(\cC) \cong \Aut^{\infty}_+(\cC) \rtimes \OutO(\cC), 
	\end{displaymath}
	\noindent If $\Bbbk$ is algebraically closed, then $\Aut_{\circ}^{\infty}(\cC)$ is a locally algebraic group, see \cite[Proposition 4.14]{OpperIntegration}. In practice,  $\Aut_{\circ}^{\infty}(\cC)$ is very amenable to computations. This is classical problem for $\OutO(\cC)$ at least in the case of algebras, e.g.~\cite{GuilAsensioSaorinMonomialPicardGroup, GuilAsensioSaorinPicardGroup}. On the other hand, if $\operatorname{char} \Bbbk=0$, $\Aut^{\infty}_+(\cC)$ is isomorphic, as a group, to a concrete subspace of $\HHH^1(\cC, \cC)$ together with its group structure defined by a Baker-Campbell-Hausdorff-type product. Details are recalled in Section \ref{SectionIntegrationHochschild}.
	
	 With the expectation that $\Aut^{\infty}_{\circ}(\cC)$ describes the identity component of $\DPic(\cC)$ under additional assumptions on $\cC$, the author conjectured in \cite[Conjecture 4.10, Remark 6.19]{OpperIntegration} that $\Aut^{\infty}_{\circ}(\cC)$ is a Morita invariant. Implicitly this requires every element in the identity component of $\DPic(\cC)$ to be representable by an $A_\infty$-endofunctor of $\cC$ which may not be the case. The root of the issue is related to possible non-trivial deformations of objects of $\cC$ inside $\cD(\cC)$. 
	 
\begin{exa}\label{ExampleNonRigid} Consider the graded gentle algebra $B$ associated to the gentle quiver
	
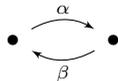
\begin{displaymath}
		\begin{tikzcd}
			\bullet \arrow[bend left]{r}{\alpha} & \arrow[bend left]{l}{\beta} \bullet
		\end{tikzcd}
	\end{displaymath}
	\noindent where $|\alpha|+1=|\beta|=1$ and with the only relation $\alpha\beta=0$. It surface $\Sigma_B$ is a cylinder with a single marked interval on each boundary component. The line field is determined in this case by the winding number $\omega$ of the unique simple loop which is $w=0$ in this case. Let $P$ denote the dg module corresponding to the vertex $s(\alpha)=t(\beta)$. The associated projective arc is the boundary segment and for every $\lambda \in \Bbbk$, the dg module $P_{\lambda}=(P,\lambda \beta \alpha)$ (with $\beta\alpha$ defining the differential) can be seen as a deformation of $P=P_0$ in the sense of Definition \ref{DefinitionBlurringRelation}. One has an equivalence  $\Dfd{B} \simeq \Dfd{\cK}\simeq  \cD^b(\cK)$ where $\cK$ denotes the Kronecker algebra. Under Beilinsons' equivalence $\cD^b{B} \simeq \cD^b(\mathbb{P}^1)$, $P$ and its deformations correspond to the skyscraper sheaves $\Bbbk(x) \in \cD^b(\mathbb{P}^1)$,  $x \in \mathbb{P}^1$. The automorphism group $\PGL_2(\Bbbk)$ of $\mathbb{P}^1$ agrees with the entire derived Picard group and acts transitively on their isomorphism classes. Thus, some of the elements of $\PGL_2(\Bbbk)$ cannot be represented by an $A_\infty$-endofunctor of $B$. In order to avoid additional technicalities, we however circumvent the problem instead by passing to derived equivalent gentle algebras. For example, the problem does not occur when working with $\cK$ instead of the algebra $B$ because the direct summands of $B$ inside $\cD(B)$ do not admit any non-trivial deformations for degree reasons.  
\end{exa}

	\section{Reflexive dg categories}\label{SectionReflexivity}
	\noindent We recall a few  results on reflexive dg categories. By \cite[Proposition 6.6]{GoodbodyReflexivity}, the assignment $\cU \mapsto \Dfd{\cU}$  induces an adjoint pair of functors
	\begin{displaymath}
		\begin{tikzcd}
		\Dfd{-}\colon  \Hqe \arrow[shift left=0.35em]{rr} && \arrow[shift left=0.35em]{ll}\Hqe^{\op} \colon  \Dfd{-}.
		\end{tikzcd}
	\end{displaymath}
	Its unit is the evaluation map and a dg category $\cU$ is \textbf{reflexive} is the evaluation $\cU \rightarrow \Dfd{\Dfd{\cU}}$ is a Morita equivalence, or equivalently if $\cU$ is perfect, it is a quasi-equivalence. The name stems from the fact that such categories are precisely the reflexive objects in the closed monoidal category $\Hmo$ \cite{GoodbodyReflexivity}.  The definition is due to Kuznetsov-Shinder \cite{KuznetsovShinder} but precursors appeared in \cite{BenZviNadlerPreygel}. Reflexive dg categories generalise the duality found between bounded and perfect derived categories of projective schemes and finite-dimensional algebras.

	\noindent For any pair of dg categories $\cA, \cB$, the functor $\Dfd{-}\colon  \Hqe \rightarrow \Hqe^{\op}$ induces a dg functor
	\begin{equation}\label{EquationMapReflexiveDuality}
		\begin{tikzcd}
		\Fun(\cA, \cB) \arrow{r} & \Fun(\Dfd{\cB}, \Dfd{\cA}),
		\end{tikzcd}
	\end{equation}
	 \noindent which is natural in $\cA$ and $\cB$. The following theorem was proved by Kuznetsov-Shinder for functors on the triangulated level and for morphisms in $\Hqe$ follows directly from the results in  \cite{GoodbodyReflexivity}.
		\begin{thm}[{\cite[Corollary 3.16]{KuznetsovShinder}, \cite{GoodbodyReflexivity}}]\label{TheoremBijectionDGFunctors}
		Let $\bA, \bB$ be perfect  dg categories with $\bB$ reflexive. Then \eqref{EquationMapReflexiveDuality} is a quasi-equivalence and restricts to an equivalence between the corresponding subcategories of quasi-equivalences. Likewise, there exists a bijection between the isomorphism classes of triangulated functors $\bA \rightarrow \bB$ and triangulated functors $\Dfd{\bB} \rightarrow \Dfd{\bA}$ which restricts to a bijection between the set of equivalences.
	\end{thm}

	\noindent If $\bA$ is reflexive, then the naturality of \eqref{EquationMapReflexiveDuality} implies the existence of a canonical isomorphism of groups  
	\begin{displaymath}
		\begin{tikzcd}
		\Hom_{\Hmo}(\bA, \bA) \arrow{r} & \Hom_{\Hmo}(\Dfd{\bA}, \Dfd{\bA})^{\operatorname{op}}.			
		\end{tikzcd}
	\end{displaymath} After composition with the canonical anti-automorphism $f \mapsto f^{-1}$ of $\DPic(\Dfd{\bA})$, one therefore has the following.
	\begin{cor}
	Let $\bA$ be a reflexive dg category. There exists an isomorphism of groups
	\begin{equation}\label{EquationIsomorphismPicardGroupsReflexive}
		\DPic(\bA) \cong \DPic(\Dfd{\bA}).
	\end{equation}
	\noindent Moreover, if $\bA$ is proper, then the inverse $\DPic(\Dfd{\bA}) \rightarrow \DPic(\bA)$ is given by restriction of quasi-equivalences to the subcategory $\Perf(\bA) \subseteq \Dfd{\bA}$ and the isomorphism $\DPic(\bA) \simeq \DPic(\Perf(\bA))$\footnote{This was pointed out to the  author by Isambard Goodbody.}
	\end{cor} 
\noindent Our main interest in reflexivity stems from the following result.
\begin{thm}[{\cite[Theorem 9.4.3]{BoothGoodbodyOpper}}]\label{TheoremProperGentleReflexive} Let $A$ be a proper graded gentle algebra. Then $A$ and $\Dfd{A}$ are reflexive. In particular, $\DPic(A)\cong \DPic(\Dfd{A})$.
\end{thm}

\begin{rem}The example $A=\Bbbk[x]$ already shows that the analogue of Theorem \ref{TheoremProperGentleReflexive} is false for homologically smooth graded gentle algebras in general. Of course, if $A$ is homologically smooth and $\Sigma_A$ has contains no fully marked components with vanishing winding number, \Cref{PropositionKoszulFunctors} shows that $\Perf(A) \simeq \Dfd{A^!}$ which is reflexive.
\end{rem}

\section{Derived equivalence classification and geometrisation homomorphisms}\label{Section Geometrisation Homomorphism}

\noindent We recall some results from \cite{OpperDerivedEquivalences} about the relationship between autoequivalence groups of proper graded gentle algebras and mapping class groups. Subsequently we then extend these results to the homologically smooth case, and in Section \ref{SectionDerivedPicardGroupWrappedFukayaCategory} to wrapped Fukaya categories of punctured surfaces. Throughout this section we fix a graded marked surface $(\Sigma, \cM, \eta)$.

\subsection{Deformations of loop objects and rigidity}\label{SectionDeformationsRigidity}

\begin{definition}
 A boundary component $B \subseteq \partial\Sigma$ is \textbf{degenerating} if $\omega_{\eta}(B)=0$ and $B \cap \cM \cong [0,1]$. 
\end{definition}
\noindent Equivalently, a boundary component $B$ is degenerating if $\omega_{\eta}(B)=0$ and $B$ contains a single marked interval. In particular, $B$ contains a unique boundary arc $\gamma$ and we note that every grading of the simple loop $\gamma'$ around $B$ canonically induces a grading on $\gamma$ and vice versa. We recall the following equivalence relation from \cite{OpperDerivedEquivalences}.
\begin{definition}\label{DefinitionBlurringRelation}Let $\gamma, \gamma' \subseteq \Sigma$ be graded curves both equipped with indecomposable local systems. We write $\gamma \sim_{\ast} \gamma$ if $\gamma \simeq \gamma$ as graded curves, or if the set $\{\gamma, \gamma'\}$ contains both a graded boundary arc of a degenerating boundary component $B$ as well as the corresponding simple loop around $B$ with the induced grading equipped with a local system corresponding to a $\Bbbk[x,x^{-1}]$-module of the form $\Bbbk[x]/(X-\lambda)$, $\lambda \in \Bbbk^{\times}$. If $\gamma \simeq_{\ast} \gamma'$ we say that $X_{\gamma'}$ is a \textbf{deformation} of $X_{\gamma}$. A curve $\gamma$ is \textbf{rigid} if its $\simeq_{\ast}$-class contains no loops. A graded gentle algebra is \textbf{rigid} is all its projective arcs are rigid.
\end{definition} 
\noindent In other words, a rigid arc is one which is not homotopic to a boundary arc on a degenerating boundary component. In plain terms, the equivalence relation $\sim_{\ast}$ forgets any information about local systems but also blurs the lines between arcs and loops around degenerating boundary components. The reason for this equivalence relations is to rectify a inevitable imperfection of the correspondences from Theorem \ref{TheoremBijectionObjectsCurves} which arises even in very simple situations, cf.~Example \ref{ExampleImperfectionsKronecker}.

\begin{prp}[{\cite[Proposition 2.16]{OpperDerivedEquivalences}}]\label{prop: rigidity characterisation}
	Suppose $A$ is a proper graded gentle algebra and let $X \in \cT_A=\Dfd{A}$ be indecomposable. If $X \not \in \Perf(A)$, then $\gamma_X$ is rigid. If $X \in \Perf(A)$, then $\gamma_X$ is rigid if and only if  $\tau X \not \cong X$.
\end{prp}

\begin{exa}\label{ExampleImperfectionsKronecker} Contrary to the Kronecker algebra $\cK$, the Morita equivalent gentle algebra $B$ from \Cref{ExampleNonRigid} is not rigid. This is because the projective arcs of $B$ correspond to the object $P_0$ and a projective $\cK$-module under the equivalence $\Dfd{B} \simeq \Dfd{\cK}$. If $\lambda \not \in \{0, \infty\}$; then under the equivalence $\Dfd{B} \simeq \Dfd{\cK}$, the object $P_{\lambda} \in \Dfd{B}$ corresponds to the indecomposable object $C_{\lambda} \in \Dfd{\cK}$ associated to the central simple loop in which case $\lambda$ is encoded in the associated local system. The special cases $\lambda =0, \infty$ on the other hand are represented by the two boundary segments. If $\lambda=[a:b] \in \mathbb{P}^1$, then $C_{\lambda}$ is the mapping cone of $a \alpha + b \beta$ where $\alpha, \beta$ denote the arrows of $\cK$. Algebraically $C_0$ and $C_{\infty}$ are degenerations of the family $C_{\lambda}$ where $a=0$ or $\beta=0$. Topologically one may view the boundary segments as deformations of the central simple loop whose local system `degenerated' in one of two possible ways. As we saw before $\Aut(\Dfd{B}) \cong \PGL_2(\Bbbk)$ acts transitively on all these objects, which show that whether an indecomposable object of $\cD^b(\cK)$ corresponds to an arc or a loop is only a property of the chosen geometric model of the triangulated category $\Dfd{cK}\simeq \Dfd{B}$ but not intrinsic to its triangulated structure. One may think of the rigid arcs as those which correspond to indecomposable objects which are represented by arcs in \textit{all} geometric models. 
\end{exa}
\noindent \Cref{ExampleImperfectionsKronecker} showcases that rigidity is not preserved under Morita equivalences and the following proposition shows that all graded gentle algebras are rigid up to Morita equivalence.

\begin{prp}\label{PropositionGentleRigidity} Let $A$ be a graded gentle algebra. Then it is Morita equivalent to a rigid graded gentle algebra.
\end{prp}
\begin{proof}We start with the special case that $\Sigma_A$ is a disk $\mathbb{D}_{\omega}$ with a single puncture of winding number $\omega \in \mathbb{Z}$ whose boundary component contains a single marked point. The only two arc systems on $\mathbb{D}_w$ correspond to the the graded algebras $B_{\omega}=\Bbbk[x]/(x^2)$ with $|x|=\omega+1$ and $B_{\omega}'=\Bbbk[t]$ where $|t|=\omega$. The algebra $B_w$ is rigid as the simple loop around the puncture does not correspond to objects in $\Dfd{A}$. The algebra $B_w^{\prime}$ is rigid as the relation $\sim_{\ast}$ only applies to arcs on the degenerating boundary component but the projective arc of $B'$ connects the puncture with the marked point on the degenerating boundary component. From now on we will assume that $\Sigma_A$ is not of the form $\mathbb{D}_{\omega}$. In this case, we show that $\Sigma=\Sigma_A$ admits a formal arc system, finitely full or full depending on the case, which does not contain boundary segments of degenerating boundary components. Then the assertion follows from there. Suppose $\cA \subseteq \Sigma$ is an arc system and suppose that $\gamma \in \cA$ is the segment on a degenerating boundary component $B$. There must be a flow between $\gamma$ and another curve $\gamma' \in \cA \setminus \{\gamma\}$. Indeed, otherwise connectedness of the underyling quiver of $A$ would imply that $\{\gamma\}=\cA$ and hence that $A \cong B_{\omega}$ for some $\omega \in \mathbb{Z}$. Consequently, there is a unique $\gamma' \in \cA$ which has a (necessarily unique) irreducible flow $\gamma \rightarrow \gamma'$ and we denote by $\delta$ the concatenation of $\gamma$ and $\gamma'$ along that flow equipped with an arbitrary choice of grading. The collection $\cA'\coloneqq \left(\cA \setminus \{\gamma\}\right) \sqcup \{\gamma'\}$ is a formal arc system, which is finitely full (resp.~full) if and only if $\cA$ is finitely full (resp.~full) and which contains a strictly smaller number of segments of degenerating boundary components. After induction, we obtain an arc system $\cB \subseteq \Sigma$ without such boundary segments. By construction, its associated graded gentle algebra is Morita equivalent to $A$ by Proposition \ref{PropositionInclusionGivesMoritaEquivalences} and its analogue for finitely full arc systems.
\end{proof}

\subsection{The geometrisation homomorphism in the proper case and its lift to the graded mapping class group}

\subsubsection{Mapping class groups}

\begin{definition}
	Let $(\Sigma, \cM)$ be a  marked surface. The \textbf{extended mapping class group} of $\Sigma$ is the group $\MCG^{\pm}(\Sigma)$ of isotopy classes of diffeomorphisms $f\colon  \Sigma \rightarrow \Sigma$ such that $f(\cM)=\cM$ with respect to isotopies relative $\cM$. The \textbf{mapping class group} is the subgroup $\MCG(\Sigma) \subseteq \MCG^{\pm}(\Sigma)$ of index two consisting of isotopy classes of orientation preserving diffeomorphisms. If $\eta$ is a line field on $\Sigma$, then the \textbf{ graded  mapping class group} of the graded marked surface $(\Sigma, \cM, \eta)$ is the group of isotopy classes of isomorphisms $(f, \tilde{f})\colon (\Sigma, \eta) \rightarrow (\Sigma, \eta)$. 
\end{definition}
\noindent By construction, there is a natural forgetful homomorphism  $\MCG_{\gr}(\Sigma) \rightarrow  \MCG(\Sigma)$ with kernel $\mathbb{Z}$.

\begin{rem}\label{rem: graded MCG forgetful map not surjective}The forgetful morphism is not surjective in general. For example, the Dehn twist $D$ about any non-gradable loop $\delta$ on $\Sigma$ is not in its image. To see this, note that for any immersed loop $\gamma$, $\omega(D(\gamma))=\omega(\gamma) + n \cdot \omega(\delta)$, where $n$ denotes the (algebraic) intersection number of $\gamma$ and $\delta$. If $\omega(\delta) \neq 0$, then $D$ does not preserve all winding numbers and hence cannot be extended to a graded mapping class.
	\end{rem}

\begin{thm}\label{TheoremgeometrisationHomomorphism}
	Let $A$ and $B$ be proper graded gentle algebras. Then any triangle equivalence $\cT_A \rightarrow \cT_B$ induces a diffeomorphism $\Sigma_A \rightarrow \Sigma_B$ of graded marked surfaces. Moreover, there exists a group homomorphism
	\begin{displaymath}
		\begin{tikzcd}
			\Psi'=\Psi_A'\colon  \Aut\big(\cT_A\big) \arrow{r}  & \MCG(\Sigma_A),
		\end{tikzcd}
	\end{displaymath}
\noindent such that for all rigid arcs $\gamma \subseteq \Sigma_A$, and all $F \in \Aut\big(\cT_A\big)$, $\Psi'(F)(\gamma_X) \simeq \gamma_{F(X)}$. 
\end{thm}
\begin{proof}
This follows from the main result in \cite{OpperDerivedEquivalences} which shows that such a homomorphism exists if $\cT_A$ is a \textit{surface-like} category in the sense of \cite[Definition 2.1]{OpperDerivedEquivalences}. The relevant part works equally when the assumption 2) III) is replaced by the assumption that the elements ``$\mathfrak{B}(q)(j)$'' form a \textit{Schauder basis} instead of a basis, cf.~\cite[Theorem A.2]{OpperPlamondonSchroll}, in other words if we pass to graded completions of the morphism spaces. All other assumptions in  \cite[Definition 2.1]{OpperDerivedEquivalences} except for 4) are a direct consequence of the results in the appendices of \cite{OpperPlamondonSchroll}. Finally, \cite[Definition 2.1, 4)]{OpperDerivedEquivalences} is proved with minor adaptations in the same way as \cite[Theorem 2.8]{OpperDerivedEquivalences}.
\end{proof}
\noindent The general version of \cite{OpperDerivedEquivalences} (under the assumptions of \cite{OpperDerivedEquivalences}) allows arbitrary curves but replaces $\simeq$ by the weaker relation $\simeq_{\ast}$ and states that $\Psi'(F)(X_{\gamma}) \simeq_{\ast} X_{F(X)}$ for any curve $\gamma$ and all local systems, under the assumption that $\Bbbk$ is algebraically closed. We don't expect this to change in the case of proper graded gentle algebras but it requires more care in verifying whether the proofs still work. The following is a small but important improvement to Theorem \ref{TheoremgeometrisationHomomorphism}.
\begin{prp}\label{PropositionGradedLiftgeometrisation}Let $A$ be a proper graded gentle algebra. The homomorphism $\Psi_A$ admits a unique lift $$\Psi=\Psi_A\colon \Aut(\cT_A) \rightarrow \MCG_{\gr}(\Sigma_A)$$ along the forgetful homomorphism $\MCG_{\gr}(\Sigma_A) \rightarrow \MCG(\Sigma_A)$ with the property that 
	\begin{displaymath}
		X_{\Psi(F)(\gamma)} \cong X_{F(X_\gamma)}
	\end{displaymath}
	\noindent for all $F \in \Aut(\cT_A)$ and all rigid arcs $\gamma \subseteq \Sigma_A$.
\end{prp}
\begin{proof}
	The kernel of the projection $\MCG_{\gr}(\Sigma_A) \rightarrow \MCG(\Sigma_A)$ is given by the shift and hence uniqueness follows immediately. Let $\cA \subseteq \Sigma_A$ be a finitely full arc system consisting of rigid arcs. This exists by \Cref{PropositionGentleRigidity} and we denote by $\chi \subset \cT_A$ the collection of objects $X_{\gamma}$, $\gamma \in \cA$. By construction of $\Psi'$ in \Cref{TheoremgeometrisationHomomorphism}, the essential image of $\chi$ under an autoequivalence $F \in \Aut(\cT_A)$ corresponds to a graded arc system $\cA_F \subseteq \Sigma$ whose underlying ungraded arc system is $\cA'\coloneqq f(\cA)$, where $f \coloneqq \Psi_A'(F)$. Since $\cA$ is finitely-full and $F$ commutes with the shift functor, the grading on the arcs of $\cA'$ can be extended to a graded diffeomorphism $\Psi(F)\coloneqq (f,\tilde{f})$ such that for each $\gamma \in \cA$ and each point $z=\gamma(t) \in \gamma$, the path $\tilde{f}_{\gamma(t)}$ in $\mathbb{P}(T\Sigma)_{z} \cong S^1$ from $f^{\ast}\eta(z)$ to $\eta(z)$ coincides with the concatenation $\delta_1 \ast \delta_2$, where 
	\begin{itemize}
		\item $\delta_1$ denotes the path from $f^{\ast}\eta(z)$ to $\dot\gamma(t)$ which is associated to the $f^{\ast}(\eta)$-grading on $\gamma$ obtained by pulling back the given $\eta$-grading on $f(\gamma) \simeq \gamma_{F(X)} \in \cA'$ along $f$;
		\item  $\delta_2$ denotes the path from $\dot\gamma(t)$ to $\eta(z)$ given by the inverse of the path associated to the given $\eta$-grading on $\gamma$.
	\end{itemize}
\end{proof}
\begin{definition}\label{DefinitionInfinityGeometrisationHomomorphism}
	Let $A$ be a proper graded gentle algebra. Define 
	\begin{displaymath}
		\begin{tikzcd}
			\Psi^{\infty}_A\colon  \DPic(\cT_A) \arrow{r} & \MCG_{\gr}(\Sigma_A)
		\end{tikzcd}
	\end{displaymath}
	\noindent as the composition $\Psi_A\colon  \Aut(\cT_A) \rightarrow \MCG_{\gr}(\Sigma_A)$ with the comparison homomorphism \eqref{EquationComparisonMapDerivedPicardGroupAutoequivalenceGroup}.
\end{definition}
\noindent We refer to $\Psi_A$ and $\Psi_A^{\infty}$ as \textit{geometrisation homomorphisms}.

\subsection{Geometrisation homomorphisms in the homologically smooth case}\ \medskip

\noindent Our next objective is the generalisation of $\Psi_A$ to all homologically smooth graded gentle algebras. There are two natural pathways here. The first would require us to extend the geometric models to this case, e.g.~via localisation techniques, cf.~Section \ref{SectionFukayaCategories}. Whilst we do not expect any major difficulties, a proof would require the discussion of a number of details as in the special case considered in \cite{OpperKodairaCycles}. To avoid this, we propose a more direct approach via Koszul duality. For what follows, we recall that the graded marked surface of the homologically smooth gentle algebra $\Bbbk[t]$, $|t|=0$, is a cylinder with a single marked interval and a fully marked component with vanishing winding number\footnote{Note that this data specifies the graded marked surface uniquely up to isomorphism.}. Consequently, its Koszul dual, the proper graded gentle algebra $\Bbbk[x]/(x^2)$, $|x|=1$, has the same surface.

\begin{thm}\label{TheoremGeometrisationHomomorphismSmoothCase}
Let $A$ and $B$ be homologically smooth graded gentle algebras. Then every quasi-equivalence  $\Perf(A) \rightarrow \Perf(B)$ induces a diffeomorphism of graded marked surfaces $\Sigma_A \rightarrow \Sigma_B$. Moreover, there exists a group homomorphism
	\begin{displaymath}
		\begin{tikzcd}
			\Psi^{\infty}=\Psi_A^{\infty}\colon  \DPic(A) \arrow{r}  & \MCG_{\gr}(\Sigma_A),
		\end{tikzcd}
	\end{displaymath}
	\noindent which factors over a homomorphism $\Psi_A\colon \Aut(\cT_A) \rightarrow \MCG_{\gr}(\Sigma_A)$. It has the property that for every simple rigid arc $\gamma \subseteq \Sigma_A$ which is finite or semi-infinite and all $F \in \Aut(\cT_A)$, $\Psi_A(F)(\gamma) \simeq \gamma_{F(X)}$.
\end{thm}
\begin{proof}
	Let $F$ be a triangle equivalence $\cT_A \rightarrow \cT_B$. Then $F$ restricts to an equivalence $\cT_A^{\fin} \rightarrow \cT_B^{\fin}$. By \Cref{rem: block decomposition}, $\cT_A^{\fin}$ admits an orthogonal decomposition
	\begin{equation}
		\cT_A^{\fin} \simeq \Perf(A^!) \times \prod_{C} \Dfd{\Bbbk[x, x^{-1}]},  
	\end{equation}
	\noindent where $|x|=0$ and $C$ ranges over all fully marked components with vanishing winding number. The same decomposition exists for $B$ and we want to show that $F$ restricts to an equivalence between $\Perf(A^!) \rightarrow \Perf(B^!)$ in which case the result follows from the proper case. In the case $A=B$ we then define $\Psi_A(F)$ simply as $\Psi_{A^!}(F|_{\Perf(A^!)})$.
We observe that $\Perf(A^!)$ is indecomposable in the sense that is cannot be further decomposed as product of of a pair non-trivial orthogonal categories. This immediately follows from \Cref{prop: description homs} and the fact that every two curves on a marked surface, considered as a surface with marked points, can be connected by a finite sequence of curves so that each consecutive pair in the sequence is in minimal position and intersects at a marked point or the interior. Each block $\Dfd{\Bbbk[x, x^{-1}]}$ itself admits an orthogonal decomposition into indecomposable blocks which are homogeneous tubes of rank $1$, that is, equivalent as a triangulated category to $\Perf(\Bbbk[x]/(x^2)) \simeq \Perf({\Bbbk[t]}^{!})$ with $|x|=1$ and $|t|=0$. Now, we observe that by \Cref{TheoremgeometrisationHomomorphism} an equivalence $\Perf(A^!) \simeq \Perf(R)$, $R=\Bbbk[x]/(x^2)$ implies that $\Sigma_A \cong \Sigma_{A^!} \cong \Sigma_{R}$ as graded marked surfaces. Hence, if all indecomposable blocks of $\Perf(A) \simeq \Perf(B)$ are isomorphic, then $\Sigma_A \cong \Sigma_R \cong \Sigma_B$ as graded marked surfaces. But there is only one finitely full arc system on $\Sigma_R$ and hence we must have $A \cong B \cong R$. From now on, we will therefore assume that $\Sigma_A \not \simeq \Sigma_R \not \simeq \Sigma_B$. Under these assumptions and by our previous discussion it follows that the indecomposable block $\Perf(A^!)$ is distinguished among all the others and shows that $F$ restricts to a triangle equivalence of $\Perf(A^!) \rightarrow \Perf(B^!)$. Composition with the isomorphisms $\DPic(\Dfd{A^!})\cong \DPic(A^!)$ and $\Aut(\Dfd{A^!})\cong \Aut(A^!)$  from \Cref{TheoremProperGentleReflexive} with $\Psi_{A^!}$ and $\Psi_A^{\infty}$ yields the desired homomorphisms $\Psi_A$ and $\Psi_A^{\infty}$. The claimed additional property of $\Psi_A(F)$ follows immediately for all $X \in \Perf(A^!) \subseteq \cT_A$ from the corresponding property of $\Psi_{A^!}$ and hence for all finite arcs $\delta$ which are not boundary arcs on a degenerating boundary component. 
	
	Next, let $\delta$ be a simple semi-infinite arc. Then by adapting the proof of \cite[Proposition 9.4.1]{BoothGoodbodyOpper}, one can show that $\Perf(A)$ admits a semiorthogonal decomposition $\langle \cU, \cV\rangle$, where $\cV \simeq \Perf(\Bbbk[t])$, $|t| \in \mathbb{Z}$, denotes the thick closure of $X_{\delta}$. We make a few observations about this decomposition. First, up to shift $X_{\delta}$ is the unique indecomposable object in $\cV$ with infinite-dimensional graded endomorphism ring. Secondly, $\cU={^{\perp}\cV}$ is uniquely determined as the left orthogonal class to $\cV$ and likewise, $\cV=\cU^{\perp}$. We claim that ${^{\perp}\cV}={^{\perp}(\cV \cap \Dfd{A})}$. The inclusion ``$\subseteq$'' is automatic. For the other, let $E \in {^{\perp}(\cV \cap \Dfd{A})}$ and suppose $E \not \in {^{\perp}\cV}$. We will use the graded module structure of $\cE=\Hom^{\bullet}(E, X_{\delta})$ over $\Bbbk[t]\cong \End^{\bullet}(X_{\delta})$. By regarding $\delta$ as part of a full arc system, the definition of morphisms and composition in $\Fuk(\Sigma_A)$ and a generation argument show that the $\cE$ is finitely generated over $\Bbbk[t]$ and hence noetherian. Now, assume that $f\colon E \rightarrow X_{\delta}$ is a non-zero morphism. It is sufficient to show that there exists an object $Y \in \cV \cap \Dfd{A}$ and a non-zero morphism $E \rightarrow Y$. Let $M \subseteq \cE$ denote the graded $\Bbbk[t]$-module generated by $f$. If $M$ is torsion free, then because $\cE$ is noetherian, there exists $f' \in \cE$ and $n \geq 0$ such that $f=t^n.f'$ and such that $n$ is maximal with this property. Consider the distinguished triangle $Y[-1] \xrightarrow{} X_{\delta} \xrightarrow{} X_{\delta} \xrightarrow{\beta} Y$ corresponding to the triangle $\Bbbk[-1] \rightarrow \Bbbk[t] \xrightarrow{t} \Bbbk[t] \rightarrow \Bbbk$. Then $Y \in \Dfd{A} \cap \cV$ and $0 \neq \beta \circ f$ by maximality of $n$. Likewise, if $f$ is torsion and $u \in \Bbbk[t]$ is an element in the annihilator, then $u \circ f=0$ and hence $f$ factors through a non-zero morphism from $E$ to the mapping cone $Z$ of $u[-1]$. Here, the latter is identified with a morphism $X_{\delta}[-1] \rightarrow X_{\delta}[-1]$. Since $Z \in \Dfd{A}$ this proves our claim. The equivalence $F$ maps $(\cU, \cV)$ to another semiorthogonal decomposition $(\cU', \cV')$. For all finite arcs $\gamma$ such that $X=X_{\gamma} \in \cV$, we know that $\Psi_A(F)(\gamma) \simeq \gamma_{F(X)}$. We conclude that $\Psi_A(F)(\delta)$ generates another semiorthogonal decomposition $(\cU^{\prime \prime}, \cV^{\prime \prime})$ such that $\cV^{\prime \prime}\cap \Dfd{A}= \cV^{\prime} \cap \Dfd{A}$. Thus, it follows $\cU^{\prime \prime}={^{\perp}(\cV^{\prime \prime} \cap \Dfd{A})}={^{\perp}(\cV^{\prime} \cap \Dfd{A})}=\cU^{\prime}$ and hence $\cV^{\prime}=\cV^{\prime \prime}$ which in turn implies that $\Psi_A(F)(\delta)\simeq \gamma_{F(X_{\delta})}$.
\end{proof}

\noindent In closing this section we provide a preliminary yet simple characterisation of elements in the kernel of $\Psi_A$ in the rigid case.

\begin{lem}\label{LemmaRigidGentleKernel}
	Let $A$ be a rigid graded gentle algebra such that $A$ is proper or such that the projective arc system on $\Sigma_A$ consists of finite and semi-infinite arcs. Then $F \in \ker \Psi_{A}$ if and only if $F(C)\cong C$ for all direct summands of $C$ of $A$ in $\Perf(A)$.
\end{lem}
\begin{proof}
	If $F \in \ker \Psi_{A}$, then $F(X) \cong X$ for all indecomposable $X \in \cT_A$ such that $\gamma_X$ is homotopic to a simple arc and hence for all direct summands of $A$. For the converse suppose $F(C) \cong C$ for all such summands $C$ of $A$. Then \Cref{TheoremgeometrisationHomomorphism} and \Cref{TheoremGeometrisationHomomorphismSmoothCase} imply that $\Psi_A(F)(X_{\gamma}) \simeq X_{F(\gamma)}$ for all simple rigid semi-infinite arcs $\gamma$. Thus by rigidity of $A$, $\Psi_{A}(F)$ is a mapping class which acts as the identity on the full (resp.~finitely-full) projective arc system of $\Sigma_A$. The complement of such an arc system is a disjoint union of annuli and disks and hence the arc system forms an Alexander system in the sense of \cite[Section 2.3]{FarbMargalit}. The 'Alexander method' \cite[Proposition 2.8]{FarbMargalit} implies that the projection of $\Psi_{A}(F)$ into $\MCG(\Sigma)$ is trivial. The triviality of $\Psi_A(F)$ itself then follows again from  $F(C) \cong C$ for all summands of $A$.
\end{proof}

\begin{cor}
Let $A$ be a rigid graded gentle algebra satisfying the assumptions of \Cref{LemmaRigidGentleKernel}. Then $\ker \Psi^{\infty}_A \subseteq \Aut^{\infty}(A)$. 
\end{cor}
\begin{proof}
Follows from \Cref{LemmaImageOfInclusion} and \Cref{LemmaRigidGentleKernel}.
\end{proof}
\noindent The previous corollary reduces the study of the kernel of $\Psi^{\infty}_A$ to the study of a much smaller group of $A_\infty$-endofunctors of $A$ and eventually, as we shall see, to the group $\Aut_{\circ}^{\infty}(A)$ from Section \ref{SectionSpecialSubgroups}.

\section{Group actions and sections of geometrisation homomorphisms}\label{SectionSplit}

\noindent We recall the construction of a (weak) group action of the graded mapping class group of a graded marked surface on the associated Fukaya category. We then establish an analogous action of  on the bounded derived category of a proper graded gentle algebra. The main result of this section is the following.
	\begin{thm}\label{TheoremActionSplitsProjectionMaps}Let $A$ be a graded gentle algebra. The action homomorphism $\alpha\colon \MCG_{\gr}(\Sigma_A) \rightarrow \Aut^{\infty}(\cT_A)$ of the group action on $\cT_A$  by $\MCG_{\gr}(\Sigma_A)$  is a section of $\Psi_A^{\infty}$.
	\end{thm}
	\noindent In Section \ref{SectionDerivedPicardGroupWrappedFukayaCategory} we will extend the theorem to wrapped Fukaya categories of punctured surfaces. However, calculations in this case would be more challenging and instead we will reduce to the homologically smooth case. We record the following immediate consequence of \Cref{TheoremActionSplitsProjectionMaps}.
	\begin{cor}\label{CorollarySemiDirectProductPicardGroup}
		Let $A$ be a rigid graded gentle algebra satisfying the assumptions of \Cref{LemmaRigidGentleKernel}. Then, $\DPic(A) \cong \ker \Psi_A^{\infty} \rtimes \MCG_{\gr}(\Sigma_A)$. 
	\end{cor}
\begin{proof}
Follows from \Cref{LemmaRigidGentleKernel} and \Cref{TheoremActionSplitsProjectionMaps}.
\end{proof}
	
	\subsection{Actions by the mapping class group}\label{SectionGroupActionFukayaCategory} \ \medskip

	\subsubsection{From simply-connected categories to group actions}\label{SectionSimplyConnectedCategoriesGroupActions}\ \medskip

	\noindent We recall the relationship between functors from a simply-connected category with a group action and group actions on their colimits. A basic familiarity of the reader with nerves of categories and their homotopy groups is assumed.\medskip
	
	\noindent Let $I$ be a small category and let $D\colon I \rightarrow \cC$ be functor into an arbitrary category $\cC$. A \textit{strict} action of a group $G$ acts  $I$ is a group homomorphism $\alpha$ from $G$ to the group of automorphisms of the category $C$. If $G$ acts on $\cC$, then $D$ is said to be \textbf{invariant} under the $G$-action if $D=D \circ \alpha(g)$ for all $g \in G$. We abbreviate the action of an object $g \in G$ on an object or morphism $x$ in $\cC$ by $g.x$.
	\begin{lem}\label{LemmaGroupActionGeneralLemma} Let $D\colon  I \rightarrow \cC$ be a functor. Suppose that $D(f)$ is invertible for all morphisms $f$ in $I$ and that the nerve $\cN(I)$ of $I$ is simply connected. Then the following are true.
		\begin{enumerate}
			\item  The colimit $\colim D$ exists in $\cC$ and for each $i_0 \in I$, the structure map $D(i_0) \rightarrow \colim D$ is an isomorphism.
			\item If a group $G$ acts strictly on $\cC$ so that $D$ is invariant, then $G$ acts canonically on $\colim D \in \cC$.
		\end{enumerate}
	\end{lem}
	\begin{proof}
		Let $i, j \in I$. By connectedness of $\cN(I)$, there exists a zig-zag $\gamma= i \rightarrow i_1 \leftarrow i_2 \rightarrow \cdots \leftarrow j$ of morphisms in $I$ connecting $i$ with $j$. By assigning $D(u)$ to every morphism $x \xrightarrow{u} y$ and  $D(v)^{-1}$ to every morphism $x \xleftarrow{v} y$ in this sequence, $\gamma$  defines a canonical morphism $\varphi_{i}^j\colon D(i) \rightarrow D(j)$. By the same reasoning, the morphism $\varphi_{i}^j$ depends only on the homotopy class of $\gamma$ and hence is independent of the concrete choice since $\cN(I)$ is simply connected. By concatenating paths and using simply-connectedness once more, we derive $\varphi_{j}^k \circ \varphi_{i}^j=\varphi_{i}^k$ for all $i, j, k \in I$.
		
		Next, we observe that for each $i_0 \in I$, $D(i_0)$ is a colimit of $D$ as the maps $(\varphi_{i_0}^{j})_{j \in I}$ satisfy the universal property of $\colim D$. In particular, we also see that all the structure maps $\iota_i\colon D(i) \rightarrow \colim D$, are isomorphisms and $\iota_j^{-1}\iota_i=\varphi_i^j$ for all $i, j \in I$. Next, suppose $D$ is invariant under the action of a group $G$. For $(g, i) \in G \times I$, set $\beta_g \coloneqq \iota_{g.i} \circ \iota_i^{-1}$. As the notation suggests, the endomorphism $\beta_g$ of $\colim D$ is independent of $i$ which follows from the invariance of $D$ by the calculation
		\begin{displaymath}
			\iota_i^{-1} \circ \iota_j= \varphi_j^i=\varphi_{g.j}^{g.i} = \iota_{g.i}^{-1} \circ \iota_{g.j}.
		\end{displaymath}
		\noindent Hence, $\iota_{g.j} \circ \iota_{j}^{-1} = \iota_{g.i} \circ \iota_{i}^{-1}$ and the independence implies that the collection $(\beta_g)_{g \in G}$ gives rise to an action on $\colim D$. 
	\end{proof}
	\noindent As typical application of the previous lemma is that an object $X$ which is defined through the choice of an input datum, e.g.~a category, is in fact independent of our choice, up to suitable equivalence. Successful implementations of this idea can be found in several places of the literature, e.g.~\cite{HaidenKatzarkovKontsevich, DyckerhoffKapranov}. Moreover, if a group $G$ acts naturally on the set of input data, one may use Lemma \ref{LemmaGroupActionGeneralLemma} to establish the existence of a $G$-action on $X$. When $X$ is a category itself and $G$ acts on $X$ via endofunctors this allows to establish a $G$-action without the need to verify a potentially very complicated set of relations on functors.
	
	\begin{rem}\label{RemarkGeneralGroupAction}
		We note that in Lemma \ref{LemmaGroupActionGeneralLemma}, we have $\beta_g=\varphi_{i_0}^{g.i_0}$ under the isomorphism $\iota_{i_0}\colon  D(i_0) \rightarrow \colim D$.
	\end{rem}
	
	\subsubsection{The mapping class group action on the Fukaya category}\ \medskip
	
	\noindent We recall the construction of the mapping class group action on the topological Fukaya category of a surface whose existence was first proved in \cite{DyckerhoffKapranov}. We give a different proof here based on the work of \cite{HaidenKatzarkovKontsevich} which we use to give an explicit description of the action on indecomposable objects. We do not claim any originality to the proof of this fact which was included mainly for the reader's convenience and lack of suitable references for the material. By \Cref{prop: Morita invariance gentle algbebras}, there exists a functor
	\begin{displaymath}
		\begin{tikzcd}
			\cF\colon  \Arc{\Sigma, \eta} \arrow{r} & \Acat,
		\end{tikzcd}
	\end{displaymath}
	\noindent where  $\Arc{\Sigma,\eta}$ denotes the category of full graded arc systems on $\Sigma$. More precisely, we regard such arc systems  as tuples $(\Gamma, \iota, \mathfrak{g})$, where $\Gamma$ is a ribbon graph, $\iota\colon  \Gamma \rightarrow \Sigma$ is an inclusion and $\mathfrak{g}$ is choice of an $\eta$-grading on the image arcs. The set of morphisms from $\cA$ to $\cB$ in $\Arc{\Sigma,\eta}$ is a singleton set if $\cA \subseteq \cB$ (ignoring gradings) and empty otherwise. Likewise we denote by $\Arc{\Sigma}$ the analogous category of ungraded arc systems. The canonical forgetful map $\theta\colon \Arc{\Sigma, \eta} \rightarrow \Arc{\Sigma}$ induces a weak homotopy equivalence between their nerves. To see this note that $\Arc{\Sigma, \eta}$ is fibred and cofibred in groupoids over $\Arc{\Sigma}$ via $\theta$ which implies that $N(\theta)$ is a Kan fibration by \cite[Proposition 2.1.1.3]{HigherToposTheory}. Moreover, the fiber over any full arc system $\cA \subseteq \Sigma$ is the full subcategory corresponding to all the possible gradings on $\cA$ and is contractible. By \Cref{PropositionInclusionGivesMoritaEquivalences}, $\cF$ maps morphisms in $\Arc{\Sigma, \eta}$ to Morita equivalences and hence induces a functor
	\begin{equation}\label{EquationFukayaFunctor}
		\begin{tikzcd}
			\Fuk\colon  \Arc{\Sigma, \eta} \arrow{r} & \HAcat, \\
			\Gamma \arrow[mapsto]{r} & \Tw^+ \cF(\Gamma)
		\end{tikzcd}
	\end{equation}
	\noindent which maps all morphisms to isomorphisms. The colimit of $\cF$ is the topological Fukaya category of $\Sigma$ from  \Cref{DefinitionFukayaCategory}. The category $\Arc{\Sigma}$ is contractible by a famous result of Harer \cite{Harer} (see also \cite[Appendix A]{DyckerhoffKapranovCrossedSimplicialGroups}) and hence so is $\Arc{\Sigma, \eta}$. Finally, the mapping class group $\MCG(\Sigma,\eta)$ acts on $\Arc{\Sigma, \eta}$ via its canonical action on the pair $(\iota, \mathfrak{g})$.  With the help of \Cref{LemmaGroupActionGeneralLemma} one recovers the following result from \cite{DyckerhoffKapranov}.
	
	\begin{thm}[{\cite[Theorem 4.1.3]{DyckerhoffKapranov}}]\label{TheoremExistenceGroupAction}Let $\Sigma$ be a graded marked surface. Then $\MCG_{\gr}(\Sigma)$ acts on $\Fuk(\Sigma)$. 
	\end{thm}
\noindent In particular, if $A$ is a homologically smooth graded gentle algebra, then $\MCG_{\gr}(\Sigma_A)$ acts on $\Perf(A)=\cT_A$. In the next step we like extend the result to proper graded gentle algebras. First, we note that the mapping class group of the surface of $\Bbbk[x]/(x^2)$, $|x| \in \mathbb{Z}$, is trivial and hence we will exclude this case in the discussion that follows. 

Let $B$ be a proper graded gentle algebra and let $A=B^{\invex}$ be the corresponding homologically smooth graded gentle algebra. By \Cref{TheoremExistenceGroupAction}, $\MCG_{\gr}(\Sigma_B)=\MCG_{\gr}(\Sigma_{A})$ acts on $\Perf(A)$ and as we saw in the proof of \Cref{TheoremGeometrisationHomomorphismSmoothCase} this restricts to an action on $\Perf(B) \simeq \Perf(A^{!}) \subseteq \Perf(A)$. Because $A$ is reflexive by \Cref{TheoremProperGentleReflexive}, it therefore extends to an action on $\cT_B=\Dfd{B}$ which restricts to the action on $\Perf(B)$ we just constructed. By taking into account how we defined the map $\Psi^{\infty}_B$ we have therefore proved the following.
	\begin{cor}\label{CorollaryEquivalenceSplittingProperty} Let $B$ be a proper graded gentle algebra. Then there exist a group action by $\MCG_{\gr}(\Sigma_B)$ on $\Dfd{B}$ and $\Perf(B)$. Moreover, if $\alpha_{B}$ and $\alpha_{B^{\invex}}$ denote the corresponding action homomorphisms $\alpha_{(-)}\colon \MCG(\Sigma_{(-)}) \rightarrow \cT_{(-)}$, then $\Psi_B \circ \alpha_B = \Psi_{B^{\invex}} \circ \alpha_{B^{\invex}}$.
	\end{cor}
\noindent It follows in particular that the mapping class group action on $\Dfd{B}$ provides a section to $\Psi_B$ if the corresponding action on $\Perf(B^{\invex})$ provides a section to $\Psi_{B^{\invex}}$.

	\subsection{The functor associated to a pair of arc systems}\label{SectionFunctorPairOfArcSystems}\ \medskip

	\noindent If $\cA, \cB \subseteq \Sigma$ are full, formal arc systems on a graded marked surface $\Sigma$, we will use the notation $\varphi_{\cA}^{\cB}\colon  \Tw^+ \cF(\cA) \rightarrow \Tw^+ \cF(\cB)$ for the functor induced by \Cref{LemmaGroupActionGeneralLemma}. We compare these functors with explicit dg functors $T_{\cA}^{\cB}\colon  \Tw^+ \cF(\cA) \rightarrow \Tw^+ \cF(\cB)$ which we define in an ad-hoc manner for every pair of \textit{formal} arcs systems. To define them, we give a slightly generalised notion of a \textit{string complex} in the sense of \cite{Bobinski} and \textit{singleton maps} and \textit{graph maps} from \cite{ArnesenLakingPauksztello}. 
\subsubsection{String complexes and singleton maps}
For $\bA$ an $A_\infty$-category, let $\mathbb{Z}\bA$ denote the $\Bbbk$-quiver with objects $\sqcup_{n \in \mathbb{Z}}\Ob{\bA}[n]$ and $\mathbb{Z}\bA(X[n], Y[m]) \coloneqq \bA(X, Y)[m-n]$ for all $X, Y \in \Ob{\bA}$.
\begin{definition}
	Let $\bA$ be an $A_\infty$-category. A \textbf{string complex} over $\bA$ is an object $(X, \delta) \in \Tw^+(\bA)$ with the property that there exists a finite quiver $\cP$ of type $A$ (with any orientation of the arrows) and a map of $\Bbbk$-quivers $\chi\colon  \Bbbk\cP \rightarrow \mathbb{Z}\bA$ with the property that $X = \bigoplus_{p \in \cP_0}  \chi(p)$, $\delta=\sum_{\alpha \in \cP_1}\chi(\alpha)$ and such that for all paths $q=\alpha_1 \cdots \alpha_l$ in $\cP$, $\alpha_i \in \cP_1$, we have $\mu_{\bA}^l(\alpha_1, \dots, \alpha_l)=0$.
\end{definition}
\noindent Every string complex is uniquely encoded in a diagram of the form
\begin{displaymath}
	\begin{tikzcd}
		\bullet \arrow[dash]{r}{\chi_1} &\bullet \arrow[dash]{r}{\chi_2} & \cdots \arrow[dash]{r}{\chi_l} & \bullet,
	\end{tikzcd}
\end{displaymath}
\noindent whose edges are oriented, that is arrows, and where each node encodes an object in $\bA$ and an arrow with label $\chi_i$ encodes an element in morphism space between the objects corresponding to its start and end point.  We will also consider special maps between these objects. Every map $f\colon  X \rightarrow Y$ in $\Tw^+(\bA)$ between string complexes $X$ and $Y$ can be encoded in terms of a collection of arrows from vertices of $X$ to vertices of $Y$ which are labeled by maps between the corresponding objects of $\bA$. 

\begin{definition}Let $X$ and $Y$ be string complexes over $\mathbb{A}$. A \textbf{graph map} of degree $n \in \mathbb{Z}$ is a cohomology class $f \in \HH^0(\Hom_{\Tw^+(\bA)}^n(X,Y))$ which has a cocycle representative of the following form:
	\noindent if $X$ and $Y$ are modeled by pairs $(\cP_X, \chi_X)$ and $(\cP_Y, \chi_Y)$, then there exists a finite quiver $\mathcal{Q}$ of type $A$ which is embedded inside $\cP_X$ and $\cP_Y$ as a full subquiver via maps $\iota_X$ and $\iota_Y$ such that $f=\sum_{x \in \mathcal{Q}^0}f_x$, where 
	\begin{displaymath}
		f_x \in \Hom_{\Tw^+(\bA)}(\chi_X \circ \iota_X(x), \chi_Y \circ \iota_Y(x)\big)
	\end{displaymath}
	\noindent is a multiple of the identity for at least one vertex of $\mathcal{Q}$ and all inner vertices of $\cQ$, that is, vertices of $\cQ$ which have two distinct adjacent vertices. Similarly, a (boundary) \textbf{singleton single map} is defined in the same way by requiring that $\mathcal{Q}$ has a single vertex $z$, the only component $f_x$ is not invertible, and $\iota_X(z)$ and $\iota_Y(z)$ are both outer vertices of $P_X$ and $P_Y$ respectively, that is, they have at most adjacent vertex.
\end{definition}
\noindent Graph and singleton maps can represented diagrammatically as in \Cref{fig: graph map and singleton map}.
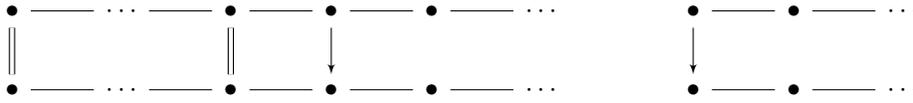
\begin{figure}
	\begin{displaymath}\arraycolsep=2em
		\begin{array}{lr}
			{
				\begin{tikzcd}
					\bullet \arrow[dash]{r}{} \arrow[equal]{d} &  \cdots \arrow[dash]{r}{} & \bullet \arrow[dash]{r}{} \arrow[equal]{d} & \bullet \arrow[dash]{r}{} \arrow{d} & \bullet \arrow[dash]{r}{} & \cdots \\
					\bullet \arrow[dash]{r}{}  &  \cdots \arrow[dash]{r}{} & \bullet \arrow[dash]{r}{}  & \bullet \arrow[dash]{r}{} & \bullet \arrow[dash]{r}{} & \cdots 
				\end{tikzcd}
			} 
			&
		{
		\begin{tikzcd}
			\bullet \arrow[dash]{r}{} \arrow{d}{} &\bullet \arrow[dash]{r}{} & \cdots   \\
			\bullet \arrow[dash]{r}{} &\bullet \arrow[dash]{r}{} & \cdots 
		\end{tikzcd}
	}
	\end{array}
	\end{displaymath}
\caption{A diagrammatic representation of a graph map (left) with identities indicated by double lines and a boundary singleton single map (right).} \label{fig: graph map and singleton map}
\end{figure}

\subsubsection{The functors} \ \medskip 

Given two formal arc systems $\cA$ and $\cB$ on $\Sigma$, every graded arc $\gamma \in \cA$ determines a unique sequence of graded arcs $\gamma_1, \dots, \gamma_l$ which are shifts of arcs in $\cB$ together with a unique sequence $f_1, \dots, f_{l-1}$ of non-constant flows $f_i \in \Flow(\gamma_i, \gamma_{i+1})$ such that $\gamma$ is homotopic to the concatenation of $\gamma_1, f_1, \gamma_2, f_2, \dots, f_{l-1}, \gamma_l$ in the given order and such the canonically induced grading on the concatenation agrees with the given one on $\gamma$ under homotopy. The above data determines a unique string complex $Y_{\gamma}$ in $\cF_{\cB}$  obtained by replacing $\gamma_i$ by its corresponding object in $\cF_{\cB}$ and regarding the morphisms associated to the flows $f_i$ in $\cF(\cA)$ as the components of the differential. It is not difficult to check that every $f \in \Flow(\gamma, \gamma')$, $\gamma' \in \cA$, determines a canonical morphism of twisted complexes $\hat{f}\colon Y_{\gamma} \rightarrow Y_{\gamma'}$ of the same degree. More precisely, if $\gamma_1 \neq \gamma_1'$, then $\alpha_f$ is simply the singleton single map with component induced by $f$. Otherwise, $\alpha_f$ is the canonical graph map whose overlap corresponds to the maximal index $i \geq 1$ with the property that $(\gamma_1, \dots, \gamma_i)=(\gamma_1', \dots, \gamma_i')$. The construction is compatible with composition in the sense that if $g \in \Flow(\gamma'',\gamma''')$ such that $f$ and $g$ are composable flows, then $\alpha_g \circ \alpha_f=\alpha_{g \ast f}$, where $g \ast f$ denotes the concatenation. In \cite{OpperPlamondonSchroll}, the reader can find more details on how one associates a morphism of complexes to a flow, or equivalently an intersection of curves at a marked point on the boundary.  The correspondence 
	\begin{displaymath}
		\begin{tikzcd}
			X_\gamma \arrow[mapsto]{r} & Y_{\gamma}
		\end{tikzcd}
	\end{displaymath}
	\noindent therefore extends canonically to a dg functor $\cF(\cA) \rightarrow \Tw^+ \cF(\cB)$ and hence a dg functor 
	\begin{displaymath}
		T_{\cA}^{\cB}\colon \Tw^+ \cF(\cA) \rightarrow \Tw^+ \cF(\cB)	
	\end{displaymath}
	\noindent after totalisation. The functors $T_{\cA}^{\cB}$ satisfy the following ``base change'' property.
	
	\begin{lem}\label{LemmaCompositonArcSystemChange} Let $(\cA, \cB, \cC)$ be a triple of proper formal arc systems. Then
		\begin{displaymath}
			T_{\cB}^{\cC}T_{\cA}^{\cB} \cong T_{\cA}^{\cC}.
		\end{displaymath}
	\end{lem}
	\begin{proof}
		It is sufficient to prove this in the case, where $\cB \cup \cC$ is an arc system and $\cB \sqcup \{\delta\}=\cB \cup \cC=\cC \sqcup \{\varepsilon\}$ for arcs $\delta \in \cC$ and $\varepsilon \in \cB$. It then follows by induction that every base change functor is the composition of such special ones and such associated to a change of grading. From this the general case follows immediately. We set $S \coloneqq T_{\cB}^{\cC} T_{\cA}^{\cB}$. Let $\gamma \in \cA$ and denote by $\gamma_1, \dots, \gamma_l \in \cB$ and $p_1, \dots, p_{l-1}$ the associated sequence of arcs and intersections appearing in the definition of $T_{\cA}^{\cB}$. Likewise we denote by $\delta_1, \dots, \delta_m \in \cC$ and $q_1, \dots, q_{m-1}$ the sequences for $\delta$ and by $Z_{\delta}$ the associated string complex.  The functor $S$ maps the object $X_{\gamma}$ to the twisted complex over $\cC$ which is obtained from $Y_{\gamma}$ in the following way. Every occurence of $\delta$ as a vertex $v$ of the underlying string diagram of $Y_{\gamma}$ is replaced by the string diagram of $Z_{\delta}$ and every component of the differental going in or out of $v$ is replaced by the corresponding component of the singleton single map or graph map which has a single component in the given situation. The result is a ``generalised'' string complex which potentially features subdiagrams of the form 
		
		\begin{equation}\label{EquationGeneralizedStringComplex}
			\begin{tikzcd}
				\text{} &\arrow[dash, dashed]{l}	\bullet & \bullet \arrow{l}{u} \arrow{r}{\operatorname{Id}} & \bullet  &  \arrow{l}{v} \bullet  \arrow[dashed, dash]{r} & \text{}
			\end{tikzcd}
		\end{equation}
		\noindent with edges labeled by identity maps. A simple calculation (similar to \cite[Lemma 2.1]{CanakciPauksztelloSchrollNewArxiv}) shows that such a string diagram $C_{\gamma}$ is isomorphic to the direct sum $C_{\gamma}^{\prime}$ of string complexes obtained by replacing each subdiagram of the form \eqref{EquationGeneralizedStringComplex} by the disconnected diagram
		
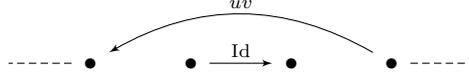
\begin{equation*}
			\begin{tikzcd}
				\text{} &\arrow[dash, dashed]{l}	\bullet & \bullet  \arrow{r}{\operatorname{Id}} & \bullet  &  \arrow[bend right, swap]{lll}{uv} \bullet  \arrow[dashed, dash]{r} & \text{}
			\end{tikzcd}
		\end{equation*}
	 \noindent Since such subdiagrams arise from a graph map with a single component, the path $uv$ is non-zero in $A$. Explicitly, $C_{\gamma}^{\prime}$ is obtained from $C_{\gamma}$ by conjugation of its differential with a base change $\Theta_{\gamma}\colon C_{\gamma} \rightarrow C_{\gamma}$. It is given by $\Theta_\gamma=\operatorname{Id}_{C_{\gamma}} + \theta$, where $\theta=\sum_D \theta_D$ is the sum of $\Bbbk$-linear endomorphisms $\theta_D\colon C_{\gamma} \rightarrow C_{\gamma}$ of the underlying graded vector space of $C_{\gamma}$. The sum is taken over all subdiagrams $D$ of the form \eqref{EquationGeneralizedStringComplex} and for every such $D$, the only two non-trivial components of $\theta_D$ are given by the two dotted arrows in \eqref{EquationBaseChangeComponents} inside of $D$.
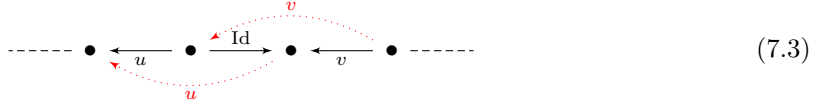
		\begin{equation}\label{EquationBaseChangeComponents}
			\begin{tikzcd}
				\text{} &\arrow[dash, dashed]{l}	\bullet & \bullet \arrow{l}{u} \arrow{r}{\operatorname{Id}} & \bullet \arrow[dotted, bend left, red]{ll}{u}  &  \arrow{l}{v} \bullet \arrow[dotted, bend right, red, swap]{ll}{v}  \arrow[dashed, dash]{r} & \text{}
			\end{tikzcd}
		\end{equation}
		\noindent Because $\theta_D \circ \theta_{D'}=0$ for all $D \neq D'$, the maps $\rho_D=\operatorname{Id}_{C_{\gamma}} + \theta_D$, commute pairwise and their composition agrees with $\Theta_{\gamma}$.
		
		We then observe that the only connected component of the underlying diagram of $C_{\gamma'}$ which does not represent an acyclic complex is the string diagram of $T_{\cA}^{\cC}(X_{\gamma})$. By pre- and postcomposing with the projection and inclusion maps associated with the identification of $T_{\cA}^{\cC}(X_{\gamma})$ as a direct summand of $C_{\gamma}^{\prime}$, the functor $T=T_{\cA}^{\cC}$ induces a dg functor $T' \cong T$ with $T'(X_{\gamma})=C_{\gamma}^{\prime}$ for all $\gamma \in \cA$. It is then not difficult to see that $S$ agrees with $T'$ after conjugation with $\Theta_{\gamma}$.
	\end{proof}
	
	\subsection{Explicit description of the mapping class group action and section property}\ \medskip
	
	\noindent We prove the following auxiliary result.
	\begin{prp}\label{PropositionExplicitDescriptionGroupAction}Let $\cA \subseteq \Sigma$ be a full formal graded arc system. For all $F \in \MCG_{\gr}(\Sigma, \eta)$, $\varphi_F \coloneqq \varphi_{\cA}^{F(\cA)}=T_{\cA}^{F(\cA)}$ in the Morita category.
	\end{prp}
	\begin{proof}
		As before there exists a sequence $\cA=\cA_0, \ldots, \cA_m=F(\cA)$ of proper formal graded arc systems such that for all $i \in [0,m)$, either $\cA_i$ and $\cA_{i+1}$ only differ by their grading, or $\cA_i \cup \cA_{i+1}$ is an arc system and $|\cA_i|+1=\cA_i \cup \cA_{i+1}=|\cA_{i+1}|+1$. According to Lemma \ref{LemmaCompositonArcSystemChange}, we have
		\begin{displaymath}
			T_{\cA}^{F(\cA)} \cong T_{\cA_{m-1}}^{\cA_m}\cdots  \circ T_{\cA_0}^{\cA_1}.
		\end{displaymath}
		\noindent Likewise and by construction, $\varphi_F$ is equal in the Morita category to  the composition of the functors $\varphi_{\cA_i}^{\cA_{i+1}}$ defined at the end of Section \ref{SectionSimplyConnectedCategoriesGroupActions}. Thus, it suffices to show that $\varphi_{\cB}^{\cB'}=T_{\cB}^{\cB'}$ whenever $\cB$ and $\cB'$ are graded proper formal arc systems such that $\cB \cup \cB'$ is an arc system and $\cB \sqcup \{\gamma'\}=\cB \cup \cB^{\prime}=\cB' \sqcup \{\gamma\}$ for arcs $\gamma \in \cB$ and $\gamma' \in \cB'$. As in \cite[Lemma 6.5]{BocklandtMirrorSymmetryPuncturedSurfaces} one proves that  $X_{\gamma'}$ is $A_\infty$-isomorphic in $\Tw^+ \cF_{\cB'}$ to $Y_{\gamma'} \in \Tw^+ \cF_{\cB}$. If $\iota_{\ast}\colon \cF(\ast) \hookrightarrow \cF(\cB \cup \cB')$, $\star \in \{\cB, \cB'\}$, denote the inclusions, then using this isomorphism between $X_{\gamma'}$ and $Y_{\gamma'}$,  one can build a weak equivalence between $\iota_{\cB}\approx \iota_{\cB'} \circ f\phi_B^{B'}$ and the restriction of $T_{\cB}^{\cB'}$ to $\cF(\cB) \subseteq \Tw^+(\cF(\cB))$. Because $\iota_{\cB'}$ is a Morita equivalence and hence invertible in $\Hmo$, we this provides us with the desired result.
	\end{proof}

	\noindent With the help of the previous proposition, we are prepared to prove the main result of this section.
\begin{proof}
		Let $\gamma \subseteq \Sigma$ be a graded essential arc and let $F \in \MCG_{\gr}(\Sigma, \eta)$. By Proposition \ref{PropositionExplicitDescriptionGroupAction}, $\varphi(F)(X_{\gamma}) \cong T_{\cA}^{F(\cA)}(X_{\gamma})=X_{F(\gamma)}$ and hence $\left(\Psi \circ \varphi\right)(F)=F$. This finishes the proof.
\end{proof}

\begin{proof}[Proof of Theorem \ref{TheoremActionSplitsProjectionMaps}] By Corollary \ref{CorollaryEquivalenceSplittingProperty}, we may restrict ourselves to the case when $A$ is a homologically smooth graded gentle algebra. Let $\alpha\colon \MCG_{\gr}(\Sigma_A) \rightarrow \DPic(A)$ denote the action homomorphism of the group action from Theorem \ref{TheoremExistenceGroupAction}. For any full, formal arc system $\cA \subseteq \Sigma_A$ and any mapping class $F \in \MCG_{\gr}(\Sigma_A)$, Proposition \ref{PropositionExplicitDescriptionGroupAction} shows that $\Psi_A \circ \alpha(F))\cong T_{\cA}^{F(\cA)})$ and hence $\alpha(F)(X_{\gamma})\cong X_{F(\gamma)}$ for all $\gamma \in \cA$. As in the proof of Lemma \ref{LemmaRigidGentleKernel}, the fact that $\cA$ is an Alexander system implies that $\Psi_A(\alpha(F))=F$. This finishes the proof. 
\end{proof}
	
	\section{Integration of Hochschild cohomology via exponential maps }\label{SectionIntegrationHochschild}
	\noindent We recall the integration map from Hochschild cohomology to the derived Picard group from the author's previous work \cite{OpperIntegration}.
	
	\subsection{The Hochschild exponential}\medskip
	
\noindent	As before, we denote by $\star$ the composition product on the Hochschild complex of an $A_\infty$-category on an $A_\infty$-category $\bA$, cf.~Section \ref{SectionBraceAlgebra}. We also recall the definition of the weight filtration $(W_iC(\bA))_{i \geq 0}$ on the Hochschild complex, cf.~\Cref{SectionBraceAlgebra}. Over a field of characteristic $0$, the composition product defines an exponential map 
	\begin{equation}\label{EquationPreLieExponential}
	\begin{tikzcd}[row sep=0.5em]
		\exp\colon W_2C^0(\bA) \arrow{r} & C^0(\bA), \\
		c \arrow[mapsto]{r} & \sum_{i=0}^{\infty} \frac{1}{i!}{c^i},
	\end{tikzcd}
	\end{equation}
	\noindent where $c^i \in W_{i+1}C^1(A)$ is inductively defined as $c^{i+1}=c^i \star c$ and $c^0\coloneqq \operatorname{Id}_{\bA}$. The assumption on $c$ guarantees that the sum converges in the topology on $C(A)$. The latter is a Banach space with a basis of neighbourhoods around $0$ given by the the weight filtration. The map $\exp$ is injective and its image is the set of \textit{group-like} elements
	\begin{displaymath}
		\operatorname{Id}_{\bA} + W_2C^0(\bA) = \{\operatorname{Id}_{\bA} + c \, | \, c \in W_2C^0(\bA)\}.
	\end{displaymath}
	\noindent The inverse $\log\colon \operatorname{Id}_{\bA} + W_2C^1(\bA) \rightarrow W_2C^1(\bA)$ is given by the pre-Magnus expansion, cf.~\cite{Ebrahimi-FardManchon}.
	
	\subsection{Exponential groups and integration}\medskip
	
	\noindent  The weight filtration  turns  $\cL \coloneqq W_2C^1(\bA)$ (and similarly $\HHH^1_+(\bA, \bA)=W_2\HHH^1(\bA, \bA)$) into a \textit{complete (dg) Lie algebra} in the sense that $F_i\cL \coloneqq W_{i+2}C^1(\bA)$ defines a descending filtration by subvector spaces such that $[F_i\cL, F_j\cL] \subseteq F_{i+j}\cL$ for all $i,j \geq 1$ and such that $$\cL \cong \lim_{n \geq 1} \quotient{\cL}{F_i\cL}$$ as dg Lie algebras. In particular, the Lie bracket is a continuous bilinear map $\cL \times \cL \rightarrow \cL$. Over a field of characteristic zero, any complete Lie algebra $\cL$ admits a strictly associative product, the \textbf{Baker-Campbell-Hausdorff product}
	\begin{equation}\label{EquationBCHProduct}
		\BCH{u}{v} \coloneqq  u+v + \sum_{i=1}^{\infty} {\frac{{(-1)}^{i-1}}{i} \sum_{\underline{m},\underline{n} \in \mathbb{N}^i\colon m_j+n_j > 0}{ \frac{1}{|\underline{m}+\underline{n}|} \frac{1}{\underline{m}! \underline{n}!} \cdot [u^{m_1}, v^{n_1}, u^{m_2}, \dots, u^{m_i}, v^{n_i}]},}
	\end{equation}
	\noindent for all $u, v \in \cL$.  Here, we used the following shorthand notation:
	\begin{itemize}
		\item For $\underline{p} \in \mathbb{N}^i$, $\underline{p}!=p_1!p_2! \cdots p_i!$ and $|\underline{p}|=\sum_{j=1}^i p_j$;
		\item $[u_1^{r_1}, u_2^{r_2}, \dots, u_i^{r_i}]=\ad_{u_1}^{r_1} \circ \ad_{u_2}^{r_2} \cdots \circ \ad_{u_{i-1}}^{a_{i-1}} \circ \ad_{u_i}^{r_i-1}(u_i)$.
	\end{itemize}
\noindent The pair $\exp(\cL) \coloneqq (\cL, \mathrm{BCH})$ forms a group with neutral element $0 \in \cL$ and is called the \textbf{exponential group} of $\cL$. Note that if $[u,v]=0$, then $\BCH{u}{v}=u+v$ recovers the usual addition and the classic formula $\exp(u+v)=\exp(u) \circ \exp(v)$ which shows that $-l$ is the inverse to $l \in \cL$. One of the main results of \cite{OpperIntegration} is the following.
	\begin{thm}[{\cite[Theorem A \& Section 6.5]{OpperIntegration}}]\label{TheoremIntegrationHochschildCohomology} Let $\bA$ be an $A_\infty$-category over a field of characteristic zero. Then, the map $\exp$ from \eqref{EquationPreLieExponential} descends to a group homomorphism
	\begin{displaymath}
		\begin{tikzcd}
			\exp_{\bA}\colon \big(\HHH^1_+(\bA, \bA), \operatorname{BCH}\big) \arrow{r}  & \DPic(\bA),
		\end{tikzcd}
	\end{displaymath}
\noindent with image $\Aut^{\infty}_+(\bA)$. If $\bA$ is a graded $\Bbbk$-linear category, then $\exp_{\bA}$ is injective.
	\end{thm}
\noindent Representatives of elements of $\DPic(\bA)$ are interpreted as elements in $C^0(\bA)$ by means of the Taylor coefficients of their associated $A_\infty$-functor. The injectivity result is a special instance of a general criterion and one of several big classes of $A_\infty$-categories with injective exponential, cf.~\cite[Proposition 6.15]{OpperIntegration} and its subsequent paragraph.

 As $[u,u]=0$ for all $u \in \HHH^1_+(\bA, \bA)$, every non-zero $u$ yields a $1$-parameter subgroup  $\lambda \mapsto \exp(\lambda u)$, $\lambda \in \Bbbk$, and hence a copy of $\mathbb{G}_a=(\Bbbk, +)$ inside $\DPic(\bA)$. More generally, if $u_1, \dots, u_n \in \HHH^1_+(\bA, \bA)$ commute pairwise, that is, $[u_i, u_j]=0$ for all $i,j=1,\dots, n$, the exponential determines an embedding 
\begin{equation}\label{eq: embedding additive groups}
\begin{tikzcd}
	\mathbb{G}_a^n \arrow[hookrightarrow]{r} &\DPic(\bA).
\end{tikzcd}
\end{equation}
\noindent This type of contribution appears in our discussion of derived Picard groups of graded gentle algebras. A second and more complicated contribution which we encounter is related to the Witt algebra.

\subsection{Exponential group of the Witt algebra}\ \medskip

\noindent The \textbf{complete positive Witt algebra} is the complete Lie algebra $(\overline{\cW}^+, [-,-])$ with underlying $\Bbbk$-vector space $\prod_{i \geq 1} \Bbbk w_n$ equipped with the filtration $F_j\overline{\cW}^+= \prod_{i \geq j} \Bbbk w_i$. Its Lie bracket is the unique continuous Lie bracket such that for all $m, n \geq 1$,
	\begin{displaymath}
		[w_m, w_n] = (m-n) w_{n+m}.
	\end{displaymath} 

\noindent As is well-known, $\overline{\cW}^+$ can be realised as a Lie algebra of differential operators on $\Bbbk \llbracket x \rrbracket$ with $\omega_i$ corresponding to the derivation $-x^{i+1} \frac{d}{d x}$. Given any derivation $D$ of $\Bbbk\llbracket x \rrbracket$, then because $\Bbbk\llbracket x \rrbracket$ is complete, we may define the exponential $\exp(D)$ by the usual power series with multiplication being interpreted as composition of power series. It is well-known that $\exp(D)$ is a continuous $\Bbbk$-linear  automorphism of $\Bbbk\llbracket x \rrbracket$ and $\exp$ induces an isomorphism between the exponential group of $\overline{\cW}^+$ and the subgroup $\Aut_1(\Bbbk \llbracket x \rrbracket)$ of continuous automorphisms $\phi$ of $\Bbbk \llbracket x \rrbracket$ with $\phi(x) \in x + (x^2) \in \Bbbk \llbracket x \rrbracket$. We note that $\phi$ is uniquely determined by $\phi(x)$ which we may write as $\phi(x)=x\cdot \zeta_{\phi}$ for some unit $\zeta_{\phi} \in 1 + x\Bbbk \llbracket x \rrbracket \subseteq {(\Bbbk \llbracket x \rrbracket)}^{\times}$. This identifies the \emph{set} of such continuous automorphisms with the set $1+x\Bbbk \llbracket x \rrbracket$. We note however that the group structure on the two sets are \emph{different}. Since $\Bbbk$ is a (discrete) field, every algebra automorphism of $\Bbbk\llbracket x \rrbracket$ is continuous which implies the following.
\begin{cor}\label{CorollaryExponentialGroupWittAlgebraWittVectors}
	Assume that $\operatorname{char} \Bbbk=0$. Then, there exist an isomorphism
	\begin{displaymath}
		\mathbb{W}^+\coloneqq \exp\big(\overline{\cW}^+\big) \cong \Aut_1\!\big(\Bbbk\llbracket x\rrbracket\big).
	\end{displaymath}
\end{cor}

		\section{Hochschild cohomology of graded gentle algebras}\label{SectionLieAlgebraStructureGradedGentle}
		\noindent We present some of the results on the Hochschild cohomology of a graded gentle algebra in terms of the boundary of its surface. The computations of the entire Hochschild cohomology along with applications will appear in \cite{OpperHochschildCohomologyGentle}. Throughout the entire section we fix a graded gentle quiver $(Q,I)$ along with its associated graded gentle algebra $A$ which is assumed to be smooth or proper.
		
For $d \in \mathbb{Z}$ and $n \geq 0$, we denote by $\phi_A(n,d)$ the number of boundary components $B$ of $\Sigma_A$ with exactly $n$ boundary segments and winding number $\omega\coloneqq n-d$. In particular, $\phi_A(0,0)$ is the number of fully marked components with vanishing winding number and $\phi_A(1,0)$ describes the number of degenerating boundary components. If $A$ is proper, this is the  Avella-Alaminos-Geiss invariant $\phi_A\colon \mathbb{N} \times \mathbb{Z} \rightarrow \mathbb{N}$ from \cite{AvellaAlaminosGeiss}, and for $A$ homologically smooth, its generalisation from \cite{LekiliPolishchukGentle}. For proper $A$ and from a categorical perspective $\phi_A(m,n)$ counts isomorphism classes up to shift of special collections of $(m,m-n)$-fractional Calabi-Yau objects, that is, objects $X \in \Perf(A)$ such that $\mathbb{S}^mX \cong X[m-n]$, where $\mathbb{S}$ denotes a Serre functor. This interpretation of $\phi_A$ is found originally in \cite{AvellaAlaminosGeiss} in the ungraded case and follows from \cite[Proposition 2.16]{OpperDerivedEquivalences} for arbitrary gradings. Next, we describe the first Hochschild cohomology of $A$. We use the notation $\cW^+$ for the Lie subalgebra $\bigoplus_{i \geq 1} \Bbbk w_n \subseteq \prod_{i \geq 1} \Bbbk w_n=\overline{\cW}^+$, that is, the positive portion of the Witt algebra.
\begin{prp}[{\cite{OpperHochschildCohomologyGentle}}]\label{CorollaryLieAlgebraStructureFirstHochschild}
Let $A$ be a graded gentle algebra over a field of characteristic different from $2$. Then there exists an isomorphism of Lie algebras 
	\begin{displaymath}
	\HHH^1(A,A) \cong \left(\Bbbk^{\phi_A(1,1)} \times  {\cL}^{\phi_A(0,0)}\right) \rtimes \HH_1(\Sigma_A, \Bbbk),
\end{displaymath}
	\noindent where $\cL=\cW^+$ if  $A$ is homologically smooth and $\cL=\overline{\cW}^+$ if $A$ is proper, and where $\Bbbk$ and $\HH_1(\Sigma_A, \Bbbk)$ are abelian Lie algebras. All copies of $\overline{\cW}^+$ lie in $\HHH^1_+(A,A)$ but neither do $\HH_1(\Sigma_A, \Bbbk)$ nor do the copies of $\cW^+$.
\end{prp}

\begin{rem}The image of $\HH_1(\Sigma_A, \Bbbk)$ inside $\HHH^1(A,A)$ from \Cref{CorollaryLieAlgebraStructureFirstHochschild} is related to the fundamental group $\pi_1(A)$ of $A$ in the sense of \cite{MartinezVillaDeLaPenaFundamentalGroup} which can be generalised to graded algebras in the natural way. In general, the fundamental group depends on a chosen presentation of an algebra but is an invariant of the isomorphism class for monomial algebras in which case it is free abelian. There is a natural identification $\pi_1(A) \cong \pi_1(\Sigma_A)$ and following \cite{DeLaPenaSaorin}, for any ordinary algebra $B$ there is an explicit embedding $\Hom(\pi_1(B), \Bbbk) \hookrightarrow \HHH^1(B,B)$. This can be generalised here in a straightforward way and interpreted via the bigraded Hochschild cohomology of $A$. The work \cite{BriggsRubioYDegrassiMaximalTori} shows that if $B$ is monomial, then the rank of $\pi_1(B)$ agrees with the rank of any maximal torus in $\HHH^{1}(B, B)$. If $\Bbbk$ is algebraically closed, then every maximal torus in $\HHH^1(B,B)$ corresponds to a maximal torus in the outer automorphism group of $B$. For $A$ the situation is more complicated due to the fact that the outer automorphism group of a graded algebra is no longer a derived invariant \cite[Example 4.7]{OpperIntegration}.
\end{rem}
		
		\section{The kernel of the geometrisation homomorphism}\label{SectionKernel}
		\noindent We generalise the results from \cite{OpperDerivedEquivalences} and give a description of the kernel of 
		\begin{displaymath}
		\Psi^{\infty}_A\colon \DPic(A) \rightarrow \MCG_{\gr}(\Sigma)
		\end{displaymath}
		\noindent for any graded gentle algebra $A$ which is homologically smooth or proper. A key ingredient in characteristic zero is the Hochschild exponential from \Cref{SectionIntegrationHochschild}. In positive characteristic, we exploit a formality result for the Hochschild complex to describe the kernel under mild assumptions. 
		\subsection{Tori in the derived Picard group and action by invertible local systems}\label{SectionToriDerivedPicardGroups}\ \medskip
		
		\noindent As before we fix a graded gentle quiver $(Q,I)$, its associated gentle algebra $A$ and $\Sigma=\Sigma_A$ its graded marked surface.  We describe the subgroup of $\Out(A)$ corresponding to the Lie subalgebra $\HH_1(\Sigma_A, \Bbbk) \subseteq \HHH^1(A,A)$. Because $I \subseteq \Bbbk Q$ is generated by monomial relations, every tuple $\Lambda=(\lambda_{\alpha})_{\alpha \in Q_1} \in {(\Bbbk^{\times})}^{Q_1}$ determines an outer automorphism $\cO(\Lambda) \in \Out(A)$ which maps each vertex idempotent to itself and every arrow $\alpha \in Q_1$ to $\lambda_{\alpha} \alpha$. As in \cite{OpperDerivedEquivalences} we refer to such outer automorphism as \textbf{rescaling automorphisms}. They form a subgroup $\cR \subseteq \Out(A)$ which by \cite[Lemma 8.3.]{OpperDerivedEquivalences} fits into a short exact sequence of groups
		\begin{equation}\label{SesRescalingEquivalences}
			\begin{tikzcd}\mathbf{1} \arrow{r} & \Bbbk^{\times} \arrow{r}{\Delta} & (\Bbbk^{\times})^{Q_0} \arrow{r} & (\Bbbk^{\times})^{Q_1} \arrow{r}{\mathcal{O}} & \mathcal{R} \arrow{r} & \mathbf{1}.\end{tikzcd} 
		\end{equation}
		\noindent with $\Delta$ denoting the diagonal embedding. The middle homomorphism is given by ${(\mu_x)}_{x \in Q_0} \mapsto (\mu_{t(\alpha)} \mu_{s(\alpha)}^{-1})_{\alpha \in Q_1}$. In particular, we have the following.
		\begin{prp}[{\cite[Proposition 8.4.]{OpperDerivedEquivalences}}]
			There exists a group isomorphism $\cR \cong \HH^1(\Sigma_{\reg}, \Bbbk^{\times})$.
		\end{prp}
		\noindent Equivalently, elements of $\cR$  can be realised as an action by invertible local systems on $\Sigma_A$, that is, the invertible objects in the symmetric monoidal category of local systems. Because these are exactly the $1$-dimensional local systems, one obtains a string of  isomorphisms 
		\begin{displaymath}
			\Pic(\Loc_{\Sigma_A})\cong\Hom_{\mathbb{Z}}(\pi_1(\Sigma_A, x), \Bbbk^{\times})\cong \HH^1(\Sigma_A, \Bbbk^{\times}),
		\end{displaymath}
		\noindent where $\Pic(\Loc_{\Sigma_{A}})$ denotes the Picard group of $\Loc_{\Sigma_A}$. To make this explicit, every isomorphism class $[\cL] \in \Pic(\Loc_{\Sigma_A})$ induces a canonical element $R_{[\cL]} \in \cR$ as follows. Up to isomorphism we may and will assume that $\cL(x)=\Bbbk$ for all $x \in \Sigma_A$ and that $\cL(\gamma)=\operatorname{Id}_{\Bbbk}$ for all arcs $\gamma$ in a fixed projective arc system. This allows us to identify $\cL(f)$ of a flow $f$ with a non-zero scalar and by definition, $\cR_{[\cL]}$ maps an arrow $\alpha \in Q_1$ to $\cL(f_{\alpha}) \cdot \alpha$, where $f_{\alpha}$ denotes the flow which is associated with $\alpha$. The following proposition sums up our previous discussion. 	
		\begin{prp}
			Let $A$ be a graded gentle algebra. The assignment $[\cL] \mapsto R_{[\cL]}$ induces an isomorphism
			\begin{displaymath}
				\begin{tikzcd}
					\HH^1(\Sigma_A, \Bbbk^{\times})\cong \Pic(\Loc_{\Sigma_A}) \arrow{r}{\sim} & \cR \subseteq \Out(A) \subseteq \Aut^{\infty}(A).
				\end{tikzcd}
			\end{displaymath}
		\end{prp}
\begin{rem}\label{RemarkLocalSystemsPuncturedCase}The same definitions still make sense for the case when $A$ is the $A_\infty$-algebra (or category) associated to a full arc system on any marked surface $\Sigma$, including punctured surfaces. It is straightforward to see that $R_{[\mathcal{L}]}$ still defines a strict $A_\infty$-functor of $A$ in these cases. Indeed, if $\alpha_1, \dots, \alpha_n$ is a disc sequence and $R_{[\mathcal{L}]}(\alpha_i)=\lambda_i \alpha_i$, then $\prod_{i=1}^n \lambda_i=1$ because $\mathcal{L}$ maps contractible curves to the identity and hence $\cR_{[\cL]}$ respects the defining $A_\infty$-relations.
\end{rem}

		\subsection{The kernel in characteristic zero via the Hochschild exponential}\ \medskip
		
		\begin{lem}\label{LemmaImageIntegrableClasses}
			Let $A$ be a rigid graded gentle algebra
			If $A$ is proper, then $\Aut^{\infty}_{\circ}(A) \subseteq \ker \Psi^{\infty}$. If $A$ is homologically smooth and its projective arc systems on $\Sigma_A$ consists of finite and semi-infinite arcs, then $\Aut_+^{\infty}(A) \subseteq \ker \Psi^{\infty}$.
		\end{lem}
		\begin{proof}
			By Lemma \ref{LemmaRigidGentleKernel}, $F \in \DPic(A)$ lies in $\ker \Psi^{\infty}$ if and only if $\gamma_{F(P)} \simeq \gamma_P$ for all indecomposable direct summands $P$ of $A$. In particular, $\Aut^{\infty}_+(A) \subseteq \ker \Psi^{\infty}$. In the proper case, the projection of $A$ onto its radical quotient induces a map of algebraic groups $\pi\colon\Out(A) \rightarrow \Out(A/\rad A)=\Aut_{\gr}(A/\rad A)$. The latter is discrete and isomorphic to the symmetric group over $Q_0$. Thus, $\pi(\OutO(A))=\{\operatorname{Id}\}$ and hence for every $f \in \OutO(A)$, $\Psi^{\infty}(f)$ preserves the homotopy classes of every projective arc which proves $\OutO(A) \subseteq \ker \Psi^{\infty}$. 
		\end{proof}

		\begin{thm}\label{TheoremIntersectionPicardGroupKernel} Assume that $\operatorname{char}\Bbbk=0$ and let $A$ be a rigid graded gentle algebra which is proper or such that the projective arc system on $\Sigma_A$ consists of finite and semi-infinite arcs.
			\noindent If $A$ is derived equivalent to the path algebra of the Kronecker quiver, then  $\ker\Psi^{\infty}_A \cong \PGL_2(\Bbbk)$. If not and
			\begin{enumerate}
				\item $A$ is proper, then
				\begin{displaymath}
				\ker \Psi^{\infty}_A  \cong \Aut^{\infty}_{\circ}(A) \cong \Big(\mathbb{G}_a^{\phi_A(1,1)} \times  {\mathrm{Aut}_1(\Bbbk\llbracket x \rrbracket)}^{\phi_A(0,0)} \Big) \rtimes \HH^1(\Sigma_A, \Bbbk^{\times}),
				\end{displaymath}
					\item $A$ is homologically smooth, then 
						\begin{displaymath}
					\ker \Psi^{\infty}_A  \cong \Big(\mathbb{G}_a^{\phi_A(1,1)} \times  {\mathbb{G}_a}^{\phi_A(0,0)} \Big) \rtimes \HH^1(\Sigma_A, \Bbbk^{\times}),
					\end{displaymath}
			\end{enumerate}
			\noindent where $\phi_A$ denotes the Avella-Alaminos-Geiss invariant.
		\end{thm}
		\begin{proof}Let $F \in \ker \Psi^{\infty}_A$. Then because $A$ is rigid, it follows from \Cref{TheoremgeometrisationHomomorphism} and \Cref{TheoremGeometrisationHomomorphismSmoothCase}, that $F$ maps each indecomposable direct summand of $A$ to itself up to isomorphism.
			Because $F \circ (F^1)^{-1} \in \Aut_{+}^{\infty}(A) \subseteq \ker \Psi^{\infty}$, it follows that $F^1 \in \ker \Psi^{\infty}$. Again, because $A$ is rigid, it now follows in the same way as in the proof of \cite[Lemma 5.9]{OpperDerivedEquivalences} that every element in $\ker \Psi^{\infty} \cap \Out(A)$ is a composition of an element of $\cR$ and an algebra automorphism of the form $\operatorname{Id} + h$ which corresponds to a copy of $\Bbbk$ in \Cref{CorollaryLieAlgebraStructureFirstHochschild}. Such automorphism were called \textit{affine coordinate transformations} in \cite{OpperDerivedEquivalences}. If $A$ is proper, this proves $F^1 \in \OutO(A)$. By composing $F$ with the inverse of $F^1$, we may reduce to the case $F \in \Aut^{\infty}_+(A)$. Since $\operatorname{char} \Bbbk=0$, Theorem \ref{TheoremIntegrationHochschildCohomology} implies that $\Aut^{\infty}_+(A)$ is isomorphic to the exponential group of $\HHH^1_+(A,A)$. The structure of the latter is a consequence of \Cref{CorollaryExponentialGroupWittAlgebraWittVectors} and \Cref{CorollaryLieAlgebraStructureFirstHochschild}. It follows from \Cref{TheoremgeometrisationHomomorphism} that every rigid graded gentle algebra which is derived equivalent to the Kronecker algebra $\cK$ is isomorphic to it. Hence the result is already contained in \cite{OpperDerivedEquivalences} and  $\PGL_2(\Bbbk)=\Out(\cK)=\ker \Psi_{\cK}^{\infty}$. 
		\end{proof}

	\begin{cor}\label{cor: Theorem A}
	Assume that $\operatorname{char} \Bbbk=0$ and let $A$ be a graded gentle algebra which is homologically smooth or proper. If $A$ is derived equivalent to the path algebra of the Kronecker quiver, then $\DPic(A) \cong \PGL_2(\Bbbk) \rtimes \MCG_{\gr}(\Sigma_A)$. If not and
	\begin{enumerate}
		\item $A$ is proper, then $$\DPic(A) \cong \Big(\big(\mathbb{G}_a^{\phi_A(1,1)} \times  {\mathrm{Aut}_1(\Bbbk\llbracket x \rrbracket)}^{\phi_A(0,0)} \big) \rtimes \HH^1(\Sigma_A, \Bbbk^{\times})\Big) \rtimes \MCG_{\gr}(\Sigma_A),$$
		\item $A$ is homologically smooth, then $$\DPic(A) \cong \Big(\big(\mathbb{G}_a^{\phi_A(1,1)} \times  {\mathbb{G}_a}^{\phi_A(0,0)} \big) \rtimes \HH^1(\Sigma_A, \Bbbk^{\times})\Big) \rtimes \MCG_{\gr}(\Sigma_A),$$
	\end{enumerate}
\noindent where $\phi_A$ denotes the Avella-Alaminos-Geiss invariant.
	\end{cor}
\begin{proof}
Follows from \Cref{PropositionGentleRigidity}, \Cref{CorollarySemiDirectProductPicardGroup} and the fact that every graded gentle algebra is Morita equivalent to a graded gentle algebra satisfying the assumptions of \Cref{TheoremIntersectionPicardGroupKernel}. If $A$ is proper, this is \Cref{PropositionGentleRigidity}. If $A$ is homologically smooth, we may replace it by virtue of \Cref{prop: Morita invariance gentle algbebras} with a derived equivalent rigid graded gentle algebra which has a projective arc system consisting of finite and semi-infinite arcs. To find such an algebra, one inductively constructs an arc system on $\Sigma_A$ with the desired properties.
\end{proof}
	
	\noindent In the subsequent section, we prove an analogue of the previous theorem in positive characteristic. Due to the absence of the exponential map, we exploit a different feature, namely the formality of the dg algebra structure of the Hochschild complex.

	\begin{rem}
		Because the Kronecker algebra $\cK$ is concentrated in degree $0$, $\Aut_{+}^{\infty}(\cK)=0$.  The surface of the Kronecker quiver is the only surface where distinct boundary components are homotopic as curves and are gradable.
	\end{rem}		
		
		\section{Formality and derived Picard groups in positive characteristic}\label{SectionFormalityHochschildComplex}
		\noindent In a large number of cases we show that the dg algebra structure on Hochschild complex of a graded gentle algebra is formal. We then use this result to show that the group $\Aut^{\infty}_+(A)$ in those cases is independent of the characteristic of the underlying field. Ultimately, this will lead to the proof of Theorem \ref{IntroThmA} in positive characteristic. We fix a graded gentle quiver $(Q,I)$ and denote by $A$ its associated graded gentle algebra.
		\subsection{Formality of the Hochschild dg algebra}
		\noindent We recall that an \textit{anti-path} (resp.~\textit{path}) in $A$ is a path $p=\alpha_1 \cdots \alpha_l$, $\alpha_i \in Q_1$ such that $\alpha_i \alpha_{i+1} \in I$ (resp.~$\alpha_i \alpha_{i+1} \not \in I$) for all $1 \leq i < l$. Moreover, $p$ is \textit{maximal} if it is not properly contained in another anti-path (resp.~path) in $A$. If $p$ is a maximal anti-path, then by the types of relations in a gentle quiver, there exists at most one parallel path $\para{p}$, that is, $(s(p), t(p))=(s(\para{p}, t(\para{p}))$, such that $\para{p} \not \in I$ and such that $\para{p}$ and $p$ do not share the same first arrow and do not share the same last arrow. If $\para{p}$ does not exist, we indicate this by writing $\para{p}=\emptyset$. If $\para{p}$ exists, then it is automatically maximal due to the gentle relations and the maximality of $p$.
		
		\begin{prp}\label{PropositionFormalityHochschildComplex}
			Suppose that $A$ is a homologically smooth and proper graded gentle algebra. Then the Hochschild dg algebra of $A$ is formal.
		\end{prp}
		\begin{proof} It is shown in \cite{OpperHochschildCohomologyGentle} shows that under the given assumptions $\HHH^{\bullet}(A,A)$ has a basis over $\Bbbk$ consisting of elements $f_p$ and $f_q$ of the following form. For $p=\alpha_1 \cdots \alpha_l$ a maximal anti-path of length $l$ such that $\para{p}\neq \emptyset$, we define an element $f_p \in \HHH^{\omega}(A,A)$, where $\omega=|\para{p}|-|p|+l$ so that $f_p$ has a single non-zero component function $f\colon A^{\otimes l} \rightarrow A$, which maps an element  $u_1 \otimes \cdots \otimes u_l$ composed of paths $u_i$ in $Q$ with $u_i \not \in I$, to $\para{p}$ if $u_i=\alpha_i$ for all $1 \leq i \leq l$ and to zero otherwise. Likewise, if $q$ is a closed maximal path of $A$, the constant function associated to $q$ is a Hochschild cocycle. Given now any two (possibly identical) maximal anti-paths $p, p'$ and closed maximal paths $q, q'$, the chain level cup products $f_{\star} \cup f_{\ast}$, $\star, \ast \in \{p, p', q, q'\}$ is always zero if $\{\star, \ast\}=\{p, p'\}$ or $\{\star, \ast\}=\{q, q'\}$ as it involves the multiplication of two maximal paths or anti-paths. In particular, formality of the Hochschild dg algebra follows if $\Sigma_A$ has a single marked point. If it has more than one marked interval, then up to Morita equivalence and by Morita invariance of the Hochschild dg algebra, we may assume that the projective arc system of $\Sigma_A$ contains no closed arc, and hence that $A$ does not have a closed maximal path. Under this assumption, it follows that the inclusion $\HHH^{\bullet}(A,A) \hookrightarrow C(A)$ via the chosen basis (consisting only of elements of type $f_p$) is a quasi-isomorphism of dg algebras.
		\end{proof}
		\noindent We expect formality beyond the discussed cases, however, a proof seems to require more refined arguments and we hope to investigate this problem in the future. While it might often be possible to guess a certain number of elements in the group $\Aut^{\infty}_+(B)$ of a graded algebra $B$, it is difficult to verify whether one has found all of them. In cases where the Hochschild dg algebra of $B$ is formal we provide an easily computable upper bound for $\Aut^{\infty}_+(B)$ in  \Cref{prop:formalityMaurerCartan} below using deformation theory.
		
		\subsection{Maurer-Cartan theory for dg algebras}\ \medskip
		
		\noindent Let us briefly recall the relevant portion of Maurer-Cartan theory for dg algebras. In what follows, we will always assume that $E$ is a non-unital dg algebra which admits a complete descending filtration $E=F_1E \supseteq F_2E \supseteq \cdots$, compatible with the differential and the multiplication. For any dg algebra $B$, the dg algebra $W_1C(B)[-1]$ is an example. The elements $E^0$ form a group with multiplication $e \circledast f \coloneqq e + f + ef$. For each $e \in E^0$, $$e^{-1}=1+\sum_{i=0}^{\infty}(-1)^i e^i,$$
		\noindent where the infinite sum is understood as the limit of its finite partial sums in the metric topology induced by the filtration. By regarding  $E$ as a subalgebra of the unital dg algebra $E^+=E \oplus \Bbbk \mathbf{1}$ with the obvious differential and multiplication, $\circledast$ is equivalently defined as the unique group operation such that the map $E^0\cong 1+E^0 \subseteq {(E^+)}^{\times}$, $e \mapsto 1+e$ becomes a group homomorphism.\medskip
		
	\noindent The set of \text{Maurer-Cartan elements} in $E$ is
		\begin{displaymath}
			\MC(E) \coloneqq \{\zeta \in E^1 \mid d_E(\zeta) +\zeta^2=0\}.
		\end{displaymath} 
	\noindent It is acted upon by the group $(E^0, \circledast)$ with $x=1+e \in 1+ E^0$ acting  via the formula
	$$x.\zeta \coloneqq x \zeta x^{-1} - d_E(x)x^{-1} \in E^1.$$
	\noindent Then $\zeta, \xi \in \MC(E)$ are said to be \textbf{gauge equivalent} (denoted by $\sim$) if they lie in the same orbit under this action and the set of gauge equivalence classes in $\MC(E)$ is denoted by $\overline{\MC}(E)$. We note that if $E$ is graded commutative, $d_E=0$ and $\operatorname{char} \Bbbk\neq 2$ (as we assume), e.g.~$E=\HHH_+^1(A,A)$, then $E^1\cong \overline{\MC}(E)=\MC(E)$. Any filtration of $E$ induces a filtration on $\overline{\MC}(E)$: if $F \subseteq E$ is a step of the filtration, then $$\{\zeta \in \overline{\MC}(E) \mid \exists\zeta' \in F^1\colon \zeta \sim  \zeta'\},$$
	is the corresponding filtration step on $\overline{\MC}(E)$. Finally, we denote by $\overline{\MC}_+(C(B)) \subseteq \overline{\MC}(W_1C(B))$ the subset associated to the filtration $0 \subseteq C_+(B) \subseteq W_1C(B)$. It then follows from  \cite[Lemma 6.10]{OpperIntegration} that $\overline{\MC}_+(W_1C(B))$ is in bijection with the group of \textit{homotopy} classes of $A_\infty$-isotopies and as such, there exists a surjection $\overline{\MC}_+(C(B)) \twoheadrightarrow \Aut^{\infty}_+(B)$. Moreover, if $B$ is a graded algebra, then $\Aut^{\infty}_+(B)=\Aut^{h}_+(B)$ by \cite[Proposition 6.15]{OpperIntegration} and hence $\overline{\MC}_+(W_1C(B))\cong \Aut^{\infty}_+(B)$.

	\subsection{Bounding $\Aut_+^{\infty}(B)$ through formality}\ \medskip
	
\noindent The next proposition provides an upper bound for the group $\Aut^{\infty}_+(B)$ in terms of Maurer-Cartan sets.
		\begin{prp}\label{prop:formalityMaurerCartan}
			Let $B$ be a graded algebra. If the Hochschild dg algebra of $B$ is formal, then there exists a bijection
			\begin{displaymath}
				\begin{tikzcd}
			\overline{\MC}\big(\HHH^{1}_+(B,B)\big) \arrow{r}{\sim} & \overline{\MC}_+(W_1C(B))\cong\Aut^{\infty}_+(B).
	\end{tikzcd}	
	\end{displaymath}
			\noindent In particular, if $\operatorname{char} \Bbbk \neq 2$, there exists a bijection
			\begin{displaymath}
			\begin{tikzcd}
			\HHH^1_+(B,B) \cong \Aut^{\infty}_+(B).
			\end{tikzcd}
			\end{displaymath}

		\end{prp}
		\begin{proof}Set $C_+=C_+(B)[-1]$, $W_1C=W_1C(B)[-1]$ so that $W_1C$ is a dg algebra and $C_+$ a dg subalgebra of $C=C(B)[-1]$. We note that because $B$ has trivial differential, for every $d \in \mathbb{Z}$, the subspaces $\Hom^{n+d}(B^{\otimes n}, B)$, $n \geq 0$, form a subcomplex $Z_d$ of $C$ and because cohomology commutes with products, one has $\HH^{\bullet}(C)\cong \prod_{d \in \mathbb{Z}}\HH^{\bullet}(Z_d)$. Moreover $Z_d \cup Z_{d'} \subseteq Z_{d+d'}$ for all $d, d' \in \mathbb{Z}$. Let $(U_iC)_{i \geq 1}$ denote the subspace of $C$ defined by
	
	\begin{displaymath}
	(U_i)^j \coloneqq \prod_{n \geq \max\{i+j, 0\}} \Hom_{\Bbbk}^j({(B[1])}^{\otimes n}, B).
\end{displaymath}
			\noindent In other words, $U_i \coloneqq \prod_{d \geq i} Z_d$. Following the above discussion we see that $U_1C$ is a dg subalgebra of $C$ and $(U_i)_{i \geq 1}$ defines a complete descending filtration of $U_1C$. We observe that $U_1^1=C_+^1$ and  $U_1^0={(W_1C)}^0$ which implies $\overline{\MC}(U_1C)=\overline{\MC}_+(C)$. Next, we take the analogous filtration $(V_i)_{i \geq 1}$ of the cohomology $\HHH^{\bullet}(B,B)$ defined by
			\begin{displaymath}
			V_i^j \coloneqq \prod_{n \geq \max\{i+j, 0\}} \HHH^{n, j}(B,B).
			\end{displaymath}
			\noindent Then $V_1 \subseteq \HHH^{\bullet}(B,B)$ is a complete dg algebra and $\HH^{\bullet}(U_i)=V_i$ for all $i \geq 1$. Moreover, $V_1^1=\HHH^1_+(B,B)$ and hence
			\begin{displaymath}
			\overline{\MC}\big(\HHH^1_+(B,B)\big)\cong\MC\big(\HHH^1_+(B,B)\big) \cong \MC\big(V_1\big) \cong \overline{\MC}\big(V_1\big) \cong  \{b \in \HHH^1_+(B,B) \mid b \cup b=0\}.
			\end{displaymath}
			because $\HHH^{\bullet}(B,B)$ is graded commutative and where $\HHH^1_+(B,B)$ is canonically viewed as a complete dg algebra equipped with the trivial filtration $0 \subseteq 0 \subseteq \cdots \subseteq  \HHH^1_+(B,B)$. By assumption, there exists a quasi-isomorphism $\phi$ of dg algebras between $C$ and $\cH\coloneqq \HHH^{\bullet}(B,B)$ and, after composition with a graded algebra automorphism of $\cH$, we may assume that $\phi$ induces the identity on the level of cohomologies. Here, we strictly identified $\HH^{\bullet}(\cH)$ and $\cH$. As a consequence of our constructions, $\phi$ restricts to a \textit{filtered quasi-isomorphism} between $U_1$ and $V_1$, which means that it  restricts to quasi-isomorphisms between $U_i$ and $V_i$ for all $i \geq 1$. Now the $A_\infty$-Goldman-Millson theorem \cite[Theorem 1]{MilhamRogers} implies $\overline{\MC}(U_1C)\cong \overline{\MC}(V_1\cH)$ which finishes the proof.
		\end{proof}
\noindent Applied to graded gentle algebras we can now prove the following.
		\begin{thm}\label{thm: kernel positive characteristic} Assume that $\operatorname{char} \Bbbk \neq 2$ and let $A$ be a  rigid graded gentle algebra which is homologically smooth and proper. Then
			\begin{displaymath}
			\ker \Psi^{\infty}_A=\Aut_{\circ}^{\infty}(A) \cong \mathbb{G}_a^{\phi_A(1,1)} \rtimes \HH^1(\Sigma_A, \Bbbk^{\times}). 
			\end{displaymath}
		In particular, Theorem \ref{IntroThmA} holds true in this case.
		\end{thm}
		\begin{proof}
			Because $b \cup b=0$ for all $b \in \HHH^1_+(B,B)$ by graded commutativity, Proposition \ref{PropositionFormalityHochschildComplex} and Proposition \ref{prop:formalityMaurerCartan} above show that there exists a bijection $\pi\colon\HHH^1_+(B,B) \xrightarrow{\sim} \Aut^{\infty}_+(B)$. The explicit description of the quasi-isomorphism between $C(B)$ and $\HHH^{\bullet}(B,B)$ from Proposition \ref{PropositionFormalityHochschildComplex} implies that $\pi$ maps a class $h=\sum_{j}\lambda_j f_{p_j}$, with $f_{q}$ as in the proof of \Cref{PropositionFormalityHochschildComplex}, $\lambda_j \in \Bbbk$, to the $A_\infty$-isotopy $\exp(h)\coloneqq \operatorname{Id}_A + h$. 
 The rest of the assertion now follows as in \Cref{TheoremIntersectionPicardGroupKernel}.
		\end{proof}
		\noindent Again, the assumption that $\operatorname{char} \Bbbk \neq 2$ does not seem essential and was chosen for convenience due to the slightly different nature of Hochschild cohomology in such cases. Since every graded gentle algebra is Morita equivalent to a rigid graded gentle algebra and $\phi_A$ is a Morita invariant, we conclude the following. 
		
			\begin{cor}\label{cor: Theorem A (2)}
			Let $A$ be a homologically smooth and proper graded gentle algebra and assume $\operatorname{char} \Bbbk \neq 2$. Then
			$$\DPic(A) \cong \big(\mathbb{G}_a^{\phi_A(1,1)} \rtimes \HH^1(\Sigma_A, \Bbbk^{\times})\big) \rtimes \MCG_{\gr}(\Sigma_A),$$
			\noindent where $\phi_A$ denotes the Avella-Alaminos-Geiss invariant.
		\end{cor}
		
		\section{Derived Picard groups of wrapped Fukaya categories}\label{SectionDerivedPicardGroupWrappedFukayaCategory}
		\noindent We extend the results in previous sections to wrapped Fukaya categories of punctured surfaces. The proofs follow a similar approach to other cases but require special treatment, also due to the fact that these categories do not admit (known) formal models. Unless mentioned otherwise, through this section $\Sigma$ will refer to a punctured graded marked surface $(\Sigma, \cM, \eta)$, that is, $\partial \Sigma=\cM$ and hence all boundary components are fully marked. We denote by $g=g(\Sigma)$ the genus and by $b=b(\Sigma)$ the number of boundary components of $\Sigma$. As before, we freely exchange fully marked components with punctures. We also fix a full graded arc system $\cA \subseteq \Sigma$ whose cardinality is minimal with this property and denote by $\cF=\cF(\cA)$ the associated model for the wrapped Fukaya category $\Fuk(\Sigma)$ of $\Sigma$.

		\subsection{Geometric model, geometrisation homomorphism and splitting property}\label{SectionGeometricModelWrapped}\ \medskip

		\noindent  There are two natural routes to extend the geometric model of $\Fuk(\Sigma)$ to the punctured case. One uses the identification of the category with a subcategory of the derived wrapped Fukaya in the symplectic sense following the discussion by Abouzaid in the Appendix of \cite{Bocklandt} and Bocklandt's results in loc.cit.~which show that $\Fuk(\Sigma)$ is equivalent to the ``usual'' wrapped Fukaya category in the symplectic sense. In this category, the relationship between intersections, morphisms and so forth are part of the definition but it only follows from \cite{HaidenKatzarkovKontsevich} that all objects are actually represented by curves of the surface.  Alternatively, one could deduce this similar to \cite{OpperKodairaCycles} from the homologically smooth and proper case which treats a particular punctured surface.
			
		\noindent Next, we want to use the geometric model to show a following generalisation of \Cref{TheoremgeometrisationHomomorphism}. Its proof follows a similar route as in the non-punctured case in \cite{OpperDerivedEquivalences} and relies on the relationship between self-diffeomorphisms of a punctured surface and simplicial automorphisms of its arc complex.

		\subsubsection{Arc complexes, simplicial automorphisms and the mapping class group} We recall a few results about arc complexes of punctured surfaces. The reader can find further information on the arc complex of a punctured surface in \cite{IrmakMcCarthy}.
		
		\begin{definition}\label{DefinitionArcComplex}Let $\Sigma$ be any marked surface.
		  The \textbf{arc complex} of $\Sigma$ is the simplicial complex $A(\Sigma)$ whose set of $n$-simplices consists of all arc systems of $\Sigma$ of cardinality $n+1$. 
		\end{definition}
		\noindent The first result which we will need is that the surface can be recovered from the arc complex.
		\begin{thm}[{\cite[Theorem 2.2]{BellDisarloTang}}]\label{thm: arc complex determines surface}
		Let $\Sigma_1, \Sigma_2$ be  marked surfaces and suppose that there exists an isomorphism $A(\Sigma_1) \cong A(\Sigma_2)$. Then $\Sigma_1$ and $\Sigma_2$ are diffeomorphic.
		\end{thm}
		\noindent Predecessors of the theorem (valid in certain subcases) can be found in \cite{IrmakMcCarthy, Disarlo}.  Next, we turn to the symmetries of an arc complex for a punctured surface. The extended mapping class group $\MCG^+(\Sigma)$ acts on $A(\Sigma)$ via simplicial automorphisms and we denote by $\alpha\colon \MCG^+(\Sigma) \rightarrow \Aut(A(\Sigma))$ the associated action homomorphism into the simplicial automorphism group of $A(\Sigma)$.  We say that $\Sigma$ is \textbf{exceptional} if $(g, b) \in \{(0,1), (0,2), (0,3), (1,1)\}$.  We note at this point that the case $(g,b)=(0,1)$ is irrelevant for us as the corresponding Fukaya category is trivial.
		\begin{thm}[{\cite[Theorem 1.1]{IrmakMcCarthy}}]\label{TheoremIrmarkMcCarthy}The homomorphism $\alpha$ is surjective and its kernel is the center of $\MCG^+(\Sigma)$ if $\Sigma$ is exceptional and trivial otherwise. If $\Sigma$ is non-exceptional, then $\alpha$ is an isomorphism.
		\end{thm}
		\noindent The center of $\MCG^+(\Sigma)$ in the exceptional cases is easily determined and can, for example, be found in the proof of \cite[Theorem 2.1]{IrmakMcCarthy}. If $g=0$, $\ker \alpha$ is generated by an orientation reversing involution and if $(g,b)=(1,1)$, then $\MCG^+(\Sigma)\cong \GL_2(\mathbb{Z})$ and $\ker \alpha$ is generated by the hyperelliptic involution.  Geometrically, the latter is a rotation with the four fixed points $(a,b)$, $a,b \in \{\pm1\}$ and corresponds to  $-I \in \GL_2(\mathbb{Z})$, where $I$ denotes the identity matrix.
		\begin{cor}\label{CorollaryExceptionalCases}
		If $(g,b) \in \{(0,2),(0,3)\}$, then the restriction of $\alpha$ to $\MCG(\Sigma)$ is an isomorphism onto $\Aut(A(\Sigma))$. 
		\end{cor}

		\subsubsection{Interior morphisms} Similar to \cite{OpperDerivedEquivalences, OpperKodairaCycles} we first construct a homomorphism $$\Aut(\Fuk(\Sigma)) \rightarrow \Aut(A(\Sigma))$$ whose image lies in the $\alpha(\MCG(\Sigma))$. Then postcomposition with the inverse of $\alpha|_{\MCG(\Sigma)}$,  yields the desired homomorphism $\Psi\colon \Aut(\Fuk(\Sigma)) \rightarrow \MCG(\Sigma)$ which we then promote to a homomorphism to the graded mapping class group. In order to do so, we need to translate the notion of interior intersection to a property of morphisms in $\Fuk(\Sigma)$. 
		
		\begin{definition}\label{DefinitionInteriorMorphisms}
		Let $\cI \subseteq \Fuk(\Sigma)$ denotes the full subcategory whose objects are isomorphic to finite direct sums of objects corresponding to graded loops. Let $\gamma, \delta \subseteq \Sigma$ be graded arcs in minimal position. A morphism $f\colon X_{\gamma} \rightarrow X_{\delta}$ is \textbf{interior} if it factors over an object from $\cI$.
		\end{definition}
\noindent In other words, an object $X \in \Fuk(\Sigma)$ lies in $\cI$ if and only if it is a finite direct sum of objects whose curves are disjoint from $\partial \Sigma$. For $X, Y \in \Fuk(\Sigma)$, we denote by $$\Hom_{\Int}(X, Y) \subset \Hom(X, Y)$$ the subset of interior morphisms and as with morphisms we write $$\Hom_{\Int}^{\bullet}(X,Y)\coloneqq \bigoplus_{n \in \mathbb{Z}}{\Hom_{\Int}(X,Y[n])}.$$ 
		\begin{lem}\label{LemmaBandCategory}
		The category $\cI$ coincides with the full subcategory of $\Fuk(\Sigma)$ of objects with finite-dimensional graded endomorphism ring. In particular, $\cI$ is a hom-finite thick subcategory of $\Fuk(\Sigma)$.
		\end{lem}
	\begin{proof}
	Follows from the relationship between morphisms and intersections, cf.~Section \ref{SectionGeometricModelWrapped}.
	\end{proof}
	
\begin{cor}\label{CorollaryEquivalencesRestrictToBoundary} Every triangle auto-equivalence of $\Fuk(\Sigma)$ restricts to an auto-equivalence of $\cI$ and induces a triangle auto-equivalence on the Verdier quotient $\cQ\coloneqq\Fuk(\Sigma)/\cI$.
	\end{cor}
\begin{cor}\label{CorollaryInteriorMorphismIsIdeal}For all $X, Y \in \Fuk(\Sigma)$, $\Hom_{\Int}(X,Y)$ coincides with the kernel of the canonical map 
	\begin{displaymath}
	\begin{tikzcd}
	\Hom_{\Fuk(\Sigma)}(X,Y) \arrow{r} & \Hom_{\cQ}(X,Y),
	\end{tikzcd}
	\end{displaymath}
	\noindent induced by the $\Fuk(\Sigma) \rightarrow \cQ$. Thus, $\{\Hom_{\Int}(X,Y)\}_{X,Y \in \Fuk(\Sigma)}$ defines a shift closed ideal of the underlying additive category of $\Fuk(\Sigma)$ which is invariant under triangle auto-equivalences of $\Fuk(\Sigma)$, that is, for every auto-equivalence $F$ of $\Fuk(\Sigma)$ and all $X, Y \in \Fuk(\Sigma)$, the canonical isomorphism $\Hom(X,Y) \rightarrow \Hom(F(X), F(Y))$ restricts to an isomorphism $\Hom_{\Int}(X,Y) \rightarrow \Hom_{\Int}(F(X), F(Y))$.
\end{cor}

	\noindent The following is an equivalent characterisation of interior morphisms. 
\begin{prp}\label{PropositionInteriorMorphismsInteriorIntersections}Let $\gamma, \delta$ be graded arcs on $\Sigma$ and $f\colon X_{\gamma} \rightarrow X_{\delta}$ be a morphism in $\Fuk(\Sigma)$. Then $f$ is interior if and only if its support consists of interior intersections. 
\end{prp}
\begin{proof}
	The proof is similar to \cite[Lemma 7.35]{OpperKodairaCycles}. The description of compositions in $\Fuk(\Sigma)$, implies that every basis morphism which factors over $\cI$ must be supported at an interior intersection, cf.~\cite[Lemma 4.19]{OpperKodairaCycles}. Thus, the support of any interior morphism consists of interior intersections. For the converse, let $p$ be an interior intersection between arcs $\gamma$ and $\delta$ and let $f=f_p\colon X_{\gamma} \rightarrow X_{\delta}$ denote the associated morphism. Consider a closed curve $\gamma_{\operatorname{loop}}$ as defined in \cite[Remark 6.6]{OpperKodairaCycles} where in the construction one replaces the clockwise primitive closed loops around the boundary components $A$ and $B$ in \cite[Figure 15]{OpperKodairaCycles} by their $n_A$-th and $n_B$-th powers for $n_A, n_B \in \mathbb{Z}$ such that $n_A \omega_{\eta}(A)=-n_B \omega_{\eta}(B)$ is the smallest common multiple of $|\omega_{\eta}(A)|$ and $|\omega_{\eta}(B)|$. Then one argues as in \cite[Remark 6.6]{OpperKodairaCycles} to show that the resulting loop is primitive and gradable as long as it is not contractible. By construction this would require $A$ and $B$ to be homotopic in the sense that their boundary curves $\gamma_A, \gamma_B\colon S^1 \rightarrow \partial\Sigma \subseteq \Sigma$ are homotopic as maps with target $\Sigma$. However, this happens only if $(g, b)=(0,2)$ and $\gamma$ is the up to homotopy unique arc on $\Sigma$. But $\gamma$ does not have self-intersections outside the boundary in this case, so we were never in this situation to begin with. Finally, it follows from the geometric description of compositions that $f$ factors over $X_{\gamma_{\mathrm{loop}}}$.
\end{proof}

\subsubsection{The geometrisation homomorphism} We are ready to show the existence of the geometrisation homomorphism.
\begin{prp}\label{PropositionIntermediategeometrisationHomomorphism}
There exist a homomorphism $\Psi'\colon\Aut(\Fuk(\Sigma)) \rightarrow \MCG^+(\Sigma)$. In case $(g, b) \in \{(0,2), (0,3)\}$, it has image in $\MCG(\Sigma)$.
\end{prp}
\begin{proof} Any triangle autoequivalence $F$ of $\Fuk(\Sigma)$ preserves indecomposablitity of objects as well as their endomorphism ring. Thus by Lemma \ref{LemmaBandCategory} and the correspondence between indecomposable objects and curves implies that $F$ induces a permutation $\cA(F)$ on the set of arcs on $\Sigma$. By Proposition \ref{PropositionInteriorMorphismsInteriorIntersections}, a collection of arcs $\{\gamma_0, \dots, \gamma_n\} \subseteq \Sigma$ constitutes an $n$-simplex of $A(\Sigma)$ if and only if $\Hom_{\Int}^{\bullet}(X_{\gamma_i}, X_{\gamma_j})=0$ for all $i, j \in \{0, \dots, n\}$. Consequently, $\cA(F)$ restricts to a bijection of the set of vertices of $A(\Sigma)$ and defines a simplicial automorphism. If $\Sigma$ is exceptional but $(g, b)\neq (1,1)$, then Corollary \ref{CorollaryExceptionalCases} implies that $\cA(F)$ corresponds to a unique (orientation preserving) mapping class of $\Sigma$ which we denote by $\Psi(F)$. Similar, if $\Sigma$ is non-exceptional, then by Theorem \ref{TheoremIrmarkMcCarthy} $\cA(F)$ corresponds to a unique extended mapping class $\Psi(F)$. Finally, if $(g,b)=(1,1)$, then $|\cA|=2$. We identify $\Sigma$ and its puncture with $S^1 \times S^1$ and the point $(1,1)$. Since $\MCG(\Sigma)$ acts transitively on the simple closed curves we may assume that $\cA$ contains the simple arc $\gamma$ with image $S^1 \times \{1\}$. Then up to homotopy, the other arc of $\cA$ must be the simple arc $\gamma'$ with image $\{1\} \times S^1$. Thus there are two irreducible flows from $\gamma$ to $\gamma'$ and two irreducible flows from $\gamma'$ to $\gamma$. 
The hyperelliptic involution fixes the arcs $\gamma$ and $\gamma'$ but permutes the irreducible flows. By Proposition \ref{PropositionEquivalencePreserveIrreducibleFlows} below, there therefore exists a \emph{unique} extended mapping class $\Psi(F)$ such that  $\cA(F)$ is the simplicial automorphism induced by $\Psi(F)$ and such that $\Psi(F)$ maps each irreducible flow to its image under $F$. 
\end{proof}
\noindent Next, we show that the image of the homomorphism from the previous proposition always has image in the mapping class group. Because of the slightly different different definition of the Fukaya category in this case, we will exclude the case that $\Sigma$ is an annulus with two fully marked components from some of the discussion below. As a preparation we analyse the structure of endomorphism algebras and the corresponding module structure on the morphism spaces in $\Fuk(\Sigma)$.

	\begin{lem}\label{LemmaEndomorphismsAndModuleStructure}Assume $(g,b)\neq (0,2)$.
	Let $\gamma \subseteq \Sigma$ be a simple graded arc with endpoint in boundary components $B_0 \ni \gamma(0)$ and $B_1 \ni \gamma(1)$ and let $E_{\gamma}\coloneqq\End^{\bullet}(X_{\gamma})$ denotes its graded endomorphism algebra in $\Fuk(\Sigma)$. Then the following hold.
	
	\begin{enumerate}
		\item There exist an isomorphism of graded $\Bbbk$-algebras 
			\begin{displaymath}
			E_{\gamma} \cong \begin{cases}\Bbbk\langle x,y\rangle /(x^2, y^2), & \text{if $B_0=B_1$}; \\ \Bbbk[x,y]/(xy), & \text{if $B_0 \neq B_1$}.\end{cases}
			\end{displaymath}
			\noindent with $|x|=\omega_{\eta}(p)$ and $|y|=\omega_{\eta}(q)$. The isomorphism identifies the irreducible flow which starts at $\gamma(0)$ with $x$ and the irreducible flow which starts at $\gamma(1)$ with $y$. The group of algebra automorphisms of $E_{\gamma}$ is generated by the involution which exchanges $x$ and $y$ and the automorphisms of the form $(x,y) \mapsto (\alpha x, \beta y)$ for $\alpha, \beta \in \Bbbk^{\times}$. 
			
		\item \label{ItemModuleStructure} For any simple graded arc $\delta \subseteq \Sigma$ such that $\{\gamma, \delta\}$ is an arc system and with endpoints in boundary components $C_0 \ni \delta(0) $ and $C_1 \ni \delta(1)$, let $H(\delta, \gamma)\coloneqq \Hom^{\bullet}_{\Fuk(\Sigma)}(X_{\delta}, X_{\gamma})$. Then there exists an isomorphism of graded left $E_{\gamma}$-modules
		\begin{displaymath}
		 H(\delta, \gamma) \cong \begin{cases}(x,y), & \text{if $\{B_0,B_1\}=\{C_0,C_1\}$}; \\ 
		(x,y)/(z_y),  & \text{if $\{B_0,B_1\}\neq\{C_0,C_1\}$ and $B_0 \in \{C_0,C_1\}$}; \\
			(x,y)/(z_x),  & \text{if $\{B_0,B_1\}\neq\{C_0,C_1\}$ and $B_1 \in \{C_0,C_1\}$}; \\
			0 & \text{otherwise,}\end{cases}
		\end{displaymath}
	\noindent where $z_x=x$ and $z_y=y$ if $B_0 \neq B_1$ and $z_x=yx$ and $z_y=xy$ if $B_0=B_1$. The isomorphism identifies the irreducible flow with end point $\gamma(0)$ with $x$ and the irreducible flow with end point $\gamma(1)$ with $y$. Every  $E_{\gamma}$-module automorphism of $H(\gamma, \delta)$ is of the form $x \mapsto \lambda x$ and $y \mapsto \mu y$, $\lambda, \mu \in \Bbbk^{\times}$.
	
	\item There are analogous results for the right $E_{\delta}$-module structure of $H(\delta, \gamma)$.
	\end{enumerate}  

	\end{lem}
\begin{proof}
Because $\gamma$ and $\delta$ are simple, the description of the ring $E_{\gamma}$ and the $E_{\gamma}$-module $H(\delta, \gamma)$ is a direct consequence of the correspondence between basis morphisms and flows between $\gamma$ and $\delta$ and their compositions. For the description of algebra automorphisms of $E_{\gamma}$, we distinguish between two cases.

 First, if $B_0 \neq B_1$, there exist injective algebra homomorphisms $\iota\colon\Bbbk[x] \hookrightarrow E_{\gamma}$ and $\kappa\colon\Bbbk[y] \hookrightarrow E_{\gamma}$  where $\iota(x)=x$ and $\kappa(y)=y$. Their images intersect in the subalgebra generated by the identity and jointly span $E_{\gamma}$ as a vector space. In particular, every automorphism of $E_{\gamma}$ is uniquely determined by its restrictions $F_{\iota}\coloneqq F \circ \iota$ and $F_{\kappa}\coloneqq F \circ \kappa$. Now if $F\colon E_{\gamma} \rightarrow E_{\gamma}$ is an automorphism, we may write $F_{\iota}(x)=\lambda  + P_x + P_y$ for polynomials $P_{\star} \in \Bbbk[\star]$ and a scalar $\lambda \in \Bbbk$. Similarly, we can decompose $F_{\kappa}(y)$ as a sum $\mu + Q_{x} + Q_{y}$. Because $F$ is an automorphism and because $xy=0$, it follows $P_x Q_y=0=P_y Q_x$. Hence $0=\lambda \mu + P_xQ_x + P_yQ_y$ which implies $\lambda \mu=0$ as well as $P_x Q_x=0$ and $P_y Q_y=0$. Moreover, if $P_x=0$ (resp.~$Q_x=0$), then $P_y \neq 0$ (resp.~$Q_y \neq 0$) as otherwise the image of $F_{\iota}$ would be finite-dimensional in contradiction to $F$ being an isomorphism. Similarly, if $P_x=0$, then $Q_x \neq 0$ since $F_{\iota}$ and $F_{\kappa}$ are jointly surjective. Thus upon exchanging $x$ and $y$ via the obvious involution, we may assume $P_y=0$, $Q_x=0$. But this means that $F$ restricts to an isomorphism of $\Bbbk[x]$ and $\Bbbk[y]$ (considered as subalgebras) and are hence of the form $F_{\iota}(x)=\lambda + \alpha x$ and $F_{\kappa}(y)=\mu + \beta y$, where $\alpha, \beta \in \Bbbk^{\times}$. Again, the relation $xy=0$ then implies $0=F(x)F(y)=F_{\iota}(x) F_{\kappa}(y)$ and hence $\lambda=\mu=0$. Hence $F$ is of the form $F(x)=\alpha x$ and $F(y)=\beta y$, $\alpha, \beta \in \Bbbk^{\times}$.

Next, in the case $B_0 = B_1$, we observe that $E_{\gamma}$ is isomorphic to the category algebra of the graded free $\Bbbk$-linear category $\cU$ generated by two objects $U_0, U_1$ and a single morphism $f_i\colon U_i \rightarrow U_{1-i}$ for each $i \in \{0,1\}$. It follows that automorphisms of $E_{\gamma}$ are in bijection with automorphisms of $\cU$. It is easy to see that all such automorphisms are generated by the involution which exchanges $U_1$ and $U_2$ and the automorphisms which act trivially on the set of objects but multiply $f_i$ by a non-zero scalar. Finally, the assertion about automorphisms of the $E_{\gamma}$-module $H(\delta, \gamma)$ follows from very similar arguments.
\end{proof}
\noindent The previous lemma is now used to show the following.
\begin{prp}\label{PropositionEquivalencePreserveIrreducibleFlows}
Let $\delta, \gamma$ be non-homotopic graded simple arcs, $F$ a triangle auto-equivalence of $\Fuk(\Sigma)$ and $f\colon  X_{\delta} \rightarrow X_{\gamma}$ be a morphism corresponding to an irreducible flow $\alpha$. Then $F(f)$ corresponds to a scalar multiple of the morphism associated to the irreducible flow $F \circ \alpha$.
\end{prp}
\begin{proof}
Irreducible flows correspond to the generators $x$ and $y$ of $H(\delta, \gamma)$ and $H(F(\delta), F(\gamma))$ respectively. Since $F$ induces isomorphisms of $E_{\gamma}\cong E_{F(\gamma)}$ and $E_{\gamma}$-isomorphisms between $H(\delta,\gamma)$ and $H(F(\delta), F(\gamma))$, it follows from Lemma \ref{LemmaEndomorphismsAndModuleStructure} (\ref{ItemModuleStructure}) and the description of automorphisms of $E_{\gamma}$ that $F$ must map a generator of $H(\delta, \gamma)$ to the scalar multiple of the generator of $H(F(\delta), F(\gamma))$.
\end{proof}
\noindent Suppose now that we are given a finite set of disjoint simple arcs $\{\gamma_1, \dots, \gamma_m\}$, where $\gamma_i\colon[0,1] \rightarrow \Sigma$ and where we now consider $\Sigma$ as a surface with fully marked boundary. Then on every component $B \subseteq \partial \Sigma$ the set of end points of the arcs divide $B$ into a set intervals $\cI=\cI_B$ each of which corresponds to an irreducible flow between them and each of which follows the induced orientation on $B$. The set $\cI$ inherits a cyclic order from $B$ where $g \in \cI$ is the successor of $f \in \cI$ if and only if the end of $f$ coincides with the start of $g$. We also define a label function $l\colon  \cI \rightarrow \{1,\dots, m\}$ which assigns to an irreducible flow $f$ the unique index $i$ such that $f(1) \in \gamma_i$. By assigning to element of $\cI$ its label, we obtain a cyclic sequence $\mathbbm{m}_B$ with values in the set $\{1, \dots, m\}$.   Via the correspondence with generators in the $E_{\gamma_j}$-modules $H(\gamma_i, \gamma_j)$, the cyclic order and the label function are encoded categorically as follows: a generator corresponding to an irreducible flow $g$ is the successor of the generator of an irreducible flow $f$ if and only if $g \circ f \neq 0$. The label of a morphism is simply the index $i$ of its codomain $\gamma_i$.  Here we note again that the choice of generators for $H(\gamma_i, \gamma_j)$ as an $E_{\gamma_j}$-module is unique up to scalar multiples which implies that the cyclic sequence $\mathbbm{m}_B$ is well-defined. 

 If now $F$ is a triangle auto-equivalence, we can now compare $\cI$ and $l$ to the cyclically ordered set $\cI_{F(B)}^F$ and its label function $l^F\colon  \cI_{F(B)}^F \rightarrow \{1,\dots, m\}$ obtained from the arcs $F(\gamma_1), \dots, F(\gamma_m)$ and the boundary component $F(B)$. We denote by $\mathbbm{m}_{F(B)}^F$ the associated cyclic sequence with values in $\{1,\dots, m\}$. Because every triangle auto-equivalence of $\Fuk(\Sigma)$ covariantly commutes with compositions and maps generators of $H(\gamma_i, \gamma_j)$  to generators of $H(F(\gamma_i), F(\gamma_j))$, we conclude the following.
\begin{cor}\label{CorollaryPreserveCyclicSequences}
Let $\cI_B$ and $\cI_{F(B)}^F$ be as defined above. Then every triangle auto-equivalence of $\Fuk(\Sigma)$ induces an isomorphism of cyclic ordered sets $\varphi\colon \cI_B \rightarrow \cI_{F(B)}^F$ such that $l(i)=l^F(\varphi(i))$ for all $i \in \cI_B$.
\end{cor}
\noindent We are finally prepared to prove the following.
\begin{prp}\label{prp: orientation preserving}
For all $\Sigma$, the homomorphism $\Psi'$ from Proposition \ref{PropositionIntermediategeometrisationHomomorphism} has image in $\MCG(\Sigma)$.
\end{prp}
\begin{proof}Let $\gamma_1, \dots, \gamma_m$ be a sequence of disjoint simple arcs and $B \subseteq \partial \Sigma$ be a boundary component. Let further $F$ be a triangle auto-equivalence of $\Fuk(\Sigma)$ and let $\mathbbm{m}_B$ and $\mathbbm{m}_{F(B)}^F$ denote the associated $\{1,\dots, m\}$-valued cyclic sequences. It is clear that if $\Psi(F)$ was orientation reversing then the cyclically ordered set $\cI_{F(B)}^F$ from Corollary \ref{CorollaryPreserveCyclicSequences} would be abstractly isomorphic as a cyclically ordered set to $\cI_B$ when equipped with the \emph{opposite} of its cyclic order as defined previously. Moreover, the isomorphism would be compatible with the label functions which implies by Corollary \ref{CorollaryPreserveCyclicSequences} that the cyclic sequence $\mathbbm{m}_B=i_1 \dots i_r$ would coincide with its inverse sequence $i_r \dots i_1$ up to rotation. Thus, in order to prove that $\Psi(F)$ is orientation preserving it is now sufficient to find a suitable configuration of arcs as above and a boundary component $B$ for which this is not the case. Since $\Psi(F)$ is orientation preserving by construction if $(g,b) \in \{(0,2),(0,3)\}$, we assume that $(g,b)=(1,1)$ or that $\Sigma$ is non-exceptional. An arc system $\gamma_1, \gamma_2, \gamma_3$ can be chosen in a way that the associated cyclic sequence is 
	
\begin{displaymath}
\mathbbm{m}_B= \begin{cases}(1,2,3),  & \text{if $g=0, b\geq 4$;} \\
(1,3,2,1,3,2), & \text{if $g \geq 1$}.\end{cases}
\end{displaymath}
The existence of such arc collections is clear if $g=0$. If $g=1$ we may identify $\Sigma$ with $S^1 \times S^1$ with a puncture at $(1,1) \in S^1 \times S^1$ so that its other punctures are disjoint from the set $\{(s,s) \mid s \in S^1\}\cup \{(s,t) \mid s=1 \text{ or } t=1\}$. One can choose $\gamma_1$ and $\gamma_2$ as the longitudinal and latitudinal arcs obtained from the homeomorphism of $S^1$ with the subsets  $S^1 \times \{1\}, \{1\} \times S^1 \subseteq S^1 \times S^1$  
and for $\gamma_3$ the diagonal inclusion $S^1 \rightarrow S^1 \times S^1$. If $g \geq 2$, one chooses the arcs in the analogous way.
\end{proof}
\begin{prp}
The homomorphism $\Psi'\colon  \Aut(\Fuk(\Sigma)) \rightarrow \MCG(\Sigma)$ admits a unique lift to a homomorphism $\Psi\colon  \Aut(\Fuk(\Sigma)) \rightarrow \MCG_{\gr}(\Sigma)$ along the projection $\MCG_{\gr}(\Sigma) \rightarrow \MCG(\Sigma)$ with the property that 
\begin{displaymath}
X_{\Psi(F)(\gamma)} \cong X_{F(X_\gamma)}
\end{displaymath}
\noindent for all $F \in \Aut(\Fuk(\Sigma))$ and all $\gamma \in \cA$.
\end{prp}
\begin{proof}Analogous to the proof of \Cref{PropositionGradedLiftgeometrisation}.
\end{proof}
\begin{prp}\label{PropositionAlexanderMethodPunctured}
Let $F \in \Aut(\Fuk(\Sigma))$. Then $F \in \ker \Psi$ if and only if $F(X_{\gamma})\cong X_{\gamma}$ for all $\gamma \in \cA$.
\end{prp} 
\begin{proof}
Analogous to the proof of \Cref{LemmaRigidGentleKernel} and follows from an application of the Alexander method to $\cA$.
\end{proof}
As the final result of this section, we show that the graded marked surface is an invariant of the wrapped Fukaya category.
\begin{thm}\label{thm: derived invariant Fukaya categories}
Let $\Sigma_1, \Sigma_2$ be graded marked surfaces. If $\Fuk(\Sigma_1) \simeq \Fuk(\Sigma_2)$, then $\Sigma_1 \cong \Sigma_2$ as graded marked surfaces. Moreover, if $A, B$ are graded gentle algebras which are homologically smooth or proper and $\Fuk(\Sigma_1) \simeq \cT_A$, then $\Sigma_1 \cong \Sigma_A$ as graded marked surfaces. Finally, if $\cT_A \simeq \cT_B$, then $\Sigma_A \cong \Sigma_B$ as graded marked surfaces.
\end{thm}
\begin{proof}We may and will assume that $A$ or $B$ (say $A$) are proper as otherwise the cases involving the two are already covered by the comparison between Fukaya categories. Let us assume that $\Fuk(\Sigma_1)\simeq \Fuk(\Sigma_2)$. Then, if neither of $\Sigma_1$ and $\Sigma_2$ is a punctured surface, this follows form \Cref{TheoremGeometrisationHomomorphismSmoothCase}.
	
Next suppose that $\Sigma_1$ is a punctured surface but $\Sigma_2$ is not. Then $\Fuk(\Sigma_2) \simeq \Fuk(\Sigma_1)$ is not proper and hence must have a fully marked boundary component but also a marked interval because it is not a punctured surface. Hence $\Sigma_2$ contain a semi-infinite simple arc and therefore there exists $X \in \Fuk(\Sigma_2)$ such that $\End^{\bullet}(X)\cong \Bbbk[x]$, $|x| \in \mathbb{Z}$. But by the same arguments as in the proof of \Cref{LemmaEndomorphismsAndModuleStructure} (1), the graded endomorphism algebra of every object in $\Fuk(\Sigma_1)$ contains zero-divisors whereas  $\Bbbk[x]$ does not. Thus $\Fuk(\Sigma_1)$ cannot be equivalent to $\Fuk(\Sigma_2)$.

As our next case, we assume that $\Sigma_1$ and $\Sigma_2$ are punctured surfaces. Then by the same arguments as in \Cref{PropositionIntermediategeometrisationHomomorphism}, one constructs an isomorphism $A(\Sigma_1) \simeq A(\Sigma_2)$ which implies that $\Sigma_1$ and $\Sigma_2$ are diffeomorphic by \Cref{thm: arc complex determines surface}. Then one proceeds as in  \Cref{prp: orientation preserving} to show that the induced diffeomorphism preserves the orientation. Finally, one can argue\footnote{The proof of \cite[Proposition 4.21]{OpperDerivedEquivalences} relies on the notion of an interior morphism. By definition, in loc.cit., a boundary morphism is a morphism which is not interior. With the help of \Cref{CorollaryEquivalencesRestrictToBoundary}, one can repeat the proof using \Cref{DefinitionInteriorMorphisms} but one can ignore the assumptions that the arc is ``finite'' or ``not $\tau$-invariant.}  as in \cite[Proposition 4.21]{OpperDerivedEquivalences} to show that the homotopy classes of the line fields of $\Sigma_1$ and $\Sigma_2$ correspond to each other under the diffeomorphism.  Thus, $\Sigma_1 \cong \Sigma_2$ as graded marked surfaces. 

Next, if $\Fuk(\Sigma) \simeq \cT_A$, we may assume that $A$ is not homologically smooth as otherwise, again, this case would be already covered. Hence $A$ and $\Fuk(\Sigma)$ are proper and the result follows from \Cref{TheoremgeometrisationHomomorphism}.

Finally if $\cT_A \simeq \cT_B$, then $B$ must be proper too because $\cT_A$ is proper, and hence the result follows from \Cref{TheoremgeometrisationHomomorphism}.
\end{proof}

\subsubsection{Splitting property} We show that the mapping class group action on $\Fuk(\Sigma)$ provides a section to the geometrisation homomorphism in the punctured case.
		
		\begin{prp}\label{PropositionSplitPunctured}
		The group action of $\MCG_{\gr}(\Sigma)$ on $\Fuk(\Sigma)$ constitutes a section to the geometrisation homomorphism $\Psi$.
		\end{prp}
		\begin{proof}Set $\cT\coloneqq \Fuk(\Sigma)$ and let $\overline{\Sigma}$ denote the graded marked surface obtained from $\Sigma$ by removing a set homeomorphic to $[0,1] \sqcup [0,1]$ from each connected component of $\cM$. Then $\overline{\cT}\coloneqq \Fuk(\overline{\Sigma})$ is homologically smooth and proper and by \cite[Proposition 3.5]{HaidenKatzarkovKontsevich}, there exists a localisation functor $\overline{\cT} \rightarrow \cT$ whose kernel is generated by the boundary segments of $\overline{\Sigma}$. The boundary segments on each component of $B \subseteq \partial \overline{\Sigma}$ correspond to an exceptional cycle $E_B=E_B^1 \oplus E_B^2$ in the sense of \cite[Definition 3.2]{BroomheadPauksztelloPloog}. As such it has an associated twist functor $T_B\colon  \overline{\cT} \rightarrow \overline{\cT} \in \DPic(\overline{\cT})$ defined as the mapping cone of the evaluation map $\operatorname{ev}\colon  \overline{\cT} \rightarrow \overline{\cT}$ given by $\operatorname{ev}_X\colon \RHom_{\overline{\cT}}(E_B, X) \otimes^{\mathbb{L}}_{\Bbbk \times \Bbbk} E_B \rightarrow X$. The functor $T_B$ is a quasi-equivalence and, by construction, comes  with an associated natural transformation $\sigma^B\colon \operatorname{Id}_{\overline{\cT}} \rightarrow T_B$. Let $\bA$ denote the Drinfeld quotient of $\overline{\cT}$ along the collection of mapping cones of all the maps $\{\sigma_X^B\colon  X \rightarrow T_B(X) \mid X \in \overline{\cT}, B \subseteq \partial \overline{\Sigma}\}$. By construction the thick subcategory $\cB$ generated by these cones lies in the full subcategory $\cE$ generated by the object $E_B$. On the other hand, the cone of $\sigma_{E_B}$ a shift of $E_B$ which implies $\cE=\cB$. We also note that the different twist functors commute pairwise since the $E_B$ and $E_{B'}$ are orthogonal whenever $B$ and $B'$ are distinct components of $\partial \overline{\Sigma}$. This shows that $\bA$ and $\cT$ are quasi-equivalent. Moreover, a straightforward calculation as in \cite[Lemma 7.8]{OpperKodairaCycles} shows that $\Psi^{\infty}(T_B) \in \MCG(\overline{\Sigma}, \eta)$ is the fractional twist $R_B$ around $B$, that is, the mapping class of a diffeomorphism which acts as the identity outside a collar neighbourhood of $B$ and which rotates $B$ clockwise by an angle less than $2\pi$ so that it permutes the two boundary segments of $B$. Now if $F \in \DPic(\overline{\cT})$, then $\Psi^{\infty}(F)$ maps boundary segments to boundary segments and because $\overline{\Sigma}$ has no degenerating boundary components, it follows that $F(E_B)$ is isomorphic to a shift of $E_{\Psi^{\infty}(F)(B)}$. It follows that $F$ restricts to an autoequivalence of the kernel of the functor $\overline{\cT} \rightarrow \cT$. Thus, by the universal property of the Drinfeld quotient in $\HAcat$ \cite{TabuadaDrinfeldQuotient}, there exists a restriction homomorphism 
			\begin{displaymath}
				\operatorname{res}\colon \DPic(\overline{\cT}) \rightarrow \DPic(\cT).
			\end{displaymath} 
			\noindent  Each natural transformation $\sigma^B\colon  \operatorname{Id}_{\overline{F}} \rightarrow T_B$ descends to an isomorphism on $\cT$ which implies $T_B \in \ker \operatorname{res}$.  
We obtain the following commutative diagram of groups and group homomorphisms 
		
			\begin{displaymath}
			\begin{tikzcd}
			 \mathbf{1} \arrow{r} &\prod_{B \subseteq \overline{\Sigma}}{\langle R_B\rangle}  \arrow{r} \arrow{d}   & \MCG_{\gr}(\overline{\Sigma}) \arrow{r} \arrow{d}{\overline{\alpha}} & \MCG_{\gr}(\Sigma) \arrow{r} \arrow[dashed]{d}{\exists ! \, \alpha} & \mathbf{1}\\
			  & \prod_{B \subseteq \overline{\Sigma}}{\langle T_B \circ \varphi_B \rangle} \arrow{r} \arrow{d} & \DPic(\overline{\cT}) \arrow{d}{\Psi^{\infty}_{\overline{\cT}}}\arrow{r}{\operatorname{res}} & \DPic(\cT) \arrow{d}{\Psi^{\infty}_{\cT}} \\
			  \mathbf{1} \arrow{r} &\prod_{B \subseteq \overline{\Sigma}}{\langle R_B\rangle}  \arrow{r}  & \MCG_{\gr}(\overline{\Sigma}) \arrow{r}  & \MCG_{\gr}(\Sigma) \arrow{r} & \mathbf{1}.
			\end{tikzcd}
			\end{displaymath}			
			\noindent Its first and last row are exact. Here, $\overline{\alpha}$ denotes the action homomorphism of the mapping class group action and $\varphi_B \in \ker \Psi^{\infty}_{\overline{\cT}}$  the element  determined by the equality   $\overline{\alpha}(R_B)  \circ T_B^{-1}\cong \varphi_B$. In particular, $\overline{\alpha}$ is a section to $\Psi^{\infty}_{\overline{\cT}}$ and by  \Cref{thm: kernel positive characteristic}, we have $\ker \Psi^{\infty}_{\overline{\cT}}\cong \HH^1(\overline{\Sigma}, \Bbbk^{\times})$.  From the commutativity and exactness of the first row we now conclude that $\operatorname{res} \circ \overline{\alpha}$ induces a group homomorphism $\alpha\colon  \MCG_{\gr}(\Sigma) \rightarrow \DPic(\cT)$ which is a section to $\Psi^{\infty}_{\cT}$. That $\alpha$ in fact agrees with the action homomorphism of the $\MCG_{\gr}(\Sigma)$-action follows from the following observation: the precomposition of the functor $\Fuk_{\Sigma, \eta}$ from \eqref{EquationFukayaFunctor} with the natural projection $\Arc{\overline{\Sigma}, \eta} \rightarrow \Arc{\Sigma, \eta}$ that identifies an arc on $\overline{\Sigma}$ with its homotopy class on $\Sigma$ is naturally isomorphic to the functor with assigns to an arc system $\cA \in \Arc{\overline{\Sigma}, \eta}$ the Verdier quotient of $\Fuk_{\overline{\Sigma}, \eta}(\cA)$ by the boundary segments. 
		\end{proof}
\noindent	Combining all our previous results we have proved the following generalisation of \Cref{TheoremgeometrisationHomomorphism}.
	
		\begin{thm}\label{TheoremgeometrisationWrappedFukaya}
		Let $\Sigma$ be a graded punctured surface. Then there exists a split surjective group homomorphism
		\begin{displaymath}
		\begin{tikzcd}
		\Psi^{\infty}\colon \DPic(\Fuk(\Sigma)) \arrow{r} & \MCG_{\gr}(\Sigma),
		\end{tikzcd}
		\end{displaymath}
		\noindent which factors through a split surjective homomorphism $\Psi\colon	\Aut(\Fuk(\Sigma)) \rightarrow  \MCG_{\gr}(\Sigma)$. A section to $\Psi$ and $\Psi^{\infty}$ is given by the action homomorphism of the action of $\MCG_{\gr}(\Sigma)$ on $\Fuk(\Sigma)$ from Section \ref{SectionSplit}.
	\end{thm}

		\subsection{The kernel of the geometrisation homomorphism} \ \medskip

		\noindent By \ref{PropositionAlexanderMethodPunctured}, it follows that $G \in \ker \Psi^{\infty}$ if and only if $G(X)\cong X$ for all $X \in \Ob{\cF}$. Hence by \Cref{LemmaImageOfInclusion}, $G \in \ker \Psi^{\infty}$ is weakly equivalent to the canonical extension to perfect complexes of an $A_\infty$-functor $G\colon  \cF \rightarrow \cF$ with $G^0=\operatorname{Id}_{\Ob{\cF}}$ and we may and will henceforth view $\ker \Psi^{\infty}$ as a subset of $\Aut^{\infty}(\cF)$ consisting of functors which act trivially on objects. 
		\begin{lem}\label{LemmaAutomorphismsAreRescalingAuotmorphisms}
		Suppose $F=G^1$ for some $G \in \ker \Psi^{\infty}$. If $(g,b)\neq(0,2)$,  then $F$ is a rescaling automorphism. In particular, $F$ defines a strict $A_\infty$-automorphism of $\cF$ and $F \in \ker \Psi^{\infty}$. If $(g,b)=(0,2)$, then $\Fuk(\Sigma)$ has a single object up to shift and $F$ corresponds to an automorphism of $\Bbbk[x, x^{-1}]$.
		\end{lem}
		\begin{proof} First, let us assume $(g,b)=(0,2)$. The only arc $\gamma$ on $\Sigma$ is the one which connects the two boundary components. Thus the strict morphism $F$ is simply an automorphism of the endomorphism ring of $X_{\gamma}$ which is $\Bbbk[x,x^{-1}]$ for some degree of $x$. For the remainder of the proof, let us assume that  $(g,b) \neq (0,2)$. Then any full arc collection $\cA \subseteq \Sigma$ contains at least two distinct arcs, say $\gamma_1$ and $\gamma_2$ corresponding to objects $X_1, X_2 \in \Ob{\cF}$. Furthermore because $\Sigma$ is connected we may also assume that there exists an irreducible flow $f$ from $\gamma_1$ to $\gamma_2$. The mapping cone $C_f$ of the corresponding morphism $\alpha\colon X_{1} \rightarrow X_2$ is represented by the \emph{simple} arc $\gamma_f$ obtained by concatenating $\gamma_1$ and $\gamma_2$ via $f$. Consequently, $G \in \ker \Psi^{\infty}$ preserves the isomorphism class of $C_f$. Because the homotopy class of $\gamma_f$  determines $f$ uniquely among the irreducible flows, $F$ must map $\alpha$ to a non-zero multiple of itself. $F$ further induces automorphisms $\varphi_i$ of $E_{\gamma_i}\coloneqq\End^{\bullet}(X_{\gamma_i})$ and by Lemma \ref{LemmaEndomorphismsAndModuleStructure}, each is a composition of an involution which exchanges the respective generators $x,y \in E_{\gamma_i}$ and scalar multiplications of the generators with non-zero elements of $\Bbbk$. Suppose now that $\varphi_i$, $i \in \{1,2\}$ maps the ordered pair $(x,y)$ to the ordered pair $(\mu y, \lambda x)$ for some $\lambda, \mu \in \Bbbk^{\times}$. But because $\cF$ is minimal, $F$ commutes with the composition operation $\mu_2$ and hence this yields a contradiction to Lemma \ref{LemmaEndomorphismsAndModuleStructure} \ref{ItemModuleStructure} which says that that every $E_{\gamma_i}$-module automorphism of $H(\gamma_1, \gamma_2)$ simply multiplies each generator by a non-zero scalar. It follows that $F$ must be of the claimed form. Because $F$ acts trivially on objects it lies in $\ker \Psi^{\infty}$.
		 	\end{proof}
	 	\noindent From the previous lemma it follows that the projection $\ker \Psi^{\infty} \rightarrow \Out(\cF)$, $G \mapsto G^1$ has image in $\cR$ if $(g,b)\neq (0,2)$. Otherwise $G^1$ is an automorphism of $\Bbbk[x,x^{-1}]$, where $|x|$ is given by the winding number of any of the two boundary components of $\Sigma$. By regarding $G^1$ as a strict $A_\infty$-automorphism of $\cF$, this yields the following result. 
		 \begin{cor}If $(g,b)\neq (0,2)$ or $\Sigma$ has a boundary component with non-vanishing winding number, then 
		 $\ker \Psi^{\infty} \cong \cR \rtimes \Aut^{\infty}_+(\cF) \cong \HH^1(\Sigma, \Bbbk^{\times}) \rtimes  \Aut^{\infty}_+(\cF)$. Otherwise, $\ker \Psi^{\infty} \cong (\Bbbk^{\times} \rtimes \mathbb{Z}_2) \rtimes \Aut^{\infty}_+(\cF) \cong (\HH^1(\Sigma, \Bbbk^{\times}) \rtimes \mathbb{Z}_2) \rtimes  \Aut^{\infty}_+(\cF)$.
		 \end{cor}
	 \noindent As before and over a field of characteristic $0$,  the previous corollary reduces the study of $\ker \Psi^{\infty}$ to the structure of the Lie algebra $\HHH^1_+(\cF,\cF)$ as well as the kernel of the exponential map. The entire Hochschild cohomology of $\cF$ was computed by Bocklandt and van de Kreeke in \cite{BocklandtVanDeKreeke}.
	 \begin{prp}
	 The subspace $\HHH^1_+(\cF,\cF)$ is trivial. Thus, if $\operatorname{char} \Bbbk=0$, then $\Aut_{+}^{\infty}(\cF)$ is trivial.
	 \end{prp}
	 \begin{proof}
	 As shown in \cite[Section 3.2]{BocklandtVanDeKreeke}, the cohomology class of a Hochschild cocycle $c \in C(\cF)$ is determined by its components $c^0$ and $c^1$. In particular, $c$ is a coboundary if $c^0$ and $c^1$ are trivial which readily implies that every class in $\HHH_+^{1}(\cF,\cF)$ is trivial by definition. But over a field of characteristic $0$, the group $\Aut^{\infty}_+(\cF)$ is exactly the image of the integration map $\exp_{\cF}$.
	 \end{proof}
\noindent As a corollary we obtain the punctured case of \Cref{IntroThmA}.
		\begin{thm}\label{TheoremDerivedPicardGroupWrappedCase}
		Let $\Sigma$ be a graded punctured surface of genus $g$ and with $b$ punctures. Suppose $\operatorname{char} \Bbbk=0$. Then, there exists an isomorphism of groups
		\begin{displaymath}
			\DPic\big(\Fuk(\Sigma)\big) \cong Z \rtimes \MCG_{\gr}(\Sigma),
		\end{displaymath}
		\noindent where $Z=\Bbbk^{\times} \rtimes \mathbb{Z}_2$ if $(g,b)=(0,2)$ and the winding numbers of the boundary components of $\Sigma$ vanish. Otherwise, $Z=\HH^1(\Sigma, \Bbbk^{\times})$.
		\end{thm}
		
		\begin{rem}
		 \Cref{TheoremDerivedPicardGroupWrappedCase} is in line with the results of \cite[Corollary 8.41, Remark 8.42]{OpperKodairaCycles} which describe the autoequivalence groups of the Kodaira $n$-cycle $C_n$ of projective lines, $n \geq 1$. The corresponding surface is the torus $\cT^n$ with $n$ punctures. One has $\HH^1(\cT^n, \Bbbk^{\times}) \cong (\Bbbk^{\times})^{n+1}$, $\MCG(\cT^1)\cong \operatorname{SL}_2(\mathbb{Z})$ and $\Pic^{\circ}(C_n)\cong \Bbbk^{\times}$ as well as $\MCG_{\gr}(\Sigma)\cong \widetilde{\operatorname{SL}}_2(\mathbb{Z})$ in the notation of \cite{SeidelThomas}. Thus, \Cref{TheoremDerivedPicardGroupWrappedCase} in particular recovers \cite[Theorem D]{OpperKodairaCycles}.
		\end{rem}
		
		\bibliography{Bibliography}{}
		\bibliographystyle{alpha}
	\end{document}